\documentclass{amsart}

\usepackage{amscd,amssymb}

\usepackage{url}

\usepackage[all,2cell,cmtip]{xy} \UseAllTwocells \SilentMatrices

\usepackage{ytableau}

\usepackage[colorlinks=true]{hyperref}

\newcommand{\I}{\mathbf 1}

\newcommand{\Q}{\mathbf Q}

\newcommand{\Z}{\mathbf Z}

\newcommand{\bA}{\mathbf A}

\newcommand{\bm}{\mathbf m}
\newcommand{\bn}{\mathbf n}
\newcommand{\br}{\mathbf r}

\newcommand{\GL}{\mathrm {GL}}

\newcommand{\rM}{\mathrm M}

\newcommand{\sA}{\mathcal A}

\newcommand{\sC}{\mathcal C}
\newcommand{\sD}{\mathcal D}

\newcommand{\sF}{\mathcal F}

\newcommand{\sH}{\mathcal H}
\newcommand{\sI}{\mathcal I}
\newcommand{\sJ}{\mathcal J}

\newcommand{\sM}{\mathcal M}
\newcommand{\sN}{\mathcal N}
\newcommand{\sO}{\mathcal O}

\newcommand{\sR}{\mathcal R}
\newcommand{\sS}{\mathcal S}

\newcommand{\sU}{\mathcal U}
\newcommand{\sV}{\mathcal V}
\newcommand{\sW}{\mathcal W}

\newcommand{\iso}{\xrightarrow{\sim}}

\newcommand{\nd}{\nobreakdash-\hspace{0pt}}

\renewcommand{\theenumi}{(\roman{enumi})}

\DeclareMathOperator{\Coker}{Coker}
\DeclareMathOperator*{\colim}{colim}
\DeclareMathOperator{\End}{End}

\DeclareMathOperator{\Gal}{Gal}

\DeclareMathOperator{\Hom}{Hom}
\DeclareMathOperator{\Id}{Id}

\DeclareMathOperator{\Img}{Im}

\DeclareMathOperator{\Iso}{Iso}
\DeclareMathOperator{\Ker}{Ker}

\DeclareMathOperator{\Mod}{Mod}
\DeclareMathOperator{\MOD}{MOD}
\DeclareMathOperator{\Nat}{Nat}

\DeclareMathOperator{\pr}{pr}

\DeclareMathOperator{\Reg}{Reg}

\DeclareMathOperator{\Spec}{Spec}
\DeclareMathOperator{\Sym}{Sym}

\DeclareMathOperator{\tr}{tr}

\newtheorem{thm}{Theorem}[section]
\newtheorem{cor}[thm]{Corollary}
\newtheorem{lem}[thm]{Lemma}
\newtheorem{prop}[thm]{Proposition}

\theoremstyle{definition}
\newtheorem{defn}[thm]{Definition}

\newtheorem*{thm*}{Theorem}
\newtheorem*{defn*}{Definition}

\numberwithin{equation}{section}

\begin{document}

\title{Super Tannakian hulls}

\author{Peter O'Sullivan}
\address{Mathematical Sciences Institute \\
The Australian National University \\
Canberra ACT 2601, Australia}
\email{peter.osullivan@anu.edu.au}
\thanks{}

%
%
%
%
%
%



\date{}

\dedicatory{}

\begin{abstract}
We consider essentially small rigid tensor categories (not necessarily abelian) 
which have a faithful tensor functor to a category of super vector spaces over a field of characteristic $0$.
It is shown how to construct for each such tensor category a super Tannakian hull,
which is a universal faithful tensor functor to a super Tannakian category over a field of characteristic $0$.
The construction is analogous to the passage from an integral domain to its field of fractions.
\end{abstract}

\maketitle


\section{Introduction}

In the theory of motives, rigid tensor categories arise which have a faithful tensor functor
to a category of super vector spaces over a field of characteristic $0$, but which
are not known to be super Tannakian.
In view of the good properties of super Tannakian categories, this raises the questions
of whether for each such rigid tensor category $\sC$ there is a super 
Tannakian category $\sC'$ which most closely approximates it, 
and of how the objects and morphisms of $\sC'$ are related to those of $\sC$.

To describe the situation in more detail, we first fix some terminology.
By a tensor category we mean a symmetric monoidal category
whose hom-sets have structures of abelian group for which
the composition and tensor product are bilinear.
Such a category will be called rigid if every object is dualisable.
A tensor functor between tensor categories is an additive strong symmetric monoidal functor.

We call a tensor category \emph{pseudo-Tannakian} if it is essentially small and
has a faithful tensor functor to a category of super vector spaces over a field
of characteristic $0$.
If $\sD$ is an abelian pseudo-Tannakian category, then $\End_{\sD}(\I)$ is a field
of characteristic $0$, and $\sD$ is a super Tannakian category in the usual sense over this field.
We then say that $\sD$ is super Tannakian.
By a \emph{super Tannakian hull} of a pseudo-Tannakian category $\sC$ we mean a faithful tensor
functor $U:\sC \to \sC'$ with $\sC'$ super Tannakian such that for every super Tannakian 
category $\sD$, composition with $U$ defines an equivalence from the groupoid of 
faithful tensor functors $\sC' \to \sD$ to the groupoid of faithful tensor functors
$\sC \to \sD$.
Such a $\sC'$ if it exists will be the required closest super Tannakian approximation to $\sC$.

That a super Tannakian hull for every pseudo-Tannakian category $\sC$
exists will be proved in Section~\ref{s:supTann}.
It can be described explicitly as follows.
Denote by $\widehat{\sC}$ the category of additive functors from $\sC^{\mathrm{op}}$ to
abelian groups.
Then $\widehat{\sC}$ is abelian and we have a fully faithful additive Yoneda embedding 
\begin{equation*}
\sC \to \widehat{\sC}.
\end{equation*}
The tensor structure of $\sC$ induces a tensor structure on $\widehat{\sC}$,
and the embedding has a canonical structure of tensor functor.
An object $M$ of $\widehat{\sC}$ will be called a \emph{torsion object} if for each object
$B$ of $\sC$ and element $b$ of $M(B)$ there exists a non-zero morphism $a:A \to \I$ in $\sC$
such that
\begin{equation*}
M(a \otimes B):M(B) \to M(A \otimes B)
\end{equation*}
sends $b$ to $0$.
The full subcategory $(\widehat{\sC})_{\mathrm{tors}}$ of $\widehat{\sC}$ consisting of
the torsion objects is a Serre subcategory,
and we may form the quotient
\begin{equation*}
\widetilde{\sC} = \widehat{\sC}/(\widehat{\sC})_{\mathrm{tors}}.
\end{equation*}
It has a unique structure of tensor category such that the projection $\sC \to \widetilde{\sC}$
is a strict tensor functor.
Since $\sC$ is rigid, the composite
\begin{equation*}
\sC \to \widehat{\sC} \to \widetilde{\sC}
\end{equation*}
of the projection with the Yoneda embedding factors through the full tensor subcategory 
$(\widetilde{\sC})_{\mathrm{rig}}$ of $\widetilde{\sC}$ consisting of the dualisable objects.
This factorisation
\begin{equation}\label{e:Tannhull}
\sC \to (\widetilde{\sC})_{\mathrm{rig}}
\end{equation}
is then the required super Tannakian hull (Theorem~\ref{t:Tannhull}).

We may factor any tensor functor essentially uniquely as a strict
tensor functor which is the identity on objects followed by a fully faithful tensor functor.
In the case of \eqref{e:Tannhull}, this factorisation is
\begin{equation*}
\sC \to \sC_\mathrm{fr} \to (\widetilde{\sC})_{\mathrm{rig}}
\end{equation*}
where $\sC_\mathrm{fr}$ is what will be called the \emph{fractional closure} of $\sC$.
A morphism $C \to C'$ in $\sC_\mathrm{fr}$ is an equivalence class of pairs $(h,f)$ with $0 \ne f:A \to A'$
and $h:A \otimes C \to A' \otimes C'$ morphisms in $\sC$ such that the square
\eqref{e:propdef} below commutes, where $(h,f)$ and $(l,g)$ for $0 \ne g:B \to B'$ are equivalent
when the square \eqref{e:propequiv} below commutes. 
If we denote the class of $(h,f)$ by $h/f$, then $\sC \to \sC_\mathrm{fr}$ sends $j$ to $j/1_{\I}$.
It follows in particular that $\sC$ can be embedded in a super Tannakian category if and only if
$\sC$ is fractionally closed, i.e. $\sC \to \sC_\mathrm{fr}$ is an isomorphism, or equivalently for
every $(h,f)$ as above we have $h = f \otimes j$ for some $j:C \to C'$.

The endomorphism ring of $\I$ in $\sC_\mathrm{fr}$ is a field which will be denoted by $\kappa(\sC)$.
It contains the field of fractions of the endomorphism ring of $\I$ in $\sC$, 
but is in general strictly larger.
The tensor category $(\widetilde{\sC})_{\mathrm{rig}}$ has a canonical structure
of super Tannakian category over $\kappa(\sC)$.
Let $k'$ be a field of characteristic $0$.
If we write $k$ for the endomorphism ring of $\I$ in $\sC$,
then for any faithful tensor functor $T$ from $\sC$ to super 
$k'$\nd vector spaces, the $k'$\nd linear extension of $T$ to a functor from 
$k' \otimes_k \sC$ is faithful if and only if $\kappa(\sC)$ is the field of fractions of $k$. 
Any $T$ induces a homomorphism $\rho$ from $\kappa(\sC)$ to $k'$,
and if $k'$ is algebraically closed, a $T$ inducing a given $\rho$ exists (Theorem~\ref{t:Tannequiv}) and 
it is unique up to tensor isomorphism (Corollary~\ref{c:fibfununique}).

The essential point in proving that \eqref{e:Tannhull} has the
required universal property is to show that $(\widetilde{\sC})_{\mathrm{rig}}$ is super
Tannakian.
This will be done by showing that $\widetilde{\sC}$ is a category 
of modules over a transitive affine super groupoid,
with $(\widetilde{\sC})_{\mathrm{rig}}$ consisting of the ones of finite type
(Theorem~\ref{t:Tannequiv}).

Consider first the case where $\sC$ has a faithful tensor functor to a category
of vector spaces over a field of characteristic $0$.
Then \cite[Lemma~3.4]{O} there is a product (not necessarily finite) $G$
of general linear groups over $\Q$ such that, after tensoring with $\Q$ and
passing to the pseudo-abelian hull, $\sC$ is a category of $G$\nd equivariant vector bundles
over an integral affine $G$\nd scheme $X$.
It follows from this that $\widehat{\sC}$ is the category of $G$\nd equivariant
quasi-coherent $\sO_X$\nd modules.

If $G = 1$, then the torsion objects of $\widehat{\sC}$ are the usual torsion sheaves, 
$\widetilde{\sC}$ is the category of vector spaces over the function field $\kappa(\sC)$ of $X$ with
$(\widetilde{\sC})_{\mathrm{rig}}$ the category of finite-dimensional ones,
and \eqref{e:Tannhull} is passage to the generic fibre.

On the other hand suppose that $G$ is arbitrary but that $X$ is of finite type.
Then \eqref{e:Tannhull} is given by pullback onto the ``generic orbit'' $X_0$ of $X$.
Explicitly, $X_0$ is the intersection of the non-empty open $G$\nd subschemes of $X$, 
with $\widetilde{\sC}$ the category of $G$\nd equivariant quasi-coherent $\sO_{X_0}$\nd modules,
and $(\widetilde{\sC})_{\mathrm{rig}}$ the category of $G$\nd equivariant vector bundles
over $X_0$.
The $\Q$\nd algebra of invariants of $H^0(X_0,\sO_{X_0})$ under $G$ is an extension $k$ of $\Q$,
and $X_0$ is a homogeneous quasi-affine $G_k$\nd scheme of finite type.
Then $(\widetilde{\sC})_{\mathrm{rig}}$ is a Tannakian category over $k$ with
category of ind-objects $\widetilde{\sC}$, and $\kappa(\sC) = k$. 

For arbitrary $G$ and $X$ we may write $X$ as the filtered limit $\lim_\lambda X_\lambda$
of integral affine $G$\nd schemes $X_\lambda$ of finite type with dominant transition morphisms, 
and the generic orbits form a filtered inverse system $(X_0{}_\lambda)$ of $G$\nd schemes.
It is not clear whether the limit $\lim_\lambda X_0{}_\lambda$ exists, but
the generic point of $X$ lies in the inverse image of each $X_0{}_\lambda$,
and passing to the generic fibre shows that
$(\widetilde{\sC})_{\mathrm{rig}}$ is the category of representations of
a transitive affine groupoid in $\kappa(\sC)$\nd schemes,
and hence is Tannakian over $\kappa(\sC)$.

In the case of an arbitrary pseudo-Tannakian 
category $\sC$, it is natural to modify the above by taking for
$G$ a product of super general linear groups and for $X$ an appropriate
affine super $G$\nd scheme.
In this case we no longer have an explicit description of $\sC$ as a category of 
equivariant vector bundles or of $\widehat{\sC}$ as a category of equivariant sheaves.
It can still however be shown (Theorem~\ref{t:Ftildeequiv}) 
that $\widetilde{\sC}$ is a category of equivariant sheaves modulo torsion.
It is then possible to argue as above.

The paper is organised as follows.
After recalling some notation and terminology in Section~\ref{s:prelim},
the fractional closure of a tensor category is defined in Section~\ref{s:frac}.
Sections~\ref{s:rep} and \ref{s:free} deal with the connection between free rigid tensor categories
and categories of representations of super general linear groups over a field
of characteristic $0$.
In Sections~\ref{s:fun}--\ref{s:mod} the definitions and basic properties of the
categories $\widehat{\sC}$ and their quotients $\widetilde{\sC}$ modulo torsion are given,
culminating with the fact that for $\sC$ pseudo-Tannakian, $\widetilde{\sC}$ is a 
category of equivariant sheaves modulo torsion over some super scheme.
Such categories of equivariant sheaves are studied in Section~\ref{s:equ},
and it is shown that modulo torsion they are categories of modules over
transitive affine super groupoids.
This result is applied in Section~\ref{s:supTann} to prove the existence 
and basic properties of super Tannakian hulls.
Applications to motives and algebraic cycles will be given in a separate paper.

 \section{Preliminaries}\label{s:prelim}

In this section we fix some notation and terminology for tensor categories and for
super groups and their representations.

By an \emph{additive} category we mean a category enriched over the category 
$\mathrm{Ab}$ of abelian groups.
It is not assumed that direct sums or a zero object exist.
A \emph{tensor category} is an additive category equipped with a structure of symmetric
monoidal category for which the tensor product is bilinear.
A \emph{tensor functor} is an additive strong symmetric monoidal functor. 
A monoidal natural isomorphism will also be called a tensor isomorphism.
It may always be assumed that the units of tensor categories are strict, and that they are strictly preserved
by tensor functors. 

Let $k$ be a commutative ring.
By a \emph{$k$\nd linear category} we mean a category enriched over the category of $k$\nd modules.
If $\sA$ is a cocomplete $k$\nd linear category and $V$ is a $k$\nd module, we have 
for every object $M$ of $\sA$ an object
\begin{equation*}
V \otimes_k M
\end{equation*}
of $\sA$ which represents the functor $\Hom_k(V,\Hom_{\sA}(M,-))$.
Its formation commutes with colimits of $k$\nd modules and in $\sA$, and when $V = k$ it coincides
with $M$.

By a \emph{$k$\nd tensor category} we mean a $k$\nd linear category with a structure of symmetric monoidal
category with bilinear tensor product.
A $k$\nd tensor category is thus a tensor category for which $\End(\I)$ has a structure of 
$k$\nd algebra.
Similarly we define $k$\nd tensor functors.

Let $\sA$ and $\sA'$ be tensor categories.
As well as tensor functors from $\sA$ to $\sA'$, it will sometimes be necessary to consider
more generally lax tensor functors, which are defined as additive symmetric monoidal functors.
Explicitly, a \emph{lax tensor functor} from $\sA$ to $\sA'$ it is an additive functor $H:\sA \to \sA'$ together
with a morphism $\I \to H(\I)$ and morphisms
\begin{equation*}
H(M) \otimes H(N) \to H(M \otimes N),
\end{equation*}
natural in $M$ and $N$, which satisfy conditions similar to those for a tensor functor. 
Any right adjoint to a tensor functor has a unique structure of lax tensor functor 
for which the unit and counit of the adjunction are compatible with the tensor 
and lax tensor structures.

By a \emph{tensor equivalence} we mean a tensor functor $\sA \to \sA'$ which has a quasi-inverse
$\sA' \to \sA$ in the $2$\nd category of tensor categories, tensor functors and tensor isomorphisms.
It equivalent to require that the underlying additive functor be an equivalence.
When a lax tensor functor $H:\sA \to \sA'$ is said to be a tensor equivalence, this will always mean
that $H$ is a tensor functor which is an equivalence in the above sense. 
It is in general \emph{not} sufficient for this merely that the underlying additive functor
of $H$ be an equivalence.

A \emph{dual} of an object $M$ in a tensor category $\sA$ is an object $M^\vee$ of $\sA$
together with a unit $\I \to N^\vee \otimes M$ and a counit $M \otimes M^\vee \to \I$ satisfying
the usual triangular identities.
If such a dual exists it is unique up to unique isomorphism, and $M$ will be said to be dualisable.
Any tensor functor preserves dualisable objects.
For $M$ dualisable, the trace in $\End(\I)$ of an endomorphism  of $M$ is defined,
as is for example the contraction $L \to N$ of a morphism from $L \otimes M$ to $N \otimes M$. 
We say that $\sA$ is \emph{rigid} if every object of $\sA$ is dualisable.
The full subcategory
\begin{equation*}
\sA_\mathrm{rig}
\end{equation*}
of $\sA$ consisting of the dualisable objects is a rigid tensor subcategory of $\sA$ which
is closed under the formation of direct sums and direct summands.

We define in the usual way algebras and commutative algebras in a tensor category $\sA$, and modules 
over such algebras.
Suppose that $\sA$ is abelian, with $\otimes$ right exact. 
Let $R$ be a commutative algebra in $\sA$.
Define the \emph{tensor product $M \otimes_R N$} of $R$\nd modules $M$ and $N$ in $\sA$ as
the target of the universal morphism from $M \otimes N$ to an $R$\nd module in $\sA$ which is an $R$\nd module 
morphism for the actions of $R$ on $M \otimes N$ through either $M$ or $N$.
Explicitly, $M \otimes_R N$ is the coequaliser of the two morphisms from $M \otimes R \otimes N$ to $M \otimes N$
defined by the actions of $R$ on $M$ and $N$. 
We then have a structure of abelian tensor category on the category
\begin{equation*}
\MOD_{\sA}(R)
\end{equation*}
of $R$\nd modules in $\sA$, with unit $R$, tensor product $\otimes_R$, and constraints defined by the
universal property.
As usual we assume the tensor product is chosen so that the unit $R$ is strict.
We write
\begin{equation*}
\Mod_{\sA}(R)
\end{equation*}
for the full rigid tensor subcategory $\MOD_{\sA}(R)_{\mathrm{rig}}$ of $\MOD_{\sA}(R)$.

Let $\sA$ be an abelian tensor category with right exact tensor product,
and $R$ be a commutative algebra in $\sA$.
Then we have a tensor functor $R \otimes -$ from $\sA$ to $\MOD_{\sA}(R)$.
If $\sA'$ is an abelian tensor category with right exact tensor product,
then any lax tensor functor (resp.\  right exact tensor functor) $H$ from $\sA$ to $\sA'$ 
induces a lax tensor functor (resp.\  right exact tensor functor)
from $\MOD_{\sA}(R)$ to $\MOD_{\sA'}(H(R))$.

Let $k$ be a field of characteristic $0$.
The $k$\nd tensor category of super $k$\nd vector spaces will be written
\begin{equation*}
\MOD(k).
\end{equation*}
The full rigid tensor subcategory $\MOD(k)_\mathrm{rig}$ consists of finite-dimensional 
super $k$\nd vector spaces, and will be written $\Mod(k)$.
If $X$ is a super $k$\nd scheme, we write
\begin{equation*}
\MOD(X)
\end{equation*}
for the $k$\nd tensor category of 
quasi-coherent $\sO_X$\nd modules.
The full rigid tensor subcategory $\MOD(k)_\mathrm{rig}$ of $\MOD(X)$ consists of
the vector bundles over $X$, i.e.\ the $\sO_X$\nd modules locally isomorphic to $\sO_X{}\!^{m|n}$,
and will be written $\Mod(X)$.

We denote by $\iota$ the canonical automorphism of order $2$ of the identity functor of the category
of super $k$\nd schemes. 
If $X$ is a super $k$\nd scheme and $\sV$ is a quasi-coherent $\sO_X$\nd module, we denote by 
$\iota_{\sV}$ the canonical automorphism of $\sV$ above $\iota_X$.

A super monoid scheme over $k$ will also be called a \emph{super $k$\nd monoid}.
By a \emph{super $k$\nd monoid with involution} we mean a pair $(M,\varepsilon)$ with $M$ a super
$k$\nd monoid and $\varepsilon$ a $k$\nd point of $M$ with $\varepsilon^2 = 1$ 
such that conjugation by $\varepsilon$ is $\iota_M$.
If $M$ is a super group scheme over $k$ we also speak of a \emph{super $k$\nd group} and 
a \emph{super $k$\nd group with involution}.

Let $(M,\varepsilon)$ be a super $k$\nd monoid with involution.
By an \emph{$(M,\varepsilon)$\nd module} we mean an $M$\nd module $V$ for which
$\varepsilon$ acts as $(-1)^i$ on the summand $V_i$ of degree $i$, 
and by a \emph{representation of $(M,\varepsilon)$} we mean a finite-dimensional 
$(M,\varepsilon)$\nd module.
Every $M$\nd submodule and $M$\nd quotient module of an $(M,\varepsilon)$\nd module is an 
$(M,\varepsilon)$\nd module. 

Let $(G,\varepsilon)$ be a super $k$\nd group with involution.
The $k$\nd tensor category of $(G,\varepsilon)$\nd modules will be written as
\begin{equation*}
\MOD_{G,\varepsilon}(k).
\end{equation*}
The full rigid $k$\nd tensor subcategory $\MOD_{G,\varepsilon}(k)_\mathrm{rig}$ consists of the
representations of $(G,\varepsilon)$, and will be written as $\Mod_{G,\varepsilon}(k)$.

By a  \emph{super $(G,\varepsilon)$\nd scheme} we mean a super $k$\nd scheme $X$
equipped with an action of $G$ such that $\varepsilon$ acts on $X$ as $\iota_X$.
If $X$ is a super $(G,\varepsilon)$\nd scheme,  we write
\begin{equation*}
\MOD_{G,\varepsilon}(X)
\end{equation*}
for the $k$\nd tensor category of 
$(G,\varepsilon)$\nd equivariant quasi-coherent $\sO_X$\nd modules, i.e.\ those 
$G$\nd equivariant quasi-coherent $\sO_X$\nd modules $\sV$ for which 
$\varepsilon$ acts as $\iota_{\sV}$.
The full rigid $k$\nd tensor subcategory $\MOD_{G,\varepsilon}(X)_\mathrm{rig}$ consists 
of the $\sV$ whose underlying $\sO_X$\nd module is a vector bundle, and will be written
$\Mod_{G,\varepsilon}(X)$.

An algebra in $\MOD_{G,\varepsilon}(k)$ will also be called a \emph{$(G,\varepsilon)$\nd algebra}.
If $R$ is a commutative $(G,\varepsilon)$\nd algebra then $X = \Spec(R)$ is a 
$(G,\varepsilon)$\nd scheme, and
\begin{equation*}
\MOD_{G,\varepsilon}(X) = \MOD_{\sA}(R)
\end{equation*}
with $\sA = \MOD_{G,\varepsilon}(k)$.

The full subcategory of the category of super $k$\nd schemes consisting of the 
reduced $k$\nd schemes is coreflective, and the coreflector $X \mapsto X_\mathrm{red}$
preserves finite products.
Thus $X \mapsto X_\mathrm{red}$ sends super $k$\nd groups to $k$\nd groups,
and actions of a super $k$\nd group on a super $k$\nd scheme to actions
of a $k$\nd group on a $k$\nd scheme.

\section{Fractional closures}\label{s:frac}

In this section we define the fractional closure of a tensor category.
For tensor categories with one object,
identified with commutative rings, the fractional closure is the
total ring of fractions.

Let $\sC$ be a tensor category.
A morphism $f$ in $\sC$ will be called \emph{regular} if 
$f \otimes g = 0$ implies $g = 0$ for every morphism $g$ in $\sC$.
A morphism in $\sC$ is regular if and only if its image in the pseudo-abelian hull of $\sC$ is regular.

We say that $\sC$ is \emph{integral} if $1_{\I} \ne 0$ in $\sC$ and the tensor product of two non-zero
morphisms in $\sC$ is non-zero.
Equivalently $\sC$ is integral if $1_{\I} \ne 0$ in $\sC$ and every non-zero morphism of 
$\sC$ is regular.

A tensor functor will be called \emph{regular} if it preserves regular morphisms.
Any faithful tensor functor reflects regular morphisms.
A tensor functor between integral tensor categories is regular if and only if it is faithful.
The embedding of a tensor category into its pseudo-abelian hull is regular.
If direct sums exist in $\sC$, or if $\sC$ is integral,
then a tensor functor $\sC \to \sC'$ is regular if and only if the tensor functor 
it induces on pseudo-abelian hulls is regular.

Given a regular morphism $f:A \to A'$ and objects $C$ and $C'$ in $\sC$, 
denote by
\begin{equation*}
\sC_f(C,C')
\end{equation*}
the subgroup of the hom-group $\sC(A \otimes C,A' \otimes C')$ consisting of those
morphisms $h$ for which the square
\begin{equation}\label{e:propdef}
\begin{gathered}
\xymatrix{
A \otimes (A \otimes C) \ar_{f \otimes h}[d] \ar_{\sigma_{AAC}}^{\sim}[r] & 
A \otimes (A \otimes C) \ar^{f \otimes h}[d] \\
A' \otimes (A' \otimes C') \ar_{\sigma_{A'A'C'}}^{\sim}[r] & A' \otimes (A' \otimes C')
}
\end{gathered}
\end{equation}
commutes, where we write
\begin{equation*}
\sigma_{ABC}:A \otimes (B \otimes C) \iso B \otimes (A \otimes C)
\end{equation*}
for the symmetry.
We have 
\begin{equation*}
\sC_{1_{\I}}(C,C') = \sC(C,C')
\end{equation*}
for every $C$ and $C'$.

A morphism $f$ in $\sC$ will be called \emph{strongly regular} if it is regular and
the embedding
\begin{equation}\label{e:fracclose}
f \otimes -:\sC(C,C') \to \sC_f(C,C')
\end{equation}
is an isomorphism for every pair of objects $C$ and $C'$ in $\sC$.
A morphism in $\sC$ is strongly regular if and only if its image in the pseudo-abelian hull 
of $\sC$ is strongly regular.
We say that $\sC$ is \emph{fractionally closed} if every regular morphism in $\sC$
is strongly regular.

Let $C$ and $C'$ be objects of $\sC$.
We define an equivalence relation on the set of pairs $(h,f)$ with
$f$ a regular morphism in $\sC$ and $h$ in $\sC_f(C,C')$ by calling
$(h,f)$ and $(l,g)$ with $f:A \to A'$ and $g:B \to B'$ equivalent if
the square
\begin{equation}\label{e:propequiv}
\begin{gathered}
\xymatrix{
B \otimes (A \otimes C) \ar_{g \otimes h}[d] \ar_{\sigma_{BAC}}^{\sim}[r] & 
A \otimes (B \otimes C) \ar^{f \otimes l}[d] \\
B' \otimes (A' \otimes C') \ar_{\sigma_{B'A'C'}}^{\sim}[r] & A' \otimes (B' \otimes C')
}
\end{gathered}
\end{equation}
commutes.
That this relation is reflexive is clear from \eqref{e:propdef},
that it is symmetric from the fact that $\sigma_{ABC}$ is the inverse of $\sigma_{BAC}$,
and that it is transitive can be seen as follows:
given $(h_i,f_i)$ for $i = 1,2,3$, if we write $\sD_i$ for the square expressing the equivalence
of the two $(h_j,f_j)$ for $j \ne i$, then $f_2 \otimes \sD_2$ can be obtained by combining
$f_1 \otimes \sD_1$, $f_3 \otimes \sD_3$, and three commutative squares expressing the naturality
of the symmetries.
For a given regular $f$, the pairs $(h_1,f)$ and $(h_2,f)$ are equivalent if and only if $h_1 = h_2$.

Let $f:A \to A'$ be regular in $\sC$.
If $j:B \to A$ and $j':A' \to B'$ are morphisms in $\sC$ with $j' \circ f \circ j$
(and hence $j$ and $j'$) regular, then we have a homomorphism
\begin{equation}\label{e:propcomp}
\sC_f(C,C') \to \sC_{j' \circ f \circ j}(C,C')
\end{equation}
which sends $h$ to $(j' \otimes C') \circ h \circ (j \otimes C)$.
If $l:D \to D'$ is a regular morphism in $\sC$, we have a homomorphism
\begin{equation}\label{e:proptens}
\sC_f(C,C') \to \sC_{l \otimes f}(C,C')
\end{equation}
which, modulo the appropriate associativities, sends $h$ to $l \otimes h$.
The image of $h$ under \eqref{e:propcomp} is the unique element $m$ of $\sC_{j' \circ f \circ j}(C,C')$
with $(m,j' \circ f \circ j)$ equivalent to $(h,f)$,
and similarly for \eqref{e:proptens}.

Define a structure of preorder on the set $\Reg(\sC)$ regular morphisms of $\sC$
by writing $f \le g$ for $f:A \to A'$ and $g:B \to B'$ if 
there exist (necessarily regular) morphisms $l:D \to D'$, $j:B \to D \otimes A$ and 
$j':D' \otimes A' \to B'$ in $\sC$ such that
\begin{equation*}
g = j' \circ (l \otimes f) \circ j.
\end{equation*}
The preorder $\Reg(\sC)$ is filtered, because if $f_1$ and $f_2$
are elements, both $f_1$ and (taking symmetries for $j$ 
and $j'$) $f_2$ are $\le$ $f_2 \otimes f_1$.
It is essentially small if $\sC \to \sC'$ is. 
A tensor functor $\sC \to \sC'$ is regular if it sends every element
of some cofinal subset of $\Reg(\sC)$ to a regular morphism in $\sC'$.
Any regular tensor functor $\sC \to \sC'$ induces an order-preserving map 
$\Reg(\sC) \to \Reg(\sC')$.

For $f \le g$ in $\Reg(\sC)$ and objects $C$ and $C'$ of $\sC$, we have by 
\eqref{e:propcomp} and \eqref{e:proptens} a homomorphism
\begin{equation}\label{e:proptrans}
\sC_f(C,C') \to \sC_g(C,C'),
\end{equation}
necessarily injective, which sends $h$ in $\sC_f(C,C')$ to the unique $m$ in 
$\sC_g(C,C')$ with $(m,g)$ equivalent to $(f,h)$.
We thus have a filtered system $(\sC_f(C,C'))_{f \in \Reg{\sC}}$. 
If $f \le g$ in $\Reg(\sC)$ and $g$ is strongly regular then $f$ is strongly regular.

Suppose that either finite direct sums exist in $\sC$, or that $\sC$ is integral.
Then the embedding  $\sC \to \sC'$ of $\sC$ into its pseudo-abelian hull $\sC'$
induces a cofinal map $\Reg(\sC) \to \Reg(\sC')$.
Thus $\sC$ is fractionally closed if and only if its pseudo-abelian hull is.

Suppose now that $\sC$ is essentially small.
We define as follows a fractionally closed tensor category $\sC_\mathrm{fr}$, 
the \emph{fractional closure of $\sC$}, together with a faithful, regular, strict tensor functor
\begin{equation*}
E_{\sC}:\sC \to \sC_\mathrm{fr}.
\end{equation*} 
The objects of $\sC_\mathrm{fr}$ are those of $\sC$, and the hom groups
are the filtered colimits
\begin{equation}\label{e:frachom}
\sC_\mathrm{fr}(C,C') = \colim_{f \in \Reg(\sC)} \sC_f(C,C'),
\end{equation}
which exist because $\Reg(\sC)$ is essentially small.
Thus $\sC_\mathrm{fr}(C,C')$ is the set of equivalence class of pairs $(h,f)$
with $f$ a regular morphism in $\sC$ and $h$ an element of $\sC_f(C,C')$.
We write $h/f$ for the class of $(h,f)$.
The identity in $\sC_\mathrm{fr}(C,C)$ is $1_C/1_{\I}$.
The composite of $h'/f'$ in $\sC_\mathrm{fr}(\sC)(C',C'')$ and $h/f$ in $\sC_\mathrm{fr}(C,C')$ is defined as
$(h'{}\!_1 \circ h_1)/(f'{}\!_1 \circ f_1)$ for any $(h_1,f_1)$ equivalent to $(h,f)$
and $(h'{}\!_1,f'{}\!_1)$ equivalent to $(h',f')$ with $f'{}\!_1$ and $f_1$ composable:
for $f:A \to A'$ and $f':B' \to B''$ we may take for example $f_1 = B' \otimes f$ and 
$f'{}\!_1 = f' \otimes A'$.
The bilinearity, identity and associativity properties of the composition in $\sC_\mathrm{fr}$
follow from those for $\sC$. 

The functor $E_{\sC}$ is the identity on objects, and on the hom group
$\sC(C,C')$ it is the coprojection of \eqref{e:frachom} at $f = 1_{\I}$, so that
\begin{equation*}
E_{\sC}(j) = j/1_{\I}
\end{equation*}
for any $j:C \to C'$.
The unit and tensor product of objects of $\sC_\mathrm{fr}$ are those of $\sC$,
and for $j$ in $\sC_g(D,D')$ and $h$ in $\sC_f(C,C')$ with $g:B \to B'$ and $f:A \to A'$,
the tensor product $(j/g) \otimes (h/f)$ is $l/(g \otimes h)$ with $l$ defined by 
the commutative square
\begin{equation*}
\xymatrix{
(B \otimes D) \otimes (A \otimes C) \ar_{j \otimes h}[d] \ar^{\sim}_{\sigma_{BDAC}}[r] &
(B \otimes A) \otimes (D \otimes C) \ar^l[d] \\
(B' \otimes D') \otimes (A' \otimes C') \ar^{\sim}_{\sigma_{B'D'A'C'}}[r] &
(B' \otimes A') \otimes (D' \otimes C')
}
\end{equation*}
where $\sigma_{ABCD}:(A \otimes B) \otimes (C \otimes D) \iso (A \otimes C) \otimes (B \otimes D)$
is the symmetry. 
The associativities and symmetries of $\sC_\mathrm{fr}$ are the images under $E_{\sC}$
of those of $\sC$, and the naturality and required compatibilities follow from those in $\sC$.

For $f$ regular in $\sC$, a morphism $h/f$ in $\sC_\mathrm{fr}$ is regular if and only if 
$h$ is regular in $\sC$.
In particular $E_{\sC}$ is regular.
To prove that $\sC_\mathrm{fr}$ is fractionally closed, let $l:A \to A'$ and $g:B \to B'$ 
be regular morphisms in $\sC$, and
\begin{equation*}
h:B \otimes (A \otimes C) \to B' \otimes (A' \otimes C')
\end{equation*}
be a morphism in $\sC_g(A \otimes C,A' \otimes C')$ such that $h/g$ in 
$\sC_\mathrm{fr}(A \otimes C,A' \otimes C')$ lies in $(\sC_\mathrm{fr})_{l/1_{\I}}(C,C')$.
Then the morphism
\begin{equation*}
h':(B \otimes A) \otimes C \to (B' \otimes A') \otimes C'
\end{equation*}
that coincides modulo associativities with $h$ lies in $\sC_{g \otimes l}(C,C')$, and
\begin{equation}\label{e:hgl}
h/g = (l/1_{\I}) \otimes (h'/(g \otimes l))
\end{equation}
in $\sC_\mathrm{fr}(A \otimes C,A' \otimes C')$.
If $g = 1_{\I}$ then $h = h'$, and \eqref{e:hgl} becomes 
\begin{equation*}
h/1_{\I} = (l/1_{\I}) \otimes (h/l).
\end{equation*}
The $h/1_{\I}$ for $h$ in $\Reg(\sC)$ are thus cofinal in $\Reg(\sC_\mathrm{fr})$.
Hence it is enough to prove that $l/1_{\I}$ is strongly regular in $\sC_\mathrm{fr}$
for every $l$ in $\Reg(\sC)$, which follows from \eqref{e:hgl}.

It is clear from the construction that $\sC$ is fractionally closed
if and only if $E_{\sC}$ is fully faithful if and only if $E_\sC$ is an isomorphism
of tensor categories.
If $\sC$ is integral, then $\sC_\mathrm{fr}$ is integral.

Let $T:\sC \to \sD$ be a regular tensor functor.
If a tensor functor $T_1:\sC_\mathrm{fr} \to \sD$ with $T = T_1E_{\sC}$ exists, it is unique,
and $T_1$ is then regular.
Explicitly such a $T_1$ coincides with $T$ on objects, the tensor structural isomorphisms
of $T_1$ are those of $T$, and if $f:A \to A'$ is regular in $\sC$
and $h$ is in $\sC_f(C,C')$, then the morphism
\begin{equation*}
T(h)':T(A) \otimes T(C) \to T(A') \otimes T(C')
\end{equation*}
in $\sD$ that coincides modulo the tensor structural isomorphisms of $T$ with $T(h)$
lies in $\sD_{T(f)}(T(C),T(C'))$, and $T_1(h/f)$ is the unique morphism with
\begin{equation}\label{e:Tfactor}
T(f) \otimes T_1(h/f) = T(h)'.
\end{equation}
Such a $T_1$ exists if $\sD$ is fractionally closed, and more generally
if $T$ sends regular morphisms in $\sC$ to strongly regular morphisms in $\sD$.
If $T$ and  $T'$ are regular tensor functors $\sC \to \sD$ with $T = T_1E_{\sC}$
and $T' = T'{}\!_1E_{\sC}$, 
then any tensor isomorphism $T \iso T'$ is at the same time a tensor isomorphism 
$T_1 \iso T'{}\!_1$.

Let $\sC'$ be an essentially small tensor category.
Then for any regular tensor functor $T:\sC \to \sC'$ we have $E_{\sC'}T = T_\mathrm{fr}E_{\sC}$
for a unique tensor functor 
\begin{equation}\label{e:FrT}
T_\mathrm{fr}:\sC_\mathrm{fr} \to \sC'{}\!_\mathrm{fr},
\end{equation}
and $T_\mathrm{fr}$ is regular.
If $\varphi:T \iso T'$ is a tensor isomorphism, there is a unique tensor isomorphism
$\varphi_\mathrm{fr}:T_\mathrm{fr} \iso T'{}\!_\mathrm{fr}$ with 
$E_{\sC'}\varphi = \varphi_\mathrm{fr}E_{\sC}$.
If $T$ is faithful then $T_\mathrm{fr}$ is faithful, but in general $T$ fully faithful
does not imply $T_\mathrm{fr}$ fully faithful.
However if either finite direct sums exist in $\sC$, or $\sC$ is integral,
then $T_\mathrm{fr}$ is fully faithful for $T:\sC \to \sC'$ the embedding of $\sC$ into
its pseudo-abelian hull, because $T$ then induces a cofinal map $\Reg(\sC) \to \Reg(\sC')$.

The commutative endomorphism ring $\sC_\mathrm{fr}(\I,\I)$ will be important in what follows.
Explicitly, $\sC_f(\I,\I)$ for $f:A \to A'$ regular is the subgroup of $\sC(A,A')$
consisting of those $h:A \to A'$ for which
\begin{equation*}
h \otimes f = f \otimes h,
\end{equation*}
and $h/f = l/g$ in $\sC_\mathrm{fr}(\I,\I)$ when
\begin{equation*}
h \otimes g = f \otimes l.
\end{equation*}
The addition in $\sC_\mathrm{fr}(\I,\I)$ is given by
\begin{equation*}
h/f + h'/f' = (h \otimes f' + f \otimes h')/(f \otimes f')
\end{equation*}
and the product by
\begin{equation*}
(h/f)(h'/f') = (h \otimes h')/(f \otimes f').
\end{equation*}
If $h$ is regular, then $h/f$ in $\sC_\mathrm{fr}(\I,\I)$ is invertible, with inverse $f/h$.
When $\sC$ is an integral tensor category, $\sC_\mathrm{fr}(\I,\I)$ is a field, and we then also write
\begin{equation*}
\kappa(\sC)
\end{equation*}
for $\sC_\mathrm{fr}(\I,\I)$.

Suppose that $\sC$ is rigid.
If $f:A \to A'$ and $g:A'{}^\vee \otimes A \to \I$ correspond to one another by duality,
then explicitly $g$ is obtained from
$f$ by composing the counit $A'{}^\vee \otimes A' \to \I$ with $A'{}^\vee \otimes f$,
and modulo the appropriate associativity $f$ is obtained from $g$ by composing
$A' \otimes g$ with the tensor product of the unit $\I \to A' \otimes A'{}^\vee$ and $A$. 
Thus $g$ is regular if and only if $f$ is, and $f$ and $g$ are then isomorphic in $\Reg(\sC)$.
Hence by \eqref{e:proptrans} every morphism
in $\sC_\mathrm{fr}$ can be written in the form $m/g$ for some regular $g:B \to \I$,
and similarly in the form $m'/g'$ for some regular $g':\I \to B'$. 
If $f$ is regular in $\sC$, we have an isomorphism
\begin{equation*}
\sC_f(C,C') \iso \sC_f(\I,C' \otimes C^\vee)
\end{equation*}
for every $C$ and $C'$ in $\sC$, 
defined using the unit $\I \to C \otimes C^\vee$, and with inverse defined using the
counit $C^\vee \otimes C \to \I$.
Such isomorphisms are compatible with homomorphisms of the form
\eqref{e:fracclose} and \eqref{e:proptrans}.
In particular $\sC$ is fractionally closed if and only if the homomorphism 
\begin{equation}\label{e:fraccloseI}
f \otimes - = - \circ f:\sC(\I,D) \to \sC_f(\I,D)
\end{equation}
is an isomorphism for every regular $f:A \to \I$ and $D$ in $\sC$.

\begin{lem}\label{l:abfracclose}
Every abelian rigid tensor category is fractionally closed.
\end{lem}

\begin{proof}
Let $f:A \to \I$ be a regular morphism in an abelian rigid tensor category $\sC$.
Then $f$ is an epimorphism in $\sC$, because $f \otimes l = l \circ f$ for every
$l:\I \to D$ in $\sC$.
If $h:A \to D$ lies in $\sC_f(\I,D)$ and $i:\Ker f \to A$ is the embedding, then
\begin{equation*}
(h \circ i) \otimes f = (h \otimes f) \circ (i \otimes A) = 
(f \otimes h) \circ (i \otimes A) = 0.
\end{equation*}
Thus $h \circ i = 0$, so that $h = l \circ f$ for some $l:\I \to D$.
Hence \eqref{e:fraccloseI} is an isomorphism.
\end{proof}

Let $k$ be a commutative ring 
and $\sC_1$ and $\sC_2$ be $k$\nd tensor categories.
For $i = 1,2$ we have a canonical $k$\nd tensor functor $I_i:\sC_i \to \sC_1 \otimes_k \sC_2$ 
where $I_1$ sends $M_1$ to $(M_1,\I)$ and $I_2$ sends $M_2$ to $(\I,M_2)$.
If $\End_{\sC_2}(\I) = k$, then $I_1$ is fully faithful.
Given a $k$\nd tensor functor $T_i:\sC_i \to \sC$ for $i = 1,2$, there is a $k$\nd tensor
functor $T:\sC_1 \otimes_k \sC_2 \to \sC$ with $TI_i = T_i$, and such a $T$
is unique up to unique tensor isomorphism $\varphi$ with $\varphi I_i$ the identity of $T_i$.
Indeed we may take for $T$ the composite of $\otimes:\sC \otimes_k \sC \to \sC$ with $T_1 \otimes_k T_2$.
Similar considerations apply for the tensor product of any finite number of categories.

\begin{lem}\label{l:extfaith}
Let $k$ be a field, $\sC$ and $\sC'$ be integral $k$\nd tensor categories with $\sC$
essentially small, and $T:\sC \to \sC'$ be a faithful $k$\nd tensor functor.
Suppose that $\kappa(\sC) = k$.
Then the $k$\nd tensor functor $\sC' \otimes_k \sC \to \sC'$
defined by $\Id_{\sC'}$ and $T$ is faithful.
\end{lem}

\begin{proof}
Let $f'{}\!_1,f'{}\!_2,\dots,f'{}\!_n$ be elements of some hom-space of $\sC'$ which are
linearly independent over $k$.
It is to be shown that if
\begin{equation}\label{e:tensrel}
f'{}\!_1 \otimes T(f_1) + f'{}\!_2 \otimes T(f_2) + \dots + f'{}\!_n \otimes T(f_n) = 0
\end{equation}
in $\sC'$ for elements $f_1,f_2,\dots,f_n$ of some hom-space of $\sC$, then $f_i = 0$ for each $i$.
We argue by induction on $n$.
The required result holds for $n = 1$ because $\sC'$ is integral and $T$ is faithful.
Suppose that it holds for $n < r$.
If we had an equality \eqref{e:tensrel} with $n = r$ and $f_r \ne 0$, then tensoring with $T(f_r)$ would give
\[
\sum_{i=1}^r f'{}\!_i \otimes T(f_i \otimes f_r) = 0 = \sum_{i=1}^r f'{}\!_i \otimes T(f_r \otimes f_i),
\]
so that $\sum_{i=1}^{r-1} f'{}\!_i \otimes T(f_i \otimes f_r - f_r \otimes f_i) = 0$, and hence by induction
$f_i \otimes f_r = f_r \otimes f_i$ for $i < r$.
Since $\kappa(\sC) = k$, this would imply $f_i = \alpha_i f_r$ for $i < r$ with $\alpha_i \in k$, 
and hence by \eqref{e:tensrel}
\[
(\alpha_1 f'{}\!_1 + \dots + \alpha_r f'{}\!_{r-1} + f'{}\!_r) \otimes T(f_r) = 0,
\]
which is impossible by the case $n = 1$.
The required result thus holds for $n = r$.
\end{proof}

\begin{lem}\label{l:tensfaith}
Let $k$ be a field and $\sC$ be an essentially small $k$\nd tensor category.
Then the following conditions are equivalent:
\begin{enumerate}
\renewcommand{\theenumi}{(\alph{enumi})}
\item\label{i:tensfaithkappa} 
$\sC$ is integral with $\kappa(\sC) = k$;
\item\label{i:tensfaithprod}
$1_{\I} \ne 0$ in $\sC$ and the tensor product $\sC \otimes_k \sC \to \sC$ is faithful. 
\end{enumerate}
\end{lem}

\begin{proof}
\ref{i:tensfaithkappa} $\implies$ \ref{i:tensfaithprod}:
Take $T = \Id_{\sC}$ of Lemma~\ref{l:extfaith}.

\ref{i:tensfaithprod} $\implies$ \ref{i:tensfaithkappa}:
Suppose that \ref{i:tensfaithprod} holds.
That $\sC$ is integral is clear.
Let $h/f$ be an element of $\kappa(\sC)$. 
Then $f \otimes h = h \otimes f$.
If $\sC \otimes_k \sC \to \sC$ is faithful, it follows that $I_1(f) \otimes I_2(h) = I_1(h) \otimes I_2(f)$,
so that $h$ lies in the $1$-dimensional $k$\nd subspace generated by $h$, and $h/f$
lies in $\kappa(\sC)$. 
Thus $\kappa(\sC) = k$.
\end{proof}

\begin{lem}\label{l:ext}
Let $k$ be a field, $k'$ be an extension of $k$, and $\sC$ be an essentially small integral 
$k$\nd tensor category with $\kappa(\sC) = k$.
\begin{enumerate}
\item\label{i:extint}
$k' \otimes_k \sC$ is integral with $\kappa(k' \otimes_k \sC) = k'$.
\item\label{i:extfaith}
If a $k$\nd tensor functor from $\sC$ to an integral $k'$\nd tensor category $\sC'$ is faithful,
then its extension to a $k'$\nd tensor functor from $k' \otimes_k \sC$ to $\sC'$ is faithful.
\end{enumerate}
\end{lem}

\begin{proof}
\ref{i:extint} follows from Lemma~\ref{l:tensfaith}, and \ref{i:extfaith} from Lemma~\ref{l:extfaith}
by composing with $I' \otimes_k \sC$ where $I'$ is the embedding of $\I$ into $\sC'$.
\end{proof}

\section{Representations of the super general linear group}\label{s:rep}

This section contains the results on representations of the super general linear group 
that will be required in the next section for the construction of certain quotients of 
free rigid tensor categories by tensor ideals.

Let $k$ be a field of characteristic $0$.
We denote by $|V|$ the underlying $k$\nd vector space of a super $k$\nd vector space $V$.
Let $A$ be a commutative super $k$\nd algebra.
Then $|A|$ is a (not necessarily commutative) $k$\nd algebra.
If we regard $|A|$ as a right $|A|$\nd module over itself, then elements $a$ of $|A|$ may be identifed with
endomorphisms $a-$ of right $|A|$\nd modules.
Given a $k$\nd linear map $f:|V| \to |V'|$, we have for every $a$ in $|A|$ a morphism
$f \otimes_k a:|V| \otimes_k |A| \to |V'| \otimes_k |A|$ of right $|A|$\nd modules,
which when $f$ and $a$ are homogeneous of the same degree underlies a morphism
\begin{equation*}
f \otimes_k a:V \otimes_k A \to V' \otimes_k A
\end{equation*}
of $A$\nd modules.
Given also $f:|V'| \to |V''|$ and $a'$ in $|A|$, we have
\begin{equation*}
(f' \otimes_k a') \circ (f \otimes_k a) = (f' \circ f) \otimes_k a'a.
\end{equation*} 
If $V$, $V'$, $W$, $W'$ are super $k$\nd vector spaces concentrated in respective degrees $v$, $v'$, $w$, $w'$,
and if $a$ and $a'$ in $|A|$ are homogeneous of respective degrees $v+v'$ and $w+w'$,
then for any $f:|V| \to |V'|$, $g:|W| \to |W'|$ we have 
\begin{equation}\label{e:fagb}
(f \otimes_k a) \otimes_A (g \otimes_k b) = (f \otimes_k g) \otimes_k ((-1)^{(v+v')w'}ab),
\end{equation}
where we identify for example $(V \otimes_k A) \otimes_A (W \otimes_k A)$ with $(V \otimes_k W) \otimes_k A$
using the tensor structure of the functor $- \otimes_k A$.

We write $\rM_m$ for the ring scheme over $k$ of endomorphisms of $k^m$, and
$\rM_{m|n}$ for the super ring scheme over $k$ of endomorphisms of $k^{m|n}$.
Explicitly, a point of $\rM_{m|n}$ in the commutative super $k$\nd algebra $A$ is an $A$\nd endomorphism
of the $A$\nd module $k^{m|n} \otimes_k A$.
Such a point may be identified with an endomorphism of degree $0$ of the right $|A|$\nd module
\begin{equation*}
|k^{m|n}| \otimes_k |A| = k^{m+n} \otimes_k |A|,
\end{equation*}
and hence with an $(m + n) \times (m + n)$ matrix with entries in the diagonal
$m \times m$ and $n \times n$ blocks in $A_0$ and entries in the two off-diagonal
blocks in $A_1$.
We denote by $E$ the standard representation of $\rM_{m|n}$ on $k^{m|n}$, which assigns to a point 
of $\rM_{m|n}$ in $A$ its defining $A$\nd endomorphism of $k^{m|n} \otimes_k A$.

The point $\varepsilon_{m|n}$ of $\rM_{m|n}(k)$ which is $1$ on the
diagonal $m \times m$ block, $-1$ on the diagonal $n \times n$ block, and $0$ on the off-diagonal blocks,
defines a structure $(\rM_{m|n},\varepsilon_{m|n})$ of super $k$\nd monoid with involution on $\rM_{m|n}$.
We usually write $\varepsilon_{m|n}$ simply as $\varepsilon$.
The standard representation $E$ of $\rM_{m|n}$ is a representation of $(\rM_{m|n},\varepsilon)$.

\begin{lem}\label{l:standend}
Every $\rM_{m|n}$\nd endomorphism of $E^{\otimes d}$ is a $k$\nd linear combination of symmetries of $E^{\otimes d}$.
\end{lem}

\begin{proof}

For $i,j = 1,2, \dots ,m+n$, write $e_{ij}$ for the $k$\nd endomorphism of $|E| = k^{m+n}$ 
whose matrix has $(i,j)$th entry $1$ and all other entries $0$.
If $[\sigma]$ denotes the symmetry of $E^{\otimes d}$ defined by $\sigma$ in $\mathfrak{S}_d$,
then the $k$\nd endomorphisms of $|E|^{\otimes d}$ which commute with every symmetry of $E^{\otimes d}$
are $k$\nd linear combinations of those of the form
\begin{equation}\label{e:esigma}
\sum_{\sigma \in \mathfrak{S}_d} [\sigma] \circ 
(e_{i_1j_1} \otimes_k \dots \otimes_k e_{i_dj_d})
\circ [\sigma^{-1}]
\end{equation}
The $k$\nd subalgebra $S$ of $\End_{\rM_{m|n}}(E^{\otimes d})$ generated by the symmetries of 
$E^{\otimes d}$ is a quotient of $k[\mathfrak{S}_d]$, and hence semisimple. 
Thus $S$ is its own bicommutant in $\End_k(|E|^{\otimes d})$.
Hence it is enough to show that every $\rM_{m|n}$\nd endomorphism of $E^{\otimes d}$ commutes
with every $k$\nd endomorphism of $|E|^{\otimes d}$ of the form \eqref{e:esigma}.

When $(i_r,j_r) = (i_s,j_s)$ in \eqref{e:esigma} for some $r \ne s$ with $(i_r,j_r)$ from the off-diagonal blocks, 
\eqref{e:esigma} vanishes: if $\sigma_{rs}$ is the symmetry of $E^{\otimes d}$ that interchanges 
the $r$th and $s$th factor $E$, then for any $\tau$ the terms of \eqref{e:esigma} with $\sigma = \tau$ and 
$\sigma = \tau\sigma_{rs}$ cancel.

Let $h$ be an $\rM_{m|n}$\nd endomorphism of $E^{\otimes d}$.
To show that $h$ commutes with \eqref{e:esigma}, we may suppose that $(i_r,j_r) \ne (i_s,j_s)$ in \eqref{e:esigma}
when $r \ne s$ and $(i_r,j_r)$ lies in the off-diagonal blocks.
Denote by $A$ the commutative super $k$\nd algebra freely generated by the elements $t_{ij}$ for $i,j = 1, \dots, m+n$
with $t_{ij}$ of degree $0$ for $(i,j)$ in the diagonal blocks and of degree $1$ for $(i,j)$ 
in the off-diagonal blocks.
The element of $\rM_{m|n}(A)$ with matrix $(t_{ij})$ acts on $E \otimes_k A$ as the $A$\nd endomorphism 
\begin{equation*}
\sum_{i,j = 1}^{m+n}e_{ij} \otimes_k t_{ij}
\end{equation*}
and hence on $E^{\otimes d} \otimes_k A$ as the tensor power
\begin{equation}\label{e:thetasum}
(\sum_{i,j = 1}^{m+n}e_{ij} \otimes_k t_{ij})^{\otimes d} = 
\sum_{i_1,j_1,\dots, i_d,j_d = 1}^{m+n} \theta_{i_1,j_1, \dots ,i_d,j_d}
\end{equation} 
over $A$ of $A$\nd endomorphisms, where
\begin{equation*}
\theta_{i_1,j_1, \dots ,i_d,j_d} 
= (e_{i_1j_1} \otimes_k t_{i_1j_1}) \otimes_A \dots \otimes_A (e_{i_dj_d} \otimes_k t_{i_dj_d}).
\end{equation*}
Thus $h \otimes_k A$ commutes with \eqref{e:thetasum}. 
Since $- \otimes_k A$ is a tensor functor, the symmetries of $E^{\otimes d} \otimes_k A$ are the
$[\sigma] \otimes_k A$, so that by their naturality
\begin{equation}\label{e:thetaconj}
([\sigma] \otimes_k A) \circ \theta_{i_1,j_1, \dots ,i_d,j_d} \circ ([\sigma^{-1}] \otimes_k A)
= \theta_{i_{\sigma^{-1}(1)},j_{\sigma^{-1}(1)}, \dots ,i_{\sigma^{-1}(d)},j_{\sigma^{-1}(d)}}
\end{equation}
for every $\sigma$.
Repeatedly applying \eqref{e:fagb} shows that
\begin{equation}\label{e:thetat}
\theta_{i_1,j_1, \dots ,i_d,j_d} = 
(e_{i_1j_1} \otimes_k \dots \otimes_k e_{i_dj_d}) \otimes_k 
(\delta t_{i_1j_1} \dots t_{i_dj_d})
\end{equation}
with $\delta = \pm 1$.
Now $t_{i_1j_1} \dots t_{i_dj_d}$ is non-zero if and only if the $(i_r,j_r)$ lying in the off-diagonal blocks
are distinct, and when this is so $t_{i'_1j'_1} \dots t_{i'_dj'_d}$ generates the same $1$\nd dimensional 
subspace of the $k$\nd vector space $|A|$ as $t_{i_1j_1} \dots t_{i_dj_d}$ if and only if
\begin{equation*}
(i'_1,j'_1, \dots ,i'_d,j'_d) = 
(i_{\sigma^{-1}(1)},j_{\sigma^{-1}(1)}, \dots ,i_{\sigma^{-1}(d)},j_{\sigma^{-1}(d)})
\end{equation*}
for some $\sigma \in \mathfrak{S}_d$.
Further the distinct $1$\nd dimensional subspaces generated in this way are linearly independent.
It thus follows from \eqref{e:thetaconj} and \eqref{e:thetat} that $h \otimes_k A$ commutes with
\begin{equation*}
\sum_{\sigma \in \mathfrak{S}_d} 
([\sigma] \otimes_k A) \circ \theta_{i_1,j_1, \dots ,i_d,j_d} \circ ([\sigma^{-1}] \otimes_k A),
\end{equation*}
and hence from \eqref{e:thetat} that $h$ commutes with \eqref{e:esigma}.
\end{proof}

\begin{lem}\label{l:Mhomtens}
Let $(M,\varepsilon)$ and $(M',\varepsilon')$ be super $k$\nd monoids with involution,
$V$ be a representation of $(M,\varepsilon)$ and $V'$ a representation of $(M',\varepsilon')$,
and $W$ be a $(M,\varepsilon)$\nd module and $W'$ an $(M',\varepsilon')$\nd module.
Then the canonical $k$\nd linear map
\begin{equation*}
\Hom_M(V,W) \otimes_k \Hom_{M'}(V',W') \to
\Hom_{M \times_k M'}(V \otimes_k V',W \otimes_k W')
\end{equation*}
is an isomorphism.
\end{lem}

\begin{proof}
The canonical homomorphism is that induced by the canonical isomorphism
\begin{equation}\label{e:homtens}
\Hom_k(|V|,|W|) \otimes_k \Hom_k(|V'|,|W'|) \iso
\Hom_k(|V| \otimes_k |V'|,|W| \otimes_k |W'|).
\end{equation}
It is thus enough to show that every $(M \times_k M')$\nd homomorphism $f$ from 
$V \otimes_k V'$ to $W \otimes_k W'$
lies in the image under \eqref{e:homtens} of both
\begin{equation*}
\Hom_M(V,W) \otimes_k \Hom_k(V',W')
\end{equation*}
and the similar subspace defined using $M'$\nd homomorphisms.  
In fact $f$ sends $V \otimes_k V'{}\!_i$ to $W \otimes_k W'{}\!_i$ because $(1,\varepsilon')$ acts on 
them as $(-1)^i$.
Regarding $V \otimes_k V'$ and $W \otimes_k W'$ as $M$\nd modules by restricting to the factor $M$ 
thus gives the required result for $M$\nd homomorphisms.
The argument for $M'$\nd homomorphisms is similar.
\end{proof}

Let $\bm = (m_\gamma)_{\gamma \in \Gamma}$ and $\bn = (n_\gamma)_{\gamma \in \Gamma}$ be families of integers
$\ge 0$ indexed by the same set $\Gamma$.
We write
\begin{equation*}
\rM_{\bm|\bn} = \prod_{\gamma \in \Gamma} \rM_{m_\gamma|n_\gamma}.
\end{equation*}
If $\varepsilon_{\bm|\bn}$ is the $k$\nd point 
$(\varepsilon_{m_\gamma|n_\gamma})_{\gamma \in \Gamma}$ of $M_{\bm|\bn}$,
then $(\rM_{\bm|\bn},\varepsilon_{\bm|\bn})$ is a super $k$\nd monoid with involution,
which we usually write simply $(\rM_{\bm|\bn},\varepsilon)$. 
We also write $E_\gamma$ for the representation of $(\rM_{\bm|\bn},\varepsilon)$ 
given by inflation along the projection
$\rM_{\bm|\bn} \to \rM_{m_\gamma|n_\gamma}$ of the standard representation $E$ of 
$(\rM_{m_\gamma|n_\gamma},\varepsilon)$.

\begin{lem}\label{l:standss}
Let $\bm$ and $\bn$ be families of integers $\ge 0$ indexed by $\Gamma$.
Then the full subcategory of the category of representations of $(\rM_{\bm|\bn},\varepsilon)$
consisting of the direct summands of direct sums representations of the form $\bigotimes_{i=1}^r E_{\gamma_i}$
for $\gamma_i \in \Gamma$ is semisimple abelian.
\end{lem}

\begin{proof}
It is enough to show that the $k$\nd algebra $\End_{\rM_{\bm|\bn}}(V)$ is semisimple for $V$ a direct sum of 
representations $\bigotimes_i E_{\gamma_i}$.
The centre of $\rM_{\bm|\bn}$ is $\prod_{\gamma \in \Gamma}\rM_1$ with each factor diagonally embedded, 
and the monoid of central characters is the free commutative monoid on $\Gamma$.
Since $\bigotimes_i E_{\gamma_i}$ has central character $\sum_i \gamma_i$, the hom-space between two 
non-isomorphic representations of this form is $0$.
Thus we may suppose that $V = (\bigotimes_i E_{\gamma_i})^r$ for some $r$.
Since $\End_{\rM_{\bm|\bn}}(W^r)$ is the tensor product over $k$ of $\End_{\rM_{\bm|\bn}}(W)$ 
with the full matrix algebra $\rM_r(k)$, we may suppose further that $r = 1$.
By Lemma~\ref{l:Mhomtens} we may suppose finally that $V = E_\gamma{}\!^{\otimes d}$ for some $\gamma$.
Since the group algebra $k[\mathfrak{S}_d]$ is semisimple, the required result then follows
from Lemma~\ref{l:standend}.
\end{proof}

\begin{lem}\label{l:subobj}
Let $\sC$ be an abelian category and $\sC_0$ be a strictly full abelian subcategory of $\sC$
for which the embedding $\sC_0 \to \sC$ is left exact.
Suppose that every object of $\sC$ is a subobject of an object of $\sC_0$.
Then $\sC_0 = \sC$.
\end{lem}

\begin{proof}
Let $N$ be an object of $\sC$.
By hypothesis, $N$ is a subobject of an object $N_0$ in $\sC_0$, 
and $N_0/N$ is a subobject of an object $M_0$ in $\sC_0$.
Then $N$ is the kernel of the composite $N_0 \to N_0/N \to M_0$, and hence lies in $\sC_0$. 
\end{proof}

\begin{prop}\label{p:Msummand}
Let $\bm$ and $\bn$ be families of integers $\ge 0$ indexed by $\Gamma$.
\begin{enumerate}
\item\label{i:Msummandss}
The category of representations of $(\rM_{\bm|\bn},\varepsilon)$ is semisimple abelian.
\item\label{i:Msummand}
Every representation of $(\rM_{\bm|\bn},\varepsilon)$ is a direct summand of a direct sum of a 
representations of the form $\bigotimes_{i=1}^r E_{\gamma_i}$ for $\gamma_i \in \Gamma$.
\end{enumerate}
\end{prop}

\begin{proof}
By Lemma~\ref{l:standss} it is enough to prove \ref{i:Msummand}.
Let $V$ be a representation of $(\rM_{\bm|\bn},\varepsilon)$.
Choosing an isomorphism of super $k$\nd vector spaces $V \iso k^{t|u}$
gives an embedding
\[
0 \to V \to k^{t|u} \otimes_k k[\rM_{\bm|\bn}]
\]
of $\rM_{\bm|\bn}$\nd modules, where $\rM_{\bm|\bn}$ acts trivially on $k^{t|u}$ and $k[\rM_{\bm|\bn}]$ 
is the right regular $\rM_{\bm|\bn}$\nd module.
By the isomorphism
\begin{equation*}
k[\rM_{\bm|\bn}] \iso \Sym(\coprod_{\gamma \in \Gamma} k^{m_\gamma|n_\gamma} \otimes_k E_\gamma)
\end{equation*}
of $\rM_{\bm|\bn}$\nd modules, $k^{t|u} \otimes_k k[\rM_{\bm|\bn}]$ is a coproduct of $\rM_{\bm|\bn}$\nd modules 
of the form either $\bigotimes_i E_{\gamma_i}$ or $k^{0|1} \otimes_k \bigotimes_i E_{\gamma_i}$.
There are no non-zero $\rM_{\bm|\bn}$\nd homomorphisms to $k^{0|1} \otimes_k \bigotimes_i E_{\gamma_i}$ from $V$,
because $\varepsilon$ acts on its even part as $-1$ and on its odd part as $1$. 
It follows that $V$ can be embedded in a direct sum of representations $\bigotimes_i E_{\gamma_i}$.
Taking for $\sC$ in Lemma~\ref{l:subobj} the category of representations of $(\rM_{\bm|\bn},\varepsilon)$,
for $\sC_0$ the strictly full subcategory consisting of the direct summands of direct sums
of representations $\bigotimes_i E_{\gamma_i}$, and using Lemma~\ref{l:standss}, 
now gives the required result.
\end{proof}

A point of $\rM_{m|n}$ in a commutative super $k$\nd algebra $A$ may be written uniquely as
$g + \alpha$, with $g$ the in the diagonal blocks and $\alpha$ in the off-diagonal blocks.
Then $g$ is a point of $\rM_m \times \rM_n$ in $A_0$.
A point of $\rM_{m|n}$ in $A$ is invertible if and only the corresponding point in $A_{\mathrm{red}}$ is invertible:
reduce to the case where the point in $A_{\mathrm{red}}$ is the identity.
The functor that sends $A$ to the group of invertible elements of $\rM_{m|n}(A)$ is thus represented by an affine open
super $k$\nd subgroup $\GL_{m|n}$ of $\rM_{m|n}$ with $(\GL_{m|n})_{\mathrm{red}}$ the open 
$k$\nd subgroup $\GL_m \times \GL_n$ of $(\rM_{m|n})_{\mathrm{red}} = \rM_m \times \rM_n$.
The point $g + \alpha$ of $\rM_{m|n}$ lies in $\GL_{m|n}$ if and only if the point $g$ of $\rM_m \times \rM_n$
lies in $\GL_m \times \GL_n$.
Write
\[
q \in k[\rM_{m|n}]_0
\]
for the pullback along the projection $g + \alpha \mapsto g$ from $\rM_{m|n}$ to $\rM_m \times \rM_n$
of the element of $k[\rM_m \times \rM_n]$ that sends $(h,j)$ to $\det(h)\det(j)$. 
Then $\GL_{m|n}$ is the open super subscheme $(\rM_{m|n})_q$ of $\rM_{m|n}$ where $q$ is invertible, and 
\begin{equation}\label{e:kGLkM}
k[\GL_{m|n}] = k[\rM_{m|n}]_q
\end{equation}
is obtained from $k[\rM_{m|n}]$ by inverting $q$.

If $M$ is an affine super $k$\nd monoid and $V$ is an $M$\nd module with defining homomorphism
\begin{equation*}
\mu:V \to V \otimes_k k[M]
\end{equation*}
then the \emph{coefficient} of $V$ associated to a $v$ in $|V|$ and a $k$\nd linear map $\pi:|V| \to k$
is the element $(\pi \otimes_k k[M])(\mu(v))$ of $k[M]$.
By \eqref{e:kGLkM}, the embedding of $\GL_{m|n}$ into $\rM_{m|n}$ defines an embedding
of $k[\rM_{m|n}]$ into $k[\GL_{m|n}]$.
The category of $\rM_{m|n}$\nd modules may thus be identified with the full subcategory of
the category of $\GL_{m|n}$\nd modules consisting of those $\GL_{m|n}$\nd modules for which
every coefficient lies in $k[M_{m|n}]$.

The point $\varepsilon = \varepsilon_{m|n}$ of $\rM_{m|n}(k)$ lies in $\GL_{m|n}(k)$, so that
we have a super $k$\nd group with involution $(\GL_{m|n},\varepsilon)$.

\begin{lem}\label{l:coeff}
For every integer $d \ge 0$ there exists a representation $W \ne 0$ of $(\rM_{m|n},\varepsilon)$ such that 
every coefficient of $W$ lies in $q^dk[\rM_{m|n}]$.
\end{lem}

\begin{proof}
It is enough to prove that there exists a representation $W \ne 0$ of $(\GL_{m|n},\varepsilon)$ 
such that every coefficient of $W$ lies in $q^dk[\rM_{m|n}] \subset k[\GL_{m|n}]$.
Let $r$ be an integer $\ge 0$, and $W_r$ be the $\GL_{m|n}$\nd submodule of the right regular $\GL_{m|n}$\nd module
$k[\GL_{m|n}]$ generated by $q^{2r}$.
Explicitly, if the coalgebra $k$\nd homomorphism is
\[
\mu:k[\GL_{m|n}] \to k[\GL_{m|n}] \otimes_k k[\GL_{m|n}],
\] 
then $W_r$ is the smallest super $k$\nd vector subspace $V$ of $k[\GL_{m|n}]$ such that
\begin{equation}\label{e:muqr}
\mu(q^{2r}) \in V \otimes_k k[\GL_{m|n}]. 
\end{equation}
Since left translation by $\varepsilon$ leaves $q^{2r}$ fixed, 
$W_r$ is contained in the super $\GL_{m|n}$\nd subspace of $k[\GL_{m|n}]$ 
of invariants under left translation by $\varepsilon$.
Thus $W_r$ is a representation of $(\GL_{m|n},\varepsilon)$,
because conjugation of $\GL_{m|n}$ by $\varepsilon$ acts on $k[\GL_{m|n}]$ as $(-1)^i$ in degree $i$.

Let $g + \alpha$ and $g' + \alpha'$ be points of $\GL_{m|n}$ in a commutative super $k$\nd algebra $A$.
Then
\[
q^{2r}((g+\alpha)(g'+\alpha')) = q^{2r}(gg' + \alpha\alpha') = 
q^{2r}(gg')q^{2r}(1 + (gg')^{-1}\alpha\alpha').
\]
Since $(gg')^{-1}$ is $q^{-1}(gg')a(gg')$ with the entries of $a(gg')$ polynomials
in those of $gg'$, we may write $q^{2r}(1 + (gg')^{-1}\alpha\alpha')$ as a sum of terms of the form
\begin{equation*}
q^{-l}(g)q^{-l}(g')p(g,g')\alpha_1 \dots \alpha_l\alpha'\!_1 \dots \alpha'\!_l,
\end{equation*}
where $p(g,g')$ is a polynomial in the entries of $g$ and $g'$,
and the $\alpha_i$ and $\alpha'\!_i$ are entries of $\alpha$ and $\alpha'$.
The product of the $\alpha_i$ or $\alpha'\!_i$ is $0$ if $l > 2mn$,
so that 
\begin{equation*}
\mu(q^{2r}) \in q^{2r-2mn}k[\rM_{m|n}] \otimes_k q^{2r-2mn}k[\rM_{m|n}].
\end{equation*}
Thus \eqref{e:muqr} holds with $V = q^{2r-2mn}k[\rM_{m|n}]$, so that 
\begin{equation*}
W_r \subset q^{2r-2mn}k[\rM_{m|n}].
\end{equation*}
Suppose $r \ge mn$.
Then applying $\mu$ and using the fact that
$\mu$ is a morphism of super $k$\nd algebras sending $k[\rM_{m|n}]$ into
$k[\rM_{m|n}] \otimes_k k[\rM_{m|n}]$ shows that
\begin{equation*}
\mu(W_r) \subset q^{2r-4mn}k[\rM_{m|n}]  \otimes_k q^{2r-4mn}k[\rM_{m|n}].
\end{equation*}
It follows that every coefficient of $W_r$ lies in $q^{2r-4mn}k[\rM_{m|n}]$. 
We may thus take $W = W_r$ for $r \ge 2mn + d/2$.
\end{proof}

Let $\bm$ and $\bn$ be families of integers $\ge 0$ indexed by $\Gamma$. 
We write
\begin{equation*}
\GL_{\bm|\bn} = \prod_{\gamma \in \Gamma}\GL_{m_\gamma|n_\gamma}.
\end{equation*}
The $k$\nd point $\varepsilon = \varepsilon_{\bm|\bn}$ of $\rM_{\bm|\bn}$ lies in $\GL_{\bm|\bn}$,
and we have an affine super $k$\nd group with involution $(\GL_{\bm|\bn},\varepsilon)$.
As with $\rM_{m|n}$ and $\GL_{m|n}$, we may identify $\rM_{\bm|\bn}$\nd modules with 
$\GL_{\bm|\bn}$\nd modules whose coefficients lie in $k[M_{\bm|\bn}]$.

\begin{thm}\label{t:VtensW}
For every representation $V$ of $(GL_{\mathbf{m}|\mathbf{n}},\varepsilon)$ there exists a representation $W \ne 0$ of 
$(M_{\bm|\bn},\varepsilon)$ such that $V \otimes_k W$ is a representation of $(M_{\bm|\bn},\varepsilon)$.
\end{thm}

\begin{proof}
Write $q_\gamma$ for the image of $q$ in $k[\rM_{m_\gamma|n_\gamma}]$ under the embedding
of $k[\rM_{m_\gamma|n_\gamma}]$ into $k[\rM_{\bm|\bn}]$ defined by the projection from $\rM_{\bm|\bn}$ to 
$\rM_{m_\gamma|n_\gamma}$.
The coefficients of $V$ generate a finite-dimensional $k$\nd vector subspace of $k[\GL_{m|n}]$, and hence lie in
\begin{equation*}
(\prod_{\gamma \in \Gamma_0}q_\gamma{}\!^{-d})k[\rM_{\bm|\bn}]
\end{equation*}
for some finite subset $\Gamma_0$ of $\Gamma$ and $d \ge 0$.
Since the coefficients of $V \otimes_k W$ are linear combinations of products of those of $V$ and $W$, 
it follows from Lemma~\ref{l:coeff} that we may take $W = \bigotimes_{\gamma \in \Gamma_0} W_\gamma$
with $W_\gamma \ne 0$ for $\gamma \in \Gamma$ an appropriate representation of the factor 
$\rM_{m_\gamma|n_\gamma}$ of $\rM_{\bm|\bn}$. 
\end{proof}

\begin{cor}\label{c:quotsub}
Let $\bm$ and $\bn$ be families of integers $\ge 0$ indexed by $\Gamma$. 
Then every representation of $(\GL_{\bm|\bn},\varepsilon)$ is a 
subrepresentation (resp.\ quotient representation)
of a direct sum of representations of the form 
$\bigotimes_{i=1}^r E_{\gamma_i} \otimes_k \bigotimes_{j=1}^s E_{\gamma'{}\!_j}\!^{\vee}$ for 
$\gamma, \gamma'\!{}_j \in \Gamma$.
\end{cor}

\begin{proof}
Let $V$ be a representation of $(GL_{\bm|\bn},\varepsilon)$.
If $W \ne 0$ is as in Theorem~\ref{t:VtensW}, then by Proposition~\ref{p:Msummand},
both $W$ and $V \otimes_k W$ are direct summands of direct sums of representations 
of the form $\bigotimes_i E_{\gamma_i}$.
Since $V \otimes \eta_W$ with $\eta_W:k \to W \otimes_k W^\vee$ the unit 
embeds $V$ into $V \otimes_k W \otimes_k W^\vee$, the result for subrepresentations follows.
The result for quotients follows by taking duals.
\end{proof}

\section{Free rigid tensor categories}\label{s:free}

This section deals with the connection between free rigid tensor categories
and categories of representations of super general linear groups over a field
of characteristic $0$.

Let $k$ be a field of characteristic $0$.
Given $r \in k$ we denote by 
\begin{equation*}
\sF_r 
\end{equation*}
the free rigid $k$\nd tensor category on a dualisable
object $N$ of rank $r$.
If $M$ is a dualisable object of rank $r$ in a $k$\nd tensor category $\sC$,
there is then a $k$\nd tensor functor from $\sF_r$ to $\sC$ that sends $N$ to $M$,
and if $T_1$ and $T_2$ are tensor functors
from $\sF_r$ to $\sC$ and $\theta$ is an isomorphism from $T_1(N)$ to $T_2(N)$, 
there is a unique tensor isomorphism $\varphi$ from $T_1$ to $T_2$ with $\varphi_N = \theta$.
The objects of $\sF_r$ are tensor products of copies of $N$ and $N^\vee$.
The full tensor subcategory $\sH$ of $\sF_r$ consisting of tensor powers of $N$ is independent of
$r$ and is the free $k$\nd tensor category on $N$.
The symmetries of $N^{\otimes d}$ then form a basis for its endomorphism algebra, so that
\begin{equation*}
\End(N^{\otimes d}) = k[\mathfrak{S}_d],
\end{equation*}
while if $d_1 \ne d_2$ we have $\Hom(N^{\otimes d_1},N^{\otimes d_2}) = 0$.

Dualising shows that a tensor ideal $\sJ$ of $\sF_r$ is completely determined 
by its restriction to $\sH$, and hence 
by the ideals $\sJ(N^{\otimes d},N^{\otimes d})$ of the $k$\nd algebras $\End(N^{\otimes d})$.
A sequence of ideals $\sJ^{(d)}$ of the $\End(N^{\otimes d})$ arises from a tensor ideal of
$\sF_r$ if and only if $\sJ^{(d)} \otimes N$ is contained in $\sJ^{(d+1)}$
and contraction sends $\sJ^{(d+1)}$ into $\sJ^{(d)}$
for each $d$.

Write $c_\lambda$ for the Young symmetriser in $k[\mathfrak{S}_d]$ associated to 
the partition $\lambda$ of $d$.
The two-sided ideal of $k[\mathfrak{S}_d]$ generated by $c_\lambda$ consists of those
elements that act trivially on all irreducible representations of $\mathfrak{S}_d$ 
other than the one associated to $\lambda$.
Every minimal two-sided ideal of $k[\mathfrak{S}_d]$ is of this form, for a unique $\lambda$.
If $\lambda'$ is a partition of  $d' \ge d$ and $\mathfrak{S}_d$ is embedded in 
$\mathfrak{S}_{d'}$ by an embedding of $[1,d]$ into $[1,d']$, then $c_{\lambda'}$ lies
in the two-sided ideal of $k[\mathfrak{S}_{d'}]$ generated by $c_\lambda$ if and only 
if the diagram $[\lambda']$ contains $[\lambda]$ \cite[4.44]{FulHar}.

Suppose now that $r$ is an integer.
Let $m|n$ be a pair of integers $\ge 0$ with $m-n = r$.
Then there exists a $k$\nd tensor functor from $\sF_r$ to the category 
of super $k$\nd vector spaces, unique up to tensor isomorphism, that sends 
$N$ to a super $k$\nd vector space of dimension $m|n$.
Its kernel is a tensor ideal
\begin{equation*}
\sJ_{m|n}
\end{equation*}
of $\sF_r$.
Denote by $\lambda_{m|n}$ the partition of $(m+1)(n+1)$ such that $[\lambda_{m|n}]$ has $m+1$ columns
and $n+1$ rows.
Then \cite[1.9]{Del02}
\begin{equation*}
c_\lambda \in \sJ_{m|n}(N^{\otimes d},N^{\otimes d}) \subset \End(N^{\otimes d}) = k[\mathfrak{S}_d]
\end{equation*}
if and only if $[\lambda]$ contains $[\lambda_{m|n}]$, 
so that $\sJ_{m|n}(N^{\otimes d},N^{\otimes d})$ is the two-sided ideal of $k[\mathfrak{S}_d]$
generated by $c_{\lambda_{m|n}}$, and $\sJ_{m|n}$ is the tensor ideal of $\sF_r$ 
generated by $c_{\lambda_{m|n}}$.
In particular
\begin{equation}\label{e:Jmnstrincl}
\sJ_{m+1|n+1} \subsetneqq \sJ_{m|n}.
\end{equation}
To show that the $\sJ_{m|n}$ are the only tensor ideals of $\sF_r$ other than $0$ or $\sF_r$,
we need the following lemma.

\begin{lem}\label{l:rectangle}
Let $V$ be a super $k$\nd vector space of dimension $m|n$ and $\lambda$ be a partition of $d>0$.
Suppose that the endomorphism $f$ of $V^{\otimes d}$ induced by the Young 
symmetriser $c_\lambda$ is $\ne 0$, but that each contraction of $f$ with respect to a factor 
$V$ of $V^{\otimes d}$ is $0$. 
Then $\lambda = \lambda_{m'|n'}$ for some $m'|n'$ with $m' - n' = m - n$.
\end{lem}

\begin{proof}
Since $f \ne 0$, the diagram $[\lambda]$ does not contain a box $(n+1,m+1)$.
Fix a basis $e^+_1, e^+_2, \dots , e^+_m$ of $V_0$ and $e^-_1, e^-_2, \dots , e^-_n$ of $V_1$.
The $e^+_r$ and $e^-_s$ define a basis of $|V|$ and hence of the tensor powers of $|V|$.
The boxes of $[\lambda]$ correspond to factors $V$ in $V^{\otimes d}$, and assignments to
each box of an $e^+_r$ or $e^-_s$ define basis elements of $|V|^{\otimes d}$.

Let $(1,l)$ be a box in the first row of $[\lambda]$ and $e$ be one of the $e^+_r$ or $e^-_s$. 
Assign as follows to each box in $[\lambda]$ a basis element $e^+_r$ or $e^-_s$.
Write $n_0$ for the lesser of $n$ or the number of rows of $[\lambda]$.
To $(1,l)$ assign $e$.
To $(i,j) \ne (1,l)$ with $i \le n_0$ assign $e^-_i$.
To $(i,j) \ne (1,l)$ with $i > n_0$ assign $e^+_j$,
which is possible because $j \le m$ for $i > n$. 
The diagram
\begin{equation*}
\begin{ytableau}
e^-_1 & e^-_1 & e^-_1 & e & e^-_1 & e^-_1 \\
e^-_2 & e^-_2 & e^-_2 & e^-_2 & e^-_2 \\
e^-_3 & e^-_3 & e^-_3 & e^-_3 \\
e^+_1 & e^+_2 & e^+_3 & e^+_4 \\
e^+_1 & e^+_2
\end{ytableau}
\end{equation*} 
is an example with $l = 4$ and $n_0 = 3$.
This assignment defines a basis element of $|V|^{\otimes d}$.
If we write $f_{l;e}$ for the diagonal matrix entry of 
$|f|:|V|^{\otimes d} \to |V|^{\otimes d}$
corresponding to this basis element, and $f'$ for the contraction of $f$ with respect to the factor 
$V$ of $V^{\otimes d}$ 
corresponding to the box $(1,l)$, then
\begin{equation}\label{e:contractentry}
\sum_{r = 1}^m f_{l;e^+_r} - \sum_{s = 1}^n f_{l;e^-_s}
\end{equation}
is a diagonal matrix entry of $|f'|$.
It is thus enough to show that if \eqref{e:contractentry} is $0$ for every $l$ then $[\lambda]$ is rectangular
with the difference of the number of its rows and its columns $m - n$.

Write $\lambda = (\lambda_1, \dots ,\lambda_p)$ and $\lambda^t = (\lambda^t{}\!_1, \dots ,\lambda^t{}\!_q)$.
If $n = 0$ then \eqref{e:contractentry} is
\begin{equation*}
f_{l;e^+_l} = \lambda^t{}\!_1 ! \lambda^t{}\!_2 ! \dots \lambda^t{}\!_q ! \ne 0.
\end{equation*}
We may thus suppose that $n > 0$.
Then \eqref{e:contractentry} is
\begin{multline*}
C(-1 -\frac{n-\min\{n_0,\lambda^t{}\!_l\}}{\lambda_1} + 
\frac{\max\{1,\lambda^t{}\!_l -n_0 +1\} - \delta_l}{\lambda_1} 
+ \frac{m-1 + \delta_l}{\lambda_1}) = \\
= \frac{C}{\lambda_1}(m - n - \lambda_1 + \lambda^t{}\!_l)
\end{multline*}
with $C = \lambda_1! \dots \lambda_p!(\lambda^t{}\!_l - n_0)! \dots (\lambda^t{}\!_{q_0} - n_0)! \ne 0$
where $q_0$ is the largest $i$ with $\lambda^t{}\!_i \ge n_0$, 
and with $\delta_l = 0$ when $l \le m$ and $\delta_l = 1$ when $l > m$.
Here the first term on the left corresponds to $f_{l;e^-_1}$, the second to $f_{l;e^-_s}$ for 
$\min\{n_0,\lambda^t{}\!_l\} < s \le n$, the third to $f_{l;e^+_l}$ when $l \le n$, the fourth to
$f_{l;e^+_r}$ for $r \ne l$, and the $f_{l;e^-_s}$ for $1 < s \le \min\{n_0,\lambda^t{}\!_l\}$ are $0$.
Since $\lambda_1$ is the number of columns and $\lambda^t{}\!_l$ is the length of the $l$th column of
$[\lambda]$, the result follows.
\end{proof}

\begin{lem}\label{l:idealJmn}
Let $r$ be an integer.
Then every tensor ideal of $\sF_r$ other than $0$ or $\sF_r$ is of the form $\sJ_{m|n}$ 
for a unique $m|n$ with $m-n = r$.
\end{lem}

\begin{proof}
The uniqueness is clear from \eqref{e:Jmnstrincl}.
Let $\sJ$ be a tensor ideal of $\sF_r$ other than $0$ or $\sF_r$.
Since $\sJ \ne 0$, there exists a $d>0$ and a partition $\lambda$ of $d$ such that 
$c_\lambda$ in  $\End(N^{\otimes d})$ lies in $\sJ$.
Hence $c_{\lambda_{m|n}}$ lies in $\sJ$ for $m,n$ sufficiently large, so that
\begin{equation}\label{e:Jmnincl}
\sJ_{m|n} \subset \sJ
\end{equation}
for some $m,n$ with $m-n =r$.

Let $(m_0,n_0)$ be the pair with $m_0 \ge 0$, $n_0 \ge 0$, $m_0-n_0=r$ and
$\sJ_{m_0|n_0}$ largest such that \eqref{e:Jmnincl} holds with $(m,n) = (m_0,n_0)$.
We show by induction on $d$ that
\begin{equation}\label{e:Jmnequ}
\sJ_{m_0|n_0}(N^{\otimes d},N^{\otimes d}) = \sJ(N^{\otimes d},N^{\otimes d})
\end{equation}
for each $d \ge 0$ so that $\sJ_{m_0|n_0} = \sJ$.
Since $\sJ \ne \sF_r$, \eqref{e:Jmnequ} holds for $d = 0$.
Suppose that \eqref{e:Jmnequ} holds for $d= d_0$.
Let $\lambda$ be a partition of $d_0 + 1$ with $c_\lambda \in \sJ$.
Suppose that $c_\lambda \notin \sJ_{m_0|n_0}$.
If $V$ is a super $k$\nd vector space of dimension $m_0|n_0$ there is by definition
of $\sJ_{m_0|n_0}$ a tensor functor with kernel $\sJ_{m_0|n_0}$
from $\sF_r$ to the category of super $k$\nd vector spaces which sends $N$ to $V$.
It sends $c_\lambda$ to the endomorphism $f$ of $V^{\otimes (d_0+1)}$
induced by $c_\lambda$.
Since $c_\lambda \notin \sJ_{m_0|n_0}$ we have $f \ne 0$.
By the induction hypothesis each contraction of $c_\lambda$ lies in $\sJ_{m_0|n_0}$,
and hence each contraction of $f$ is $0$.
Thus by Lemma~\ref{l:rectangle} $\lambda = \lambda_{m|n}$ for some
$m$ and $n$ with $m-n=r$.
The tensor ideal $\sJ_{m|n}$ of $\sF_r$ generated by $c_\lambda$ 
is then contained in $\sJ$ and hence by definition of $(m_0,n_0)$ in $\sJ_{m_0|n_0}$,
contradicting the assumption that $c_\lambda \notin \sJ_{m_0|n_0}$.
Thus $c_\lambda \in \sJ$ implies $c_\lambda \in \sJ_{m_0|n_0}$ for every partition
$\lambda$ of $d_0+1$, so that \eqref{e:Jmnequ} holds for $d = d_0 + 1$.
\end{proof}

By Lemma~\ref{l:abfracclose} we have $\kappa(\Mod(k)) = k$.
Hence $\kappa(\sF_r/\sJ_{m|n}) = k$ for every integer $r$ and $m|n$ with $m - n = r$,
because there is a faithful $k$\nd tensor functor from $\sF_r/\sJ_{m|n}$ to $\Mod(k)$.
By \eqref{e:Jmnstrincl} and Lemma~\ref{l:idealJmn}, the intersection of the 
$\sJ_{m|n} \subset \sF_r$ is $0$.
It thus follows from Lemma~\ref{l:idealJmn} that $\kappa(\sF_r/\sJ) = k$ for every
tensor ideal $\sJ \ne \sF_r$ of $\sF_r$.

Let $\br = (r_\gamma)_{\gamma \in \Gamma}$ be a family of elements of $k$.
We denote by $\sF_{\br}$ the free $k$\nd tensor category on a family 
$(N_\gamma)_{\gamma \in \Gamma}$ of dualisable
objects with $N_\gamma$ of rank $r_\gamma$.
If $(M_\gamma)_{\gamma \in \Gamma}$ is a family dualisable objects with 
$M_\gamma$ of rank $r_\gamma$ in a $k$\nd tensor category $\sC$,
there is then a $k$\nd tensor functor from $\sF_{\br}$ to $\sC$ that sends 
$N_\gamma$ to $M_\gamma$ for each $\gamma$,
and if $T_1$ and $T_2$ are tensor functors
from $\sF_{\br}$ to $\sC$ and $(\theta_\gamma)_{\gamma \in \Gamma}$ is a family
with $\theta_\gamma$ an isomorphism from 
$T_1(N_\gamma)$ to $T_2(N_\gamma)$, 
there is a unique tensor isomorphism $\varphi$ from $T_1$ to $T_2$ with 
$\varphi_{N_\gamma} = \theta_\gamma$.

Any $\Gamma' \subset \Gamma$ defines a $k$\nd tensor functor $\sF_{\br'} \to \sF_{\br}$
with $\br' = (r_\gamma)_{\gamma \in \Gamma'}$ which sends $N_\gamma$ to $N_\gamma$
for $\gamma \in \Gamma'$.
The $k$\nd tensor functors $\sF_{\br'} \to \sF_{\br}$ and $\sF_{\br''} \to \sF_{\br}$
given by a decomposition $\Gamma = \Gamma' \amalg \Gamma''$ define as above
a $k$\nd tensor functor from $\sF_{\br'} \otimes_k \sF_{\br''}$ to $\sF_{\br}$,
which by the universal properties is an equivalence.
Since $\End_{\sF_{\br''}}(\I) = k$, it follows that $\sF_{\br'} \to \sF_{\br}$ 
is fully faithful.
Thus we may identify $\sF_{\br'}$ with the strictly full rigid $k$\nd tensor subcategory
of $\sF_{\br}$ generated by the $N_\gamma$ for $\gamma \in \Gamma'$,
and $\sF_{\br}$ is then the filtered union of the $\sF_{\br'}$ for 
$\Gamma' \subset \Gamma$ finite.
It also follows that if $\Gamma = \{1,2, \dots, t\}$ is finite, then the $k$\nd tensor functor
\begin{equation}\label{e:freetens}
\sF_{r_1} \otimes_k \dots \otimes_k \sF_{r_t} \to \sF_{\br}
\end{equation}
defined by the embeddings $\sF_{r_i} \to \sF_{\br}$ is an equivalence.

Suppose that $\br = (r_\gamma)_{\gamma \in \Gamma}$ is a family of integers.
Let $\bm|\bn = (m_\gamma|n_\gamma)_{\gamma \in \Gamma}$ be a family of pairs 
of integers $\ge 0$ with $\bm - \bn = \br$.
There exists a tensor functor from $\sF_{\br}$ to the category 
of super $k$\nd vector spaces, unique up to tensor isomorphism, that sends 
$N_\gamma$ to a super $k$\nd vector space of dimension $m_\gamma|n_\gamma$.
Its kernel is a tensor ideal 
\begin{equation*}
\sJ_{\bm|\bn}
\end{equation*}
of $\sF_{\br}$.

A tensor ideal $\sJ$ in a tensor category $\sC$ will be called \emph{prime} if $\sC/\sJ$
is integral.

\begin{lem}\label{l:prime}
Let $\br = (r_\gamma)_{\gamma \in \Gamma}$ be a family of integers and $\sJ$ be 
a prime tensor ideal of $\sF_{\br}$.
For each $\gamma \in \Gamma$, denote by $\sJ_\gamma$ the restriction of $\sJ$ to
the full tensor subcategory $\sF_{r_\gamma}$ of $\sF_{\br}$ associated to $\gamma$.
\begin{enumerate}
\item\label{i:primegen}
$\sJ$ is the tensor ideal of $\sF_{\br}$ generated by the $\sJ_\gamma$ for 
$\gamma \in \Gamma$.
\item\label{i:primetens}
If $\Gamma = \{1,2, \dots, t\}$ is finite, then the $k$\nd tensor functor
\begin{equation*}
\sF_{r_1}/\sJ_1 \otimes_k \dots \otimes_k \sF_{r_t}/\sJ_t
\to \sF_{\br}/\sJ
\end{equation*}
induced by the embeddings $\sF_{r_i} \to \sF_{\br}$ is an equivalence.
\item\label{i:primemn}
If $\sJ_\gamma \ne 0$ for each $\gamma \in \Gamma$, then $\sJ = \sJ_{\bm|\bn}$
for a unique $\bm|\bn$ with $\bm - \bn = \br$.
\end{enumerate}
\end{lem}

\begin{proof}
\ref{i:primetens}
The fullness and essential surjectivity follow from those of \eqref{e:freetens}.
Since $\kappa(\sF_{r_i}/\sJ_i) = k$ and $\sF_{r_i}/\sJ_i \to \sF_{r_i}/\sJ$ is faithful for each $i$,
the faithfulness follows inductively from Lemma~\ref{l:extfaith}.

\ref{i:primegen}
Since $\sF_{\br}$ is the filtered union of its full tensor subcategories
$\sF_{\br'}$ associated to finite subsets $\Gamma'$ of $\Gamma$,
and since the restriction 
of $\sJ$ to each $\sF_{\br'}$ is a prime tensor ideal of $\sF_{\br'}$,
we may suppose that $\Gamma = \{1,2, \dots, t\}$ is finite.
Let $\sI \subset \sJ$ be a tensor ideal of $\sF_{\br}$ containing each $\sJ_i$.
By the fullness and essential surjectivity of \eqref{e:freetens}, the equivalence of 
\ref{i:primetens} factors through an equivalence with target $\sF_{\br}/\sI$.
Thus $\sI = \sJ$.

\ref{i:primemn}
By Lemma~\ref{l:idealJmn}, for each $\gamma$ we have $\sJ_\gamma = \sJ_{m_\gamma|n_\gamma}$ 
for a unique $m_\gamma|n_\gamma$ with $m_\gamma - n_\gamma = r_\gamma$.
Since the restriction to $\sF_{r_\gamma}$ of $\sJ_{\bm'|\bn'}$
with $\bm'|\bn' = (m'{}\!_\gamma|n'{}\!_\gamma)_{\gamma \in \Gamma}$ is 
$\sJ_{m'{}\!_\gamma|n'{}\!_\gamma}$, the required result thus follows from \ref{i:primegen}.
\end{proof}

Let $\bm|\bn = (m_\gamma|n_\gamma)_{\gamma \in \Gamma}$ be a family of pairs 
of integers $\ge 0$.
We write
\begin{equation*}
\sF_{\bm|\bn} = \sF_{\bm - \bn}/\sJ_{\bm|\bn}.
\end{equation*}
There exists a $k$\nd tensor functor 
\begin{equation}\label{e:freeGL}
\sF_{\bm|\bn} \to \Mod_{\GL_{\bm|\bn},\varepsilon}(k),
\end{equation}
unique up to tensor isomorphism, which sends $N_\gamma$ to the standard 
representation $E_\gamma$ of the factor $\GL_{m_\gamma|n_\gamma}$ at $\gamma$.
Indeed composing with the forgetful $k$\nd tensor functor from 
$\Mod_{\GL_{\bm|\bn},\varepsilon}(k)$ to super $k$\nd vector spaces shows that the
$k$\nd tensor functor from $\sF_{\bm - \bn}$ to $\Mod_{\GL_{\bm|\bn},\varepsilon}(k)$
that sends $N_\gamma$ to $E_\gamma$ has kernel $\sJ_{\bm|\bn}$.

\begin{lem}\label{l:freeGLff}
The $k$\nd tensor functor \eqref{e:freeGL} is fully faithful.
\end{lem}

\begin{proof}
Since the $k$\nd tensor functor $T:\sF_{\bm - \bn} \to \Mod_{\GL_{\bm|\bn},\varepsilon}(k)$
with $T(N_\gamma) = E_\gamma$ has kernel $\sJ_{\bm|\bn}$, 
it need only be shown that $T$ is full.
Writing $\sF_{\bm - \bn}$ as the filtered union of its full
$k$\nd tensor subcategories $\sF_{\bm' - \bn'}$ as $\bm'|\bn'$ runs over the
finite subfamilies of $\bm|\bn$, we may assume that $\Gamma$ is finite.
By the equivalence \eqref{e:freetens} and Lemma~\ref{l:Mhomtens}, we may further
assume that the family $\bm|\bn$ reduces to a single member $m|n$.
After dualising, it is then enough to show that 
$\sF_{m-n}  \to \Mod_{\GL_{m|n},\varepsilon}(k)$ sending $N$ to $E$ is surjective
on hom spaces $\Hom(N^{\otimes r},N^{\otimes s})$.
When $r \ne s$ this is clear because $\Hom(E^{\otimes r},E^{\otimes s})$ is $0$,
and when $r = s$ it follows from Lemma~\ref{l:standend}.  
\end{proof}

\begin{lem}\label{l:Fmnfracclose}
The $k$\nd tensor category $\sF_{\bm|\bn}$ is fractionally closed.
\end{lem}

\begin{proof}
The $k$\nd tensor category $\Mod_{\GL_{\bm|\bn},\varepsilon}(k)$ is integral,
and by Lemma~\ref{l:abfracclose} it is fractionally closed.
Since \eqref{e:freeGL} is fully faithful by Lemma~\ref{l:freeGLff}, 
it is thus regular and \eqref{e:fracclose}
is an isomorphism for $\sC = \sF_{\bm|\bn}$ and every $C,C'$ and regular $f$ in $\sC$.
\end{proof}

\begin{lem}\label{l:rankinteger}
Let $\sC$ be a $k$\nd tensor category with $\End_{\sC}(\I)$ reduced and indecomposable and
$M$ be a dualisable object in $\sC$.
Suppose that for some $d$ there 
exists an element $\ne 0$ of $k[\mathfrak{S}_d]$ which acts as $0$ on $M^{\otimes d}$.
Then the rank of $M$ is an integer.
\end{lem}

\begin{proof}
Since $k[\mathfrak{S}_d] \to \End_{\sC}(M^{\otimes d})$ has non-zero kernel,
there exists a partition $\lambda$ of $d$ such that the primitive idempotent $e_\lambda$
associated to $\lambda$ acts as $0$ on $M^{\otimes d}$.
Increasing $d$ if necessary, we may suppose that there is in integer $m_0$ such that
$d = m_0{}\!^2$ and $\lambda = (m_0m_0 \dots m_0)$ is square with 
$m_0$ rows and columns.

Let $\alpha$ be an element of $k[\mathfrak{S}_d]$.
Reducing to the case where $\alpha$ is an element of $\mathfrak{S}_d$ and successively contracting
shows that there is a polynomial $p_\alpha(t)$ in $k[t]$ with the following property:
If $N$ is a dualisable object of rank $r$ in a $k$\nd tensor category and $f$ is the endomorphism of 
$N^{\otimes d}$ induced by $\alpha$, then 
\begin{equation*}
\tr(f) = p_\alpha(r).
\end{equation*}
Further $p_\alpha(t)$ is unique: taking for $N$ a $k$\nd vector space of dimension $m$ 
shows that the value of $p_\alpha(t)$ is determined for $t$ any integer $m \ge 0$.

Let $V$ be a $k$\nd vector space of dimension $m$.
Then the trace $\tau(m)$ of the endomorphism of $V^{\otimes d}$ defined by $e_\lambda$
is the dimension of the $k$\nd vector space $S_\lambda V$ obtained by applying the Schur functor 
$S_\lambda$ to $V$.  
Thus \cite[Theorem~6.3(1)]{FulHar} if $m \ge 2m_0$
\begin{equation*}
\tau(m) = \prod_{1 \le i < j \le m}\frac{\lambda_i - \lambda_j + j - i}{j - i} 
= \prod_{l > 0} \left( \frac{m_0+l}{l} \right)^{\rho(l)} = \prod_{l > 0} l^{\rho(l-m_0) - \rho(l)}
\end{equation*}
where $\rho(l) = \min \{l,m-l,m_0\}$ when $0 < l < m$, and $\rho(l) = 0$ otherwise.
Hence
\begin{equation*}
p_{e_\lambda}(t) = \prod_{-m_0 < i < m_0}(m_0-i)^{|i|-m_0}(t-i)^{m_0-|i|},
\end{equation*}
because then $p_{e_\lambda}(m) = \tau(m)$ for $m \ge 3m_0$.
If $M$ has rank $r$ then $p_{e_\lambda}(r) = 0$, because $e_\lambda$ acts as $0$ on 
$M^{\otimes d}$.
Since $\End(\I)$ is reduced and indecomposable, the result follows. 
\end{proof}

\begin{prop}\label{p:FmnCesssurj}
Let $\sC$ be an essentially small, integral, rigid $k$\nd tensor category.
Suppose that for each object $M$ of $\sC$ there exists a $d$ such that
some element $\ne 0$ of $k[\mathfrak{S}_d]$ acts as $0$ on $M^{\otimes d}$.
Then for some family $\bm|\bn$ of pairs of integers $\ge 0$ there exists a
faithful essentially surjective $k$\nd tensor functor from $\sF_{\bm|\bn}$
to $\sC$.
\end{prop}

\begin{proof}
Let $(M_\gamma)_{\gamma \in \Gamma}$ be a (small) family of objects of $\sC$ with
every object of $\sC$ isomorphic to some $M_\gamma$.
By Lemma~\ref{l:rankinteger}, the rank $r_\gamma$ of $M_\gamma$ is an integer.
If $\br = (r_\gamma)_{\gamma \in \Gamma}$
there exists an essentially surjective $k$\nd tensor functor
from $\sF_{\br}$ to $\sC$ which sends $N_\gamma$ to $M_\gamma$ for each $\gamma$.
Its kernel $\sJ$ is prime, and
by hypothesis the restriction of $\sJ$ to the full tensor subcategory 
$\sF_{r_\gamma}$ of $\sF_{\br}$ associated to $\gamma$ is $\ne 0$ for each $\gamma$.
It thus follows from Lemma~\ref{l:prime}\ref{i:primemn} that $\sJ = \sJ_{\bm|\bn}$
for some $\bm|\bn$ with $\bm - \bn = \br$.
\end{proof}

\section{Functor categories}\label{s:fun}

In this section we describe how the usual additive Yoneda embedding of an additive category
extends to embedding of a tensor category into an abelian tensor category.
For tensor categories with one object, identified with commutative rings,
this is the embedding into the tensor category of modules over the ring.
The main result is Theorem~\ref{t:Fhatequiv}, which is crucial for the geometric description 
in Section~\ref{s:mod} of appropriate functor categories modulo torsion.

Let $\sC$ be an essentially small additive category.
We denote by
\begin{equation*}
\widehat{\sC}
\end{equation*} 
the additive category of additive functors from $\sC^\mathrm{op}$ to $\mathrm{Ab}$. 
The category $\widehat{\sC}$ is abelian, with (small) limits and colimits, computed argumentwise.
Finite limits in $\widehat{\sC}$ commute with filtered colimits.
For every object $A$ in $\sC$ we have an object
\begin{equation*}
h_A = \sC(-,A)
\end{equation*}
in $\widehat{\sC}$, and evaluation at $1_A$ gives the Yoneda isomorphism
\begin{equation}\label{e:Yonedaiso}
\widehat{\sC}(h_A,M) \iso M(A),
\end{equation}
natural in $A$ and $M$.
We have the fully faithful Yoneda embedding
\begin{equation}\label{e:Yoneda}
h_-:\sC \to \widehat{\sC}
\end{equation}
of additive categories, which preserves limits.

The $h_A$ for $A$ in a skeleton of $\sC$ form a small set of generators for $\widehat{\sC}$.
It follows that every object of $\widehat{\sC}$ is the quotient of a small coproduct of objects $h_A$,
and hence that every object $M$ of $\widehat{\sC}$ is a cokernel
\begin{equation}\label{e:pres}
\coprod_{\delta \in \Delta} h_{B_\delta} \to \coprod_{\gamma \in \Gamma} h_{A_\gamma} \to M \to 0
\end{equation}
for some small families $(A_\gamma)_{\gamma \in \Gamma}$ and $(B_\delta)_{\delta \in \Delta}$. 

When $\sC$ has direct sums, the embedding $h_-:\sC \to \widehat{\sC}$ is dense,
i.e.\ every $M$ in $\sC$ can be expressed as a canonical colimit $\colim_{(A,a),a \in M(A)}h_A$
over the comma category $h_-/M$.
Indeed if we write $M'$ for the colimit, then given $(A_i,a_i)$ and $f_i:B \to A_i$ in $h_{A_i}(B)$ 
such that the $M(f_i)(a_i)$ have sum $0$ in $M(B)$, the images of the $f_i$ in $M'(B)$ also 
have sum $0$, because the $f_i$ define a morphism from $(B,0)$ to $(\bigoplus_iA_i,\sum_ia_i)$
in the comma category which sends $1_B$ in $h_B(B)$ to $(f_i)$ in
$h_{\bigoplus A_i}(B)$.

When $\sC$ has finite colimits, an object of $\widehat{\sC}$ is 
a left exact functor $\sC^\mathrm{op} \to \mathrm{Ab}$ if and only if it is a filtered colimit 
of objects $h_A$.
Indeed the comma category $h_-/M$ is essentially small and for a left exact $M$ it is filtered.

Every $h_A$ is projective in $\widehat{\sC}$.
Thus by \eqref{e:pres} $\widehat{\sC}$ has enough projectives, 
and an object in $\widehat{\sC}$ is projective
if and only if it is a direct summand of a coproduct of objects $h_A$.

An object $M$ in a category with filtered colimits will be said to be \emph{of finite presentation}
(resp.\ \emph{of finite type}) if $\Hom(M,-)$ preserves filtered colimits (resp.\ filtered colimits
with coprojections monomorphisms).
In an abelian category, an object $M$ is projective of finite type if and only if it is projective
of finite presentation if and only if $\Hom(M,-)$ is cocontinuous.

It follows from \eqref{e:pres} that an object $M$ in $\widehat{\sC}$ is of finite presentation
if and only if it is a cokernel
\begin{equation}\label{e:finpres}
\bigoplus_{j=1}^n h_{B_j} \to \bigoplus_{i=1}^m h_{A_i} \to M \to 0
\end{equation}
for some finite families $(A_i)$ and $(B_j)$ of objects of $\sC$, 
and that $M$ is of finite type if and only if it is a quotient 
\begin{equation}\label{e:fintype}
\bigoplus_{i=1}^m h_{A_i} \to M \to 0
\end{equation}
for some finite family $(A_i)$ of objects of $\sC$.
By \eqref{e:finpres},
every object of $\widehat{\sC}$ is a filtered colimit of objects of finite presentation,
and by \eqref{e:fintype} every object of $\widehat{\sC}$ is the filtered
colimit of its subobjects of finite type.

By \eqref{e:fintype}, an object of $\widehat{\sC}$ is projective of finite type if
and only if it is a direct summand of a finite direct sum of objects $h_A$.
The full subcategory of $\widehat{\sC}$ consisting of such objects is thus the pseudo-abelian hull of $\sC$.

The coproduct of the  $ h_A$ for $A$ in a skeleton of $\sC$ is a generator for $\widehat{\sC}$.
The category $\widehat{\sC}$ also has a cogenerator: if $L$ is a generator for 
$\widehat{\sC^\mathrm{op}}$ and if for $M$ an object in $\widehat{\sC}$ we write $M^\dagger$ for the object
$\Hom_{\Z}(M(-),\Q/\Z)$ in $\widehat{\sC^\mathrm{op}}$,
then $M^\dagger$ is a quotient of a small coproduct of copies of $L$,
so that $M^{\dagger\dagger}$ and hence $M$ is a subobject of a small product
of  copies of $L^\dagger$. 
Since $\widehat{\sC}$ is complete and well-powered, it follows that any continuous functor from 
$\widehat{\sC}$ to $\mathrm{Ab}$ is representable.

If $\sC$ has a structure of tensor category,
we define as follows a structure of tensor category on $\widehat{\sC}$.
Given objects $L$, $M$ and $N$ in $\widehat{\sC}$, call a family of biadditive maps
\begin{equation*}
M(A) \times N(B) \to L(A \otimes B)
\end{equation*}
which is natural in the objects $A$ and $B$ in $\sC$ a \emph{bimorphism} from $(M,N)$ to $L$. 
The tensor product $M \otimes N$ is then defined as the 
target of the universal bimorphism from $(M,N)$,
which exists because the relevant functor is representable.
If we similarly define trimorphisms, then both $(M \otimes N) \otimes P$ and $M \otimes (N \otimes P)$
are the target of a universal trimorphism from $(M,N,P)$, and the associativity constraint is the 
isomorphism between them defined by the universal property. 
The symmetries of $\widehat{\sC}$ are defined similarly, and the unit is $h_{\I}$. 
The required compatibilities hold by the universal property of the tensor product.
We may assume that the tensor product is chosen so that the unit $\I = h_{\I}$ is strict.
The tensor product of $\widehat{\sC}$ is cocontinuous.

The bimorphism from $(h_A,h_B)$ to $h_{A \otimes B}$ that sends $(f,g)$ in $h_A(A') \times h_B(B')$
to $f \otimes g$ in $h_{A \otimes B}(A' \otimes B')$ defines a structure
a structure
\begin{equation}\label{e:tensfun}
h_A \otimes h_B \iso h_{A \otimes B}
\end{equation}
of tensor functor on the embedding \eqref{e:Yoneda}.

If $\sC$ has finite colimits, then the full subcategory of $\widehat{\sC}$ consisting
of the left exact functors is a tensor subcategory.

The image of $(a,b)$ under the component
\begin{equation*}
M(A) \times N(B) \to (M \otimes N)(A \otimes B)
\end{equation*}
of the universal bimorphism will be written $a \otimes b$.
Modulo \eqref{e:Yonedaiso} and \eqref{e:tensfun}, 
$a \otimes b$ is the tensor product of the morphisms $a$ and $b$ in $\widehat{\sC}$.

Let $\sC$ be an essentially small additive category, $\sD$ be cocomplete additive category,
and $T:\sC \to \sD$ be an additive functor.
Then the additive left Kan extension
\begin{equation*}
T^*:\widehat{\sC} \to \sD
\end{equation*}
of $T$ along $\sC \to \widehat{\sC}$ exists. 
It is given by the coend formula
\begin{equation}\label{e:addKanT}
T^*(M) = \int^{A \in \sC} M(A) \otimes_{\Z} T(A),
\end{equation}
and it is cocontinuous and is preserved by any cocontinuous functor from $\sD$
to a cocomplete additive category. 
Since $h_-$ is fully faithful, the universal natural transformation
from $T$ to $T^*h_-$ is a natural isomorphism
\begin{equation}\label{e:Th}
T \iso T^*h_-.
\end{equation}
By cocontinuity of $h_-{}\!^*$ and \eqref{e:pres}, 
the canonical natural transformation from $h_-{}\!^*$ to $\Id_{\widehat{\sC}}$ 
is an isomorphism.
Thus $\Id_{\widehat{\sC}}$ is the additive left Kan extension
of $h_-$ along itself, with universal natural transformation the identity.
It follows that composition with $h_-$ defines an equivalence from cocontinuous
functors $\widehat{\sC} \to \sD$ to additive functors $\sC \to \sD$, with quasi-inverse
given by additive left Kan extension.

The functor $T^*$ has a right adjoint 
\begin{equation*}
T_*:\sD \to \widehat{\sC}
\end{equation*}
with $T_*(N):\sC^{\mathrm{op}} \to \mathrm{Ab}$ given by
\begin{equation}\label{e:Tstardef}
T_*(N) = \sD(T(-),N)
\end{equation}
where the unit $\Id_{\widehat{\sC}} \to T_*T^*$ corresponds under the universal property of 
the additive left Kan extension $\Id_{\widehat{\sC}}$ of $h_-$ along itself to the composite
\begin{equation}\label{e:unitdef}
h_- \to T_*T \iso T_*T^*h_-
\end{equation}
in which the isomorphism is $T_*$ applied to \eqref{e:Th} and the first arrow has component 
\begin{equation*}
h_A \to T_*T(A) = \sD(T(-),T(A))
\end{equation*}
at $A$ defined by $1_{T(A)}$.
That $\Id_{\widehat{\sC}} \to T_*T^*$ so defined induces an isomorphism
\begin{equation*}
\sD(T^*(M),N) \iso \widehat{\sC}(M,T_*(N))
\end{equation*}
for $M$ in $\widehat{\sC}$ and $N$ in $\sD$ can be seen by reducing by cocontinuity
and \eqref{e:pres} to the case $M = h_A$, where we have isomorphisms
\begin{equation*}
\sD(T^*(h_A),N) \iso \sD(T(A),N) \iso \widehat{\sC}(h_A,T_*(N))
\end{equation*}
induced by  \eqref{e:unitdef}.

Let $\varphi:T \to T'$ be a natural transformation of additive functors $\sC \to \sD$.
Denote by
\begin{equation*}
\varphi^*:T^* \to T'{}^*
\end{equation*}
the unique natural transformation 
such that $\varphi$ and $\varphi^*h_-$ are compatible with \eqref{e:Th} and the corresponding natural
isomorphism for $T'$.
The natural transformation
\begin{equation*}
\varphi_*:T'{}\!_* \to T_*
\end{equation*}
induced on right adjoints has component $\varphi_*{}_M$ at $M$ given by
\begin{equation*}
(\varphi_*{}_M)_A = \sD(\varphi_A,M):\sD(T'(A),M) \to \sD(T(A),M).
\end{equation*}
This follows from the diagram 
\begin{equation*}
\xymatrix{
h_A \ar [d] \ar[r] & T'{}\!_*T'(A) \ar[d]^{\varphi_*{}_{T'(A)}} \ar[r] & T'{}\!_*(M) \ar[d]^{\varphi_*{}_M} \\
T_*T(A) \ar[r]^{T_*(\varphi_A)} & T_*T'(A) \ar[r] & T_*(M)
}
\end{equation*} 
defined by a morphism $T'(A) \to M$, where the right square commutes by naturality of $\varphi_*$ and
the commutativity of the left square can be seen starting from the compatibility of $T_*\varphi^*$ and 
$\varphi_*T'{}^*$ with the units by evaluating at $h_A$ and using the isomorphisms of the form \eqref{e:Th}
and their compatibility with $\varphi$ and $\varphi^*h_-$.

Now let $\sC$ be an essentially small tensor category, $\sD$ be cocomplete tensor category,
and $T:\sC \to \sD$ be a tensor functor.
By cocontinuity of $T^*:\widehat{\sC} \to \sD$, it follows either directly from \eqref{e:addKanT}
or after first replacing $\sC$ by its additive hull from the density of 
$h_-:\sC \to \widehat{\sC}$ that $T^*$ has a unique structure of
tensor functor such that \eqref{e:Th} is a tensor isomorphism.
We may suppose after replacing $T^*$ if necessary by an isomorphic functor that the component of
\eqref{e:Th} at $\I$ is the identity, so that $T^*$ preserves the unit strictly.
There then exists a unique structure of lax tensor functor on $T_*:\sD \to \widehat{\sC}$ such that
the unit and counit for $T^*$ and $T_*$ are compatible with the lax tensor structures.
Explicitly, since $T^*$ and $T$ preserve $\I$ strictly, the unit
\begin{equation*}
h_{\I} \to T_*(\I) = \sD(T(-),\I)
\end{equation*}
of $T_*$ is $\eqref{e:unitdef}$ evaluated at $\I$, and hence is defined by $1_{\I} \in \sD(\I,\I)$.
The lax tensor structure of $T_*$ is given using \eqref{e:Tstardef} by the biadditive maps
\begin{equation*}
\sD(T(A),M) \times \sD(T(B),N) \to \sD(T(A \otimes B),M \otimes N)
\end{equation*}
which send $(a,b)$, modulo the isomorphism defining the tensor structure of $T$, to $a \otimes b$.
This can be seen from the diagram
\begin{equation*}
\xymatrix@R=2pc@C=3pc{
& h_A \otimes h_B \ar[dl] \ar[r]  &  T_*T(A) \otimes T_*T(B) 
\ar[dl]!<1.5em,0ex> \ar[d] \ar[r]^-{\scriptscriptstyle{T_*(a) \otimes T_*(b)}} & T_*(M) \otimes T_*(N) \ar[d] \\
h_{A \otimes B} \ar[r] & T_*T(A \otimes B) \ar[r]^-{\sim} & T_*(T(A) \otimes T(B)) 
\ar[r]^-{\scriptscriptstyle{T_*(a \otimes b)}} & T_*(M \otimes N) 
}
\end{equation*}
in which the parallelogram commutes because \eqref{e:unitdef} and hence the first arrow of \eqref{e:unitdef}
is compatible with the tensor  structures, the triangle commutes by definition of the tensor structure
of a composite functor, and the square commutes by naturality of the tensor structure of $T_*$.

Let $\varphi:T \to T'$ be a natural transformation of tensor functors $\sC \to \sD$ which is compatible with the
tensor structures.
Then by \eqref{e:pres}, the compatibility of the isomorphisms of the form \eqref{e:Th}
with $\varphi$ and $\varphi^*h_-$, and cocontinuity of $T^*$, the natural transformation $\varphi^*:T^* \to T'{}^*$
is compatible with the tensor structures.
It follows that $\varphi_*:T'{}\!_* \to T_*$ is compatible with the tensor structures.
It also follows that composition with $h_-$ defines an equivalence from
cocontinuous tensor functors $\widehat{\sC} \to \sD$ to tensor functors $\sC \to \sD$,
with quasi-inverse $T \mapsto T^*$.

Let $\sA$ and $\sA'$ be tensor categories, $H:\sA \to \sA'$ be a tensor functor,
and $H':\sA' \to \sA$ be a lax tensor functor right adjoint to $H$.
Given $M$ in $\sA$ and $M'$ in $\sA'$, we have a canonical morphism
\begin{equation}\label{e:proj}
H'(M') \otimes M \to H'(M' \otimes H(M))
\end{equation}
in $\sA$, natural in $M$ and $M'$, given by the composite
\begin{equation}\label{e:projdef}
H'(M') \otimes M \to H'(M') \otimes H'H(M) \to H'(M' \otimes H(M))
\end{equation}
with the first arrow defined using the unit and the second by the lax tensor structure of $H'$. 
It is adjoint to the morphism
\begin{equation}\label{e:projadj}
H(H'(M') \otimes M) \iso HH'(M') \otimes H(M) \to M' \otimes H(M)
\end{equation}
defined using the tensor structure of $H$ and the counit.
If $M$ is dualisable then \eqref{e:proj} is an isomorphism, as follows from the top row of 
the commutative diagram
\begin{equation*}
\xymatrix@C-0.6cm{
\sA(L,H'(M') \otimes M) \ar[d]^{\wr} \ar[r] & \sA'(H(L),HH'(M') \otimes H(M)) 
\ar[d] \ar[r] & \sA'(H(L),M' \otimes H(M)) \ar[d]^{\wr} \\
\sA(L \otimes M^{\vee},H'(M')) \ar[r] & \sA'(H(L) \otimes H(M)^{\vee},HH'(M')) 
\ar[r] & \sA'(H(L) \otimes H(M)^{\vee},M')
}
\end{equation*}
with arrows natural in $L$, where the bottom row, modulo the tensor structure of $H$, is the adjunction isomorphism.

Suppose $\sA$ is abelian with $\otimes$ right exact.
For any commutative algebra $R'$ in $\sA'$, the lax tensor structure of $H'$ defines on $H'(R')$ a 
structure of commutative algebra in $\sA$, and similarly for modules over $R'$ in $\sA'$.
In particular, since $\I$ has a unique structure of commutative algebra in $\sA'$ and every $M'$ in $\sA'$
has a unique structure of module over $\I$, we have
a commutative algebra $H'(\I)$ in $\sA$ and a canonical structure of $H'(\I)$\nd module on every $H'(M')$.
Thus $H'$ factors as a lax tensor functor
\begin{equation}\label{e:EilMoore}
\sA' \to \MOD_{\sA}(H'(\I))
\end{equation}
followed by the forgetful functor.
The morphism of $H'(\I)$\nd modules 
\begin{equation}\label{e:EMunit}
H'(\I) \otimes M \to H'H(M)
\end{equation}
corresponding to the unit $M \to H'H(M)$ in $\sA$ is given by taking $M' = \I$ in \eqref{e:proj},
and hence is an isomorphism for $M$ dualisable.
It is natural in $M$ and is compatible with the tensor structure of $H'(\I) \otimes -$ and
the lax tensor structure of $H'H$. 
The morphism
\begin{equation}\label{e:EMtens}
H'(M') \otimes_{H'(\I)} H'(N') \to H'(M' \otimes N')
\end{equation}
defining the lax tensor structure of \eqref{e:EilMoore} is an isomorphism for $N' = H(M)$ with $M$ dualisable, 
because the composite \eqref{e:projdef} defining \eqref{e:proj} then factors as an isomorphism from 
$H'(M') \otimes M$ to $H'(M') \otimes_{H'(\I)}  H'H(M)$ followed by \eqref{e:EMtens}.
Similarly, the homomorphism
\begin{equation}\label{e:EMhom}
\sA'(M',N') \to \Hom_{H'(\I)}(H'(M'),H'(N'))
\end{equation}
defined by \eqref{e:EilMoore} is an isomorphism for $M' = H(M)$ with $M$ dualisable, 
because the adjunction isomorphism from $\sA'(H(M),N')$ to $\sA(M,H'(N'))$ then factors as 
\eqref{e:EMhom} followed by an isomorphism induced by the unit.

Let $\sC$ and $\sC'$ be essentially small additive categories and $F:\sC \to \sC'$ be an additive functor.
Then we write
\begin{equation*}
\widehat{F}:\widehat{\sC} \to \widehat{\sC'}
\end{equation*}
for $T^*$ with $T:\sC \to \widehat{\sC'}$ the composite of $h_-:\sC' \to \widehat{\sC'}$ with $F$, and
\begin{equation*}
F_{\wedge}:\widehat{\sC'} \to \widehat{\sC}
\end{equation*}
for $T_*$.
In this case, \eqref{e:Tstardef} becomes
\begin{equation*}
F_{\wedge}(N) = NF^\mathrm{op}
\end{equation*}
for $N:\sC'{}^\mathrm{op} \to \mathrm{Ab}$.
Thus $F_\wedge$ is continuous and cocontinuous.
If $F$ is fully faithful, then $\widehat{F}$ is fully faithful, by \eqref{e:pres} and the fact that
$h_A$ for every $A$ in $\sC$ or $\sC'$ is projective of finite presentation. 

Similarly we define $\widehat{\varphi}$ and $\varphi_{\wedge}$ for a natural transformation $\varphi$.
Given also an essentially small additive category $\sC''$ and an additive functor $F':\sC' \to \sC''$, 
we have $(F'F)_{\wedge} = F_{\wedge}F'{}\!_{\wedge}$. 
Thus $\widehat{F'F}$ and $\widehat{F'}\widehat{F}$ are canonically isomorphic, 
with composites of three or more isomorphisms satisfying the usual compatibilities.

Suppose that $\sC$ and $\sC'$ have structures of tensor category and $F$ has a structure of tensor functor.
Then as above $\widehat{F}$ has a structure of tensor functor and $F_{\wedge}$ of lax tensor functor,
and $\widehat{\varphi}$ and $\varphi_{\wedge}$ are compatible with the tensor structures if 
the natural transformation $\varphi$ is.
Taking $\sA = \widehat{\sC}$, $\sA' = \widehat{\sC'}$,
and $H' = F_{\wedge}$ in \eqref{e:EilMoore} shows that $F_{\wedge}$ factors as a lax tensor functor
\begin{equation}\label{e:EilMoorehat}
\widehat{\sC'} \to \MOD_{\widehat{\sC}}(F_{\wedge}(\I))
\end{equation}
followed by the forgetful functor.
If $F'$ has a structure of tensor functor, the lax tensor functors $(F'F)_{\wedge}$ and 
$F_{\wedge}F'{}\!_{\wedge}$ coincide, and the canonical isomorphism from $\widehat{F'F}$ to 
$\widehat{F'}\widehat{F}$ is a tensor isomorphism.

\begin{lem}\label{l:EilMoore}
Let $\sC$ and $\sC'$ be essentially small tensor categories with $\sC$ rigid, and $F:\sC \to \sC'$
be a tensor functor.
\begin{enumerate}
\item\label{i:EilMooretens}
The composite of \eqref{e:EilMoorehat} with $\widehat{F}:\widehat{\sC} \to \widehat{\sC'}$ 
is tensor isomorphic to 
\begin{equation*}
F_{\wedge}(\I) \otimes -:\widehat{\sC} \to \MOD_{\widehat{\sC}}(F_{\wedge}(\I)).
\end{equation*}
\item\label{i:EilMooreequiv}
If $F$ is essentially surjective, then \eqref{e:EilMoorehat} is a tensor equivalence.
\end{enumerate}
\end{lem}

\begin{proof}
Write $\sA = \widehat{\sC}$, $\sA' = \widehat{\sC'}$, $H = \widehat{F}$ and $H' = F_{\wedge}$.
Then \eqref{e:EilMoore} is \eqref{e:EilMoorehat}, and \eqref{e:EMunit} is the component at 
$M$ of a natural transformation, compatible with the tensor structures,
from $F_{\wedge}(\I) \otimes -$ to the composite of \eqref{e:EilMoorehat} with $\widehat{F}$.
Further \eqref{e:EMunit} is an isomorphism for $M = h_A$ because $M$ is
then dualisable, and hence by \eqref{e:pres} and continuity of $\otimes$, $H$ and $H'$, for every $M$.  
This gives \ref{i:EilMooretens}.

Suppose that $F$ is essentially surjective. 
Then \eqref{e:EMhom} is an isomorphism for $M' = h_{A'}$ and every $N'$, because
$h_{A'}$ is isomorphic to some $H(h_A)$ and $h_A$ is dualisable.
Thus by \eqref{e:pres} and cocontinuity of $H'$, \eqref{e:EMhom} is an isomorphism for every $M'$ and $N'$,
so that \eqref{e:EilMoorehat} is fully faithful.
Now \eqref{e:EilMoorehat} is cocontinuous, because $F_{\wedge}$ is.
Since by \ref{i:EilMooretens} the essential image of \eqref{e:EilMoorehat}
contains the $F_{\wedge}(\I)$\nd modules $F_{\wedge}(\I) \otimes h_A$,
it thus contains every $F_{\wedge}(\I) \otimes M$, and hence every $F_{\wedge}(\I)$\nd module.
Finally \eqref{e:EMunit} is an isomorphism for $M = \I$,
and \eqref{e:EMtens} is an isomorphism for $N' = H(h_A)$, and hence for every $N'$.
This proves \ref{i:EilMooreequiv}.
\end{proof}

\begin{thm}\label{t:Fhatequiv}
Let $k$ be a field of characteristic $0$ and $\sC$ be as in 
Proposition~\textnormal{\ref{p:FmnCesssurj}}.
Then for some $\bm|\bn$ there exists a commutative algebra $R$ in $\widehat{\sF_{\bm|\bn}}$ with 
\begin{equation*}
R \otimes - :(\widehat{\sF_{\bm|\bn}})_\mathrm{rig} \to \Mod_{\widehat{\sF_{\bm|\bn}}}(R)
\end{equation*} 
faithful such that $\widehat{\sC}$ is $k$\nd tensor equivalent to $\MOD_{\widehat{\sF_{\bm|\bn}}}(R)$.
\end{thm}

\begin{proof}
Lemma~\ref{l:EilMoore} with $F$ the $k$\nd tensor functor 
$\sF_{\bm|\bn} \to \sC$ of Proposition~\ref{p:FmnCesssurj} gives a $k$\nd tensor equivalence
from $\widehat{\sC}$ to $\MOD_{\widehat{\sF_{\bm|\bn}}}(R)$ with $R = F_\wedge(\I)$
whose composite with $\widehat{F}$ is tensor isomorphic to $R \otimes -$.
Since $F$ is faithful, the $k$\nd tensor functor from
$(\widehat{\sF_{\bm|\bn}})_\mathrm{rig}$ to $(\widehat{\sC})_\mathrm{rig}$ induced by
$\widehat{F}$ on pseudo-abelian hulls is faithful.
\end{proof}

\section{Torsion}\label{s:tor}

This section deals with the notion of torsion in appropriate abelian tensor 
categories, which is fundamental for the construction of super Tannakian hulls.

Let $\sA$ be an abelian category.
Recall that a full subcategory $\sS$ of $\sA$ containing $0$ is said to be a 
Serre subcategory of $\sA$ if for any short exact sequence in $\sA$
\begin{equation}\label{e:sex}
0 \to M' \to M \to M'' \to 0,
\end{equation} 
$M$ lies in $\sS$ if and only if both $M'$ and $M''$ do.
Suppose that $\sA$ is well-powered,
and let $\sS$ be a Serre subcategory of $\sC$.
Then (e.g.\ \cite[III \S~1]{Gab58}) the quotient $\sA/\sS$ of $\sA$ by $\sS$ is the abelian category
with objects those of $\sA$, where
\begin{equation}\label{e:quotcolim}
(\sA/\sS)(M,N) = \colim_{M' \subset M, \, N' \subset N, \; M/M', N' \in \sS}\sA(M',N/N')
\end{equation}
with the colimit over an essentially small filtered category, and an evident composition.
We have an exact functor 
\begin{equation}\label{e:quotproj}
\sA \to \sA/\sS
\end{equation} 
which is the identity on objects
and on the hom-group $\sA(M,N)$ is the coprojection at $M' = M$, $N' = 0$ of \eqref{e:quotcolim}.
If $\sA$ has coproducts and coproducts of monomorphisms in $\sA$ are monomorphisms
(i.e.\ if $\sA$ is an (AB4) category in the sense of Grothendieck)
and $\sS$ is closed under formation of coproducts in $\sA$, then $\sA/\sS$ has colimits
and \eqref{e:quotproj} preserves colimits:
it is enough to check that \eqref{e:quotproj} preserves coproducts, 
and if we take $M = \coprod_{\alpha}M_{\alpha}$ in \eqref{e:quotcolim}, then every
$M' \subset M$ with $M/M'$ in $\sS$ contains $\coprod_{\alpha}M'{}\!_{\alpha}$ with $M'{}\!_{\alpha}$
the kernel of $M_{\alpha} \to M/M'$ and hence $M_{\alpha}/M'{}\!_{\alpha}$ in $\sS$.

An object of $\sA$ has image $0$ in $\sA/\sS$ if and only if it lies in $\sS$.
A morphism in $\sA$ is an $\sS$\nd isomorphism, i.e.\ has image in $\sA/\sS$ an isomorphism,
if and only if both its kernel and cokernel lie in $\sS$.
A morphism in $\sS$ is $\sS$\nd trivial, i.e.\ has image $0$ in $\sA/\sS$, if and only if
its image lies in $\sS$ if and only if both the embedding of its kernel and the projection onto
its cokernel are $\sS$\nd isomorphisms.

Any exact functor from $\sA$ to an abelian category $\sA'$
which sends every object of $\sS$ to $0$
factors uniquely through \eqref{e:quotproj}, and $\sA/\sS \to \sA'$ is then exact.
Further $\sA/\sS \to \sA'$ is faithful if and only if the objects of $\sS$ are the 
only ones in $\sA$ sent to $0$ in $\sA'$.

Let $\sA'$ be an abelian category and $T:\sA \to \sA'$ be an additive functor which
sends every epimorphism with kernel in $\sS$ and every monomorphism
with cokernel in $\sS$ to an isomorphism, or equivalently which sends every $\sS$\nd isomorphism to an isomorphism.
Then $T$ factors uniquely through \eqref{e:quotproj}:
the factorisation
\begin{equation*}
\overline{T}:\sA/\sS \to \sA'
\end{equation*}
of $T$ coincides with $T$ on objects, and if $\overline{f}$ in
$(\sA/\sS)(M,N)$ is the image of $f$ in $\sA(M',N/N')$ in the colimit \eqref{e:quotcolim},
with $e:M' \to M$ the embedding and $p:N \to N/N'$ the projection, then
\begin{equation}\label{e:Tbardef}
\overline{T}(\overline{f}) = T(p)^{-1} \circ T(f) \circ T(e)^{-1}.
\end{equation}
Further if $T$ is right (resp.\ left) exact then $\overline{T}$ is right (resp.\ left) exact:
to prove that for $T$ left exact $\overline{T}$ preserves the kernel $\overline{f}:L \to M$ of $\overline{g}:M \to N$,
we may suppose after replacing $L$ by $\Ker(g \circ f)$ that $g \circ f = 0$,
in which case $f$ induces an $\sS$\nd isomorphism $L \to \Ker(g)$,
so that $T(f)$ is the kernel of $T(g)$.
If $\sA$ is an (AB4) category and $T$ preserves colimits then $\overline{T}$ preserves colimits.
Given also $T':\sA \to \sA'$ satisfying similar conditions to $T$, 
there exists for any natural transformation 
$\varphi:T \to T'$ a unique $\overline{\varphi}:\overline{T} \to \overline{T'}$ compatible
with $\varphi$ and $\sA \to \sA/\sS$.

Suppose that $\sA$ has a structure of tensor category,
and that if $f$ is an $\sS$\nd isomorphism in $\sA$ then $f \otimes N$ is an $\sS$\nd isomorphism 
for any object $N$ of $\sA$.
Then $\sA/\sS$ has a unique structure of tensor category such that \eqref{e:quotproj}
is a strict tensor functor.
Explicitly, the tensor product of $\sA/\sS$ coincides with that of $\sA$ on objects,
with associativity constraints and symmetries the images under \eqref{e:quotproj} of those
of $\sA$. 
The composite of \eqref{e:quotproj} with the endofunctor $- \otimes N$ of $\sA$
factors uniquely as the composite of an additive endofunctor $- \otimes N$ of $\sA/\sS$ 
with \eqref{e:quotproj}.
This endofunctor $- \otimes N$ of $\sA/\sS$ together with the similarly defined endofunctor
$M \otimes -$ define on $\sA$ a bifunctor: 
the condition for bifunctoriality can be verified using \eqref{e:Tbardef}.
The naturality of the constraints for this bifunctor $- \otimes -$ on $\sA/\sS$ follows
from \eqref{e:Tbardef}, their required compatibilities
from those in $\sA$, 
and  the compatibility of \eqref{e:quotproj} with  
$- \otimes -$ in $\sA$ and $\sA/\sS$ from its compatibility with each $M \otimes -$
and $- \otimes N$.

Let $\sA$ and $\sA'$ be abelian tensor categories and $T:\sA \to \sA'$ be a tensor functor.
Suppose that the above conditions on $T$ and the tensor product of $\sA$ are satisfied.
Then $\overline{T}:\sA/\sS \to \sA'$ has a unique structure of tensor functor
whose composite with the strict tensor functor \eqref{e:quotproj} is $T$.
Given similarly $T':\sA \to \sA'$, if $\varphi:T \to T'$ is compatible with the tensor structures,
then so is $\overline{\varphi}:\overline{T} \to \overline{T'}$.

Let $\sA$ be a tensor category.
For any $M$ in $\sA_{\mathrm{rig}}$, the endofunctor $M \otimes -$ is both right and left adjoint to
$M^\vee \otimes -$, and hence preserves limits and colimits.
If $\sA$ is abelian and $\I$ is projective in $\sA$, then for $M$ in $\sA_{\mathrm{rig}}$ 
the natural isomorphism
\begin{equation*}
\sA(M,-) \iso \sA(\I,M^\vee \otimes -)
\end{equation*} 
shows that $M$ is projective in $\sA$.
Similarly if $\sA$ has filtered colimits and $\I$ is of finite type (resp.\ of finite presentation)
in $\sA$, then any $M$ in $\sA_{\mathrm{rig}}$ is of finite type (resp.\ of finite presentation)
in $\sA$.

A cocomplete abelian category $\sA$ will be called a Grothendieck category
if filtered colimits are exact in $\sA$.
It is equivalent to require (\cite[III~1.9]{Mit65}, 
\cite[\href{https://stacks.math.columbia.edu/tag/0032}{Tag 0032}]{stacks-project})
that for every filtered system $(A_\lambda)$ of subobjects $A_\lambda$ of an object $A$ of $\sA$
we have
\begin{equation}\label{e:AB5}
(\cup A_\lambda) \cap B = \cup(A_\lambda \cap B)
\end{equation} 
for every subobject $B$ of $A$.

\begin{defn}
A Grothendieck tensor category $\sA$ will be called \emph{well-dualled} if $\I$ is of finite type
in $\sA$ and $\otimes$ is cocontinuous in $\sA$, and if $\sA_{\mathrm{rig}}$ is essentially 
small and its objects generate $\sA$. 
\end{defn}

\emph{For the rest of this section $\sA$ will be a well-dualled Grothendieck tensor category.}

\medskip

Since $\sA$ is an abelian category with a generator, it is well-powered.

An object of $\sA$ is of finite type if and only if it is a quotient of some object in $\sA_{\mathrm{rig}}$.
Every object of $\sA$ is the filtered colimit of its subobjects of finite type.
If $M' \to M$ is an epimorphism in $\sA$ with $M$ of finite type, there exists a $B$ in $\sA_{\mathrm{rig}}$
and a morphism $B \to M'$ such that $B \to M' \to M$ is an epimorphism.

Every object of $\sA$ is the quotient of a coproduct of objects in $\sA_{\mathrm{rig}}$.
It follows that for example every object of $\sA$ has a (left) resolution with objects such coproducts,
or that any short exact sequence in $\sA$ has a resolution by a short exact sequence of complexes with
objects such coproducts.

Call an object $M$ of $\sA$ \emph{flat} if the functor $M \otimes -$ on $\sA$ is exact.
Any filtered colimit of flat objects is flat, and any object lying in $\sA_{\mathrm{rig}}$ is flat.
Any object, or any short exact sequence, in $\sA$ has a flat resolution.

\begin{lem}\label{l:Mflat}
For any object $N$ of $\sA$, the functor $N \otimes -$ preserves those short exact sequences
\eqref{e:sex} in $\sA$ with $M''$ flat.
\end{lem}

\begin{proof}
Let $P_\bullet$ be a flat resolution of $N$.
Then we have a short exact sequence 
\begin{equation*}
0 \to P_\bullet \otimes M' \to P_\bullet \otimes M \to P_\bullet \otimes M'' \to 0
\end{equation*}
of complexes in $\sA$ with homology in degree $0$ given by applying $N \otimes -$ to \eqref{e:sex},
where the homology of $P_\bullet \otimes M''$ is concentrated in degree $0$.
The long exact homology sequence thus gives what is required.
\end{proof}

\begin{lem}\label{l:balanced}
Let $P_\bullet$ be a flat resolution of an object $M$ and $Q_\bullet$ be a flat resolution of an object
$N$ in $\sA$.
Then the homologies of any given degree of the complexes $M \otimes Q_\bullet$ and $P_\bullet \otimes N$ 
are isomorphic.  
\end{lem}

\begin{proof}
The terms $E^2_{p0}$ of the two spectral sequences associated to the double complex $P_\bullet \otimes Q_\bullet$
are isomorphic to the homologies of degree $p$ of the complexes $M \otimes Q_\bullet$ and $P_\bullet \otimes N$, 
while the terms $E^2_{pq}$ for $q>0$ are $0$.
The two spectral sequences thus both degenerate at $E^2$, with the terms $E^2_{p0}$ both isomorphic to the
homology in degree $p$ of the total complex associated to $P_\bullet \otimes Q_\bullet$.
\end{proof}

An object $M$ of $\sA$ will be called a \emph{torsion object} if for every morphism $b:B \to M$ with $B$ in 
$\sA_{\mathrm{rig}}$,
there exists a morphism $a:A \to \I$ in $\sA_{\mathrm{rig}}$ which is regular in $\sA_{\mathrm{rig}}$  such that 
\begin{equation*}
a \otimes b = 0:A \otimes B \to M 
\end{equation*}
We say that $M$ is \emph{torsion free} if for every $b:B \to M$ in $\sA$ with $B$ in 
$\sA_{\mathrm{rig}}$ and $b \ne 0$ and every regular $a:A \to \I$ in $\sA_{\mathrm{rig}}$
we have $a \otimes b \ne 0$.
Equivalent conditions for an object to be torsion or torsion free can be obtained
by dualising, with for example $b$ replaced by a morphism 
$\I \to B \otimes M$ and $a$ by a morphism $\I \to A$.
Any object of $\sA_{\mathrm{rig}}$ is torsion free.
Since the objects of $\sA_{\mathrm{rig}}$ generate $\sA$,
the tensor product of a regular morphism in $\sA_{\mathrm{rig}}$ with any non-zero 
morphism $N \to M$ in $\sA$ with $M$ torsion free is non-zero.
An object of $\sA$ is torsion (resp.\ torsion free) if and only if each of its subobjects of finite type is torsion
(resp.\ torsion free).
An object $M$ of finite type in $\sA$ is a torsion object
if and only if 
\begin{equation*}
a \otimes M = 0:A \otimes M \to M
\end{equation*}
for some regular $a:A \to \I$ in $\sA_{\mathrm{rig}}$.
The full subcategory
\begin{equation*}
\sA_{\mathrm{tors}}
\end{equation*}
of $\sA$ consisting of the torsion objects is a Serre subcategory
which is closed under the formation of colimits and tensor product with any object of $\sA$.

For any object $M$ of $\sA$, the filtered colimit $M_{\mathrm{tors}}$ of its torsion subobjects 
of finite type is the largest torsion subobject of $M$, and $M_{\mathrm{tf}} = M/M_{\mathrm{tors}}$ 
is the largest torsion free quotient of $M$.

\begin{lem}\label{l:adjtorspres}
Let $T:\sA \to \sA'$ be a cocontinuous tensor functor between well-dualled Grothendieck
tensor categories.
Suppose that the tensor functor $\sA_{\mathrm{rig}} \to \sA'{}\!_{\mathrm{rig}}$ induced by $T$
is regular.
Then $T$ preserves torsion objects.
If $T':\sA' \to \sA$ is a lax tensor functor right adjoint to $T$, 
then $T'$ preserves torsion free objects.
\end{lem}

\begin{proof}
That $T(M)$ is a torsion object if $M$ is a torsion object is clear when $M$ is of finite type,
because then $a \otimes M = 0$ for some regular $a:A \to \I$ in $\sA_{\mathrm{rig}}$.
The general case then follows from the cocontinuity of $T$ by writing $M$ as the colimit of its 
subobjects of finite type.

Let $N$ be a torsion free object of $\sA'$.
Then $M \to T'(N)$ is $0$ for any torsion object $M$ of $\sA$, because the morphism
$T(M) \to N$ corresponding to it under adjunction is $0$.
Thus $T'(N)$ is a torsion free object of $\sA$.
\end{proof}

$\sA$ has no non-zero torsion objects if and only if $\I$ has no non-zero torsion quotients in $\sA$.
Indeed for $M$ non-zero in $\sA$ there exists a non-zero $A \to M$ with $A$ in $\sA_{\mathrm{rig}}$,
and hence a non-zero $\I \to M \otimes A^\vee$, with image a torsion object if $M$ is.

\begin{lem}\label{l:regtorssub}
Let $J$ be a subobject of $\I$ in $\sA$.
Then $\I/J$ is a torsion object if and only if there exists a regular morphism
$A \to \I$ in $\sA_{\mathrm{rig}}$ which factors through $J$.
\end{lem}

\begin{proof}
Write $p:\I \to \I/J$ for the projection.
Then $\I/J$ is a torsion object if and only if 
$p \circ a = p \otimes a$ is $0$
for some regular $a:A \to \I$ in $\sA_{\mathrm{rig}}$.
\end{proof}

\begin{lem}\label{l:regtorscok}
A morphism $A \to \I$ in $\sA_{\mathrm{rig}}$ is regular if and only its 
cokernel in $\sA$ is a torsion object.
\end{lem}

\begin{proof}
Let $a:A \to \I$ be a morphism in $\sA_{\mathrm{rig}}$ whose cokernel in $\sA$
is a torsion object.
If $b:\I \to B$ is a morphism in $\sA_{\mathrm{rig}}$ for which $a \otimes b = b \circ a$
is $0$, then $b$ is $0$ because it factors through the cokernel of $a$ and $B$ is torsion free.
This proves the ``if''.
The ``only if'' follows from Lemma~\ref{l:regtorssub}.
\end{proof}

By an \emph{isomorphism up to torsion} in $\sA$ we mean a morphism in $\sA$ whose kernel and cokernel
are torsion objects, i.e.\ an $\sS$\nd isomorphism in $\sA$ for $\sS$ the Serre subcategory of $\sA$
consisting of the torsion objects. 

\begin{lem}\label{l:tensisotors}
If $N$ is an object of $\sA$ and $f$ is an isomorphism up to torsion in $\sA$, then $N \otimes f$ 
is an isomorphism up to torsion in $\sA$.
\end{lem}

\begin{proof}
Since $N \otimes -$ is right exact and sends torsion objects to torsion objects,
$N \otimes f$ is an isomorphism up to torsion for $f$ an epimorphism with torsion
kernel.
To show that $N \otimes f$ is an isomorphism up to torsion for $f$ a monomorphism
with torsion cokernel,
it is enough to show that a short exact sequence \eqref{e:sex} in $\sA$ with $M''$ a torsion object
induces an exact sequence 
\begin{equation}\label{e:torexact}
0 \to L \to N \otimes M' \to N \otimes M \to N \otimes M'' \to 0
\end{equation}
in $\sA$ with $L$ a torsion object.

The short exact sequence \eqref{e:sex} has a flat resolution
\begin{equation*}
0 \to P'{}\!_\bullet \to P_\bullet \to P''{}\!_\bullet \to 0.
\end{equation*}
Tensoring with $N$, we obtain using Lemma~\ref{l:Mflat} a short exact sequence of complexes
\begin{equation*}
0 \to N \otimes P'{}\!_\bullet \to N \otimes P_\bullet \to N \otimes P''{}\!_\bullet \to 0.
\end{equation*}
Passing to the long exact homology sequence, we obtain an exact sequence \eqref{e:torexact}
with $L$ a quotient of the homology in degree $1$ of the complex $N \otimes P''{}\!_\bullet$.
By Lemma~\ref{l:balanced}, this homology is isomorphic to the homology in degree $1$ of
$Q_\bullet \otimes M''$ with $Q_\bullet$ a flat resolution of $N$, and hence is a torsion object.
\end{proof}

Since $\sA$ is well-powered, we may form the quotient category $\sA/\sA_{\mathrm{tors}}$.
We write 
\begin{equation*}
\overline{\sA} = \sA/\sA_{\mathrm{tors}},
\end{equation*}
and denote by a bar the image in $\overline{\sA}$ of an object or morphism of $\sA$.
The projection $\sA \to \overline{\sA}$ is thus bijective on objects, with
\begin{equation}\label{e:torscolim}
\overline{\sA}(\overline{M},\overline{N}) = 
\colim_{M' \subset M, \; M/M' \; \textrm{torsion}} \sA(M',N_{\mathrm{tf}})
\end{equation}
where the colimit runs over those subobjects $M'$ of $M$ for which $M/M'$ is torsion.
The category $\overline{\sA}$ is abelian and cocomplete,
and the projection $\sA \to \overline{\sA}$
is exact and cocontinuous.
By Lemma~\ref{l:tensisotors}, $\overline{\sA}$ has a unique structure of tensor category such that 
the projection is a strict tensor functor.
Further $\otimes$ in $\overline{\sA}$ preserves coproducts and hence is cocontinuous,
and the $\overline{M}$ for $M$ dualisable in $\sA$ generate $\overline{\sA}$.

\begin{lem}\label{l:atens}
Let $B$ be an object in $\sA_{\mathrm{rig}}$ and $N$ be a torsion free object in $\sA$.
Then for any morphism $l:\overline{B} \to \overline{N}$ in $\overline{\sA}$
there exists a regular morphism $a:A \to \I$ in $\sA_{\mathrm{rig}}$ and a
morphism $h :A \otimes B \to N$ in $\sA$ such that $\overline{a} \otimes l = \overline{h}$.
\end{lem}

\begin{proof}
We may suppose after dualising that $B = \I$.
Let $J$ be a suboject of $\I$ in $\sA$ with $\I/J$ a torsion object
such that $l$ is the image in the colimit defining $\overline{\sA}(\I,\overline{N})$ of some
$j$ in $\sA(J,N)$.
By Lemma~\ref{l:regtorssub}, there exists a regular morphism $a:A \to \I$
in $\sA$ which factors through a morphism $a_0:A \to J$.
If $h$ is $j \circ a_0$, then $\overline{h}$ is $l \circ \overline{a} = \overline{a} \otimes l$.
\end{proof}

Let $N$ be a torsion free object of $\sA$.
Then for every object $M$ of $\sA$ and subobject $M'$ of $M$ with $M/M'$
a torsion object, restriction from $M$ to $M'$ defines an injective homomorphism
\begin{equation}\label{e:torsfreeinj}
0 \to \sA(M,N) \to \sA(M',N).
\end{equation}
The transition homomorphisms in the colimit of \eqref{e:torscolim} are thus injective, so that 
the projection $\sA \to \overline{\sA}$ is injective on hom-groups of $\sA$ with torsion free target.

\begin{lem}\label{l:projreg}
The tensor functor $\sA_{\mathrm{rig}} \to (\overline{\sA})_{\mathrm{rig}}$ induced by the 
projection $\sA \to \overline{\sA}$ is faithful and regular.
\end{lem}

\begin{proof}
The faithfulness follows from the above injectivity on hom groups with torsion free target.
Let $a$ be a regular morphism in $\sA_{\mathrm{rig}}$.
We show that $\overline{a} \otimes j$ is non-zero for any non-zero morphism 
$j:\overline{N} \to \overline{M}$ in $\overline{\sA}$.
We may suppose that $M$ is torsion free in $\sA$
and that $j = \overline{h}$ for some $h:N \to M$ in $\sA$.
Then $h$ and hence $a \otimes h$ is non-zero, so that $\overline{a} \otimes \overline{h}$
is non-zero by the above injectivity.  
\end{proof}

If $N$ is an object of $\sA$, then every subobject of $\overline{N}$ lifts to a subobject of $N$.
This can be seen by reducing after taking inverse images along $N \to N_\mathrm{tf}$ to the 
case where $N$ is torsion free, when by \eqref{e:torscolim} any subobject of $\overline{N}$
is of the form $\Img \overline{f}$ for some morphism $f:M' \to N$ in $\sA$, and hence lifts to the subobject
$\Img f$ of $N$.
The set of subobjects of $N$ lifting a given subobject of $\overline{N}$ is directed,
and its colimit is the unique largest such object of $N$.
By assigning to each subobject of $\overline{N}$ the largest subobject of $N$ lifting it,   
we obtain an order-preserving map from the set of subobjects of $\overline{N}$ to 
the set of subobjects of $N$.

\begin{prop}\label{p:welldualled}
$\overline{\sA}$ is a well-dualled Grothendieck tensor category.
\end{prop}

\begin{proof}
The cocompleteness of $\overline{\sA}$ has been seen.
That the equalities of the form \eqref{e:AB5} hold in $\overline{\sA}$
follows after lifting to $\sA$ from the fact that they hold in $\sA$ together
with the exactness and cocontinuity of the projection onto $\sA$.
Thus $\overline{\sA}$ is a Grothendieck category.

The cocontinuity of the tensor product of $\overline{\sA}$ has been seen.
To see that $\I$ is of finite type in $\overline{\sA}$, it is enough since $\overline{\sA}$
is a Grothendieck category to show that if $\I$ in $\overline{\sA}$ is the union
of a filtered system $(J_\lambda)$ of subobjects $J_\lambda$, then $J_\lambda = \I$
for some $\lambda$.
If $(J_0{}_\lambda)$ is a lifting of $(J_\lambda)$ to a system of subobjects
of $\I$ in $\sA$, then by cocontinuity of the projection onto $\overline{\sA}$,
the subobject
\begin{equation*}
J = \colim J_0{}_\lambda
\end{equation*}
 of $\I$ in $\sA$ has torsion quotient $\I/J$.
Thus by Lemma~\ref{l:regtorssub}, there exists a regular morphism 
$A \to \I$ in $\sA_{\mathrm{rig}}$ which factors through $J$.
Then $A \to \I$ factors through some $J_0{}_\lambda$, because $A$ is of finite type in $\sA$.
Thus $\I/J_0{}_\lambda$ is a torsion object in $\sA$ by Lemma~\ref{l:regtorssub},
so that $J_\lambda = \I$ in $\overline{\sA}$. 

Since $\I$ is of finite type in $\overline{\sA}$, so also is any dualisable object.
Thus since the $\overline{M}$ for $M$ dualisable in $\sA$
generate $\overline{\sA}$, every dualisable object in $\overline{\sA}$ is a 
quotient of such an $\overline{M}$.
It follows that $(\overline{\sA})_{\mathrm{rig}}$ is essentially small.
Thus $\overline{\sA}$ is well-dualled.
\end{proof}

\begin{lem}\label{l:regcofinal}
For every regular morphism $a:A \to \I$ in $(\overline{\sA})_{\mathrm{rig}}$ there exists
a regular morphism $a_0:A_0 \to \I$ in $\sA_{\mathrm{rig}}$ such that $\overline{a_0}$
factors through $a$.
\end{lem}

\begin{proof}
Since an epimorphism $p:\overline{B} \to A$ in $(\overline{\sA})_{\mathrm{rig}}$ exists
with $B$ in $\sA_{\mathrm{rig}}$,
and since $a \circ p$ is regular by Lemma~\ref{l:regtorscok}, we may after replacing $a$
by $a \circ p$ suppose that $A = \overline{B}$ with $B$ in $\sA_{\mathrm{rig}}$.
By Lemma~\ref{l:atens}, there is a regular $c:C \to \I$ in $\sA_{\mathrm{rig}}$ with
\begin{equation*}
a \circ (\overline{c} \otimes \overline{B}) = \overline{c} \otimes a = \overline{a_0}
\end{equation*}
for some $a_0:C \otimes B \to \I$.
By Lemma~\ref{l:projreg}, $\overline{c}$ and 
hence $\overline{a_0}$ is regular in $(\overline{\sA})_{\mathrm{rig}}$,
so that again by Lemma~\ref{l:projreg} $a_0$ is regular in $\sA_{\mathrm{rig}}$.
\end{proof}

\begin{prop}\label{p:notors}
$\overline{\sA}$ has no non-zero torsion objects.
\end{prop}

\begin{proof}
Let $N$ be torsion free object in $\sA$ with $\overline{N}$ a torsion object in 
$\overline{\sA}$,
and $b:B \to N$ be a morphism in $\sA$ with $B$ in $\sA_{\mathrm{rig}}$.
Then $a \otimes \overline{b} = 0$ for some regular $a:A \to \I$ in 
$(\overline{\sA})_{\mathrm{rig}}$.
By Lemma~\ref{l:regcofinal}, we may suppose that $a = \overline{a_0}$ with
$a_0:A_0 \to \I$ regular in $\sA_{\mathrm{rig}}$.
Then $a_0 \otimes b = 0$ by Lemma~\ref{l:projreg}.
Thus $N$ is a torsion object, and $\overline{N} = 0$.
\end{proof}

\begin{lem}\label{l:torsfrac}
Let $f:A \to \I$ be a regular morphism in $\sA_{\mathrm{rig}}$ and $h:A \to M$ be 
a morphism in $\sA$ with $M$ torsion free.
Then $f \otimes h = h \otimes f$ if and only if $\Ker h \supset \Ker f$.
\end{lem}

\begin{proof}
Write $J$ for the image of $f$ and $j:J \to \I$ for the embedding.
Then
\begin{equation*}
f = j \circ f_0.
\end{equation*}
with $f_0:A \to J$ the projection.
By Lemma~\ref{l:regtorscok} $j$ is an isomorphism up to torsion, so that by 
Lemma~\ref{l:tensisotors} $j \otimes j$ is an isomorphism up to torsion.
In particular $\Ker (j \otimes j)$ is a torsion object.
Now
\begin{equation*}
(j \otimes j) \circ (1 - \sigma) = j \otimes j - j \otimes j = 0
\end{equation*}
with $\sigma:J \otimes J \iso J \otimes J$ the symmetry.
The image of $1 - \sigma$ is thus a torsion object.

Suppose that $\Ker h \supset \Ker f$.
Then $h = h_0 \circ f_0$ for some $h_0:J \to M$,
and
\begin{equation*}
j \otimes h_0 - h_0 \otimes j = (j \otimes h_0) \circ (1 - \sigma) = 0,
\end{equation*}
because $M$ is torsion free so that $j \otimes h_0$ sends the image of 
$1 - \sigma$ to $0$.
Composing with $f_0 \otimes f_0$ then shows that $f \otimes h = h \otimes f$.

Conversely suppose that $f \otimes h = h \otimes f$.
If $i:\Ker f \to A$ is the embedding, then 
\begin{equation*}
(h \circ i) \otimes f = (h \otimes f) \circ (i \otimes A) = 
(f \otimes h) \circ (i \otimes A) = 0.
\end{equation*}
Since $f$ is regular and $M$ is torsion free, it follows that $h \circ i = 0$.
\end{proof}

\begin{prop}\label{p:fracclos}
$(\overline{\sA})_{\mathrm{rig}}$ is fractionally closed.
\end{prop}

\begin{proof}

Let $f:A \to \I$ be a regular morphism in $(\overline{\sA})_{\mathrm{rig}}$
and $h:A \to D$ be a morphism in $(\overline{\sA})_{\mathrm{rig}}$ with $f \otimes h = h \otimes f$.
Since $\overline{\sA}$ has no non-zero torsion objects  by Proposition~\ref{p:notors},
$f$ is an epimorphism in $\overline{\sA}$ by Lemma~\ref{l:regtorscok}. 
Hence $h$ is of the form $h_0 \circ f$ = $f \otimes h_0$ for some $h_0:\I \to D$
by Lemma~\ref{l:torsfrac}.
Thus \eqref{e:fraccloseI} with $\sC = (\overline{\sA})_{\mathrm{rig}}$ is surjective, 
as required.
\end{proof}

By the universal property of 
$(\sA_{\mathrm{rig}})_\mathrm{fr}$ together with Lemma~\ref{l:projreg}
and Proposition~\ref{p:fracclos}, the tensor functor
$\sA_{\mathrm{rig}} \to (\overline{\sA})_{\mathrm{rig}}$ factors uniquely as 
$E_{\sA_{\mathrm{rig}}}:\sA_{\mathrm{rig}} \to (\sA_{\mathrm{rig}})_\mathrm{fr}$
followed by a tensor functor
\begin{equation}\label{e:projfactor}
(\sA_{\mathrm{rig}})_\mathrm{fr} \to (\overline{\sA})_{\mathrm{rig}}.
\end{equation}
Explicitly, \eqref{e:projfactor} is the strict tensor functor that
coincides with $\sA_{\mathrm{rig}} \to (\overline{\sA})_{\mathrm{rig}}$ on objects,
and sends the morphism $h/f$ to the unique morphism $l$ with 
\begin{equation}\label{e:projfactordef}
\overline{f} \otimes l = \overline{h},
\end{equation} 
as in \eqref{e:Tfactor}.

\begin{prop}\label{p:Frff}
The tensor functor \eqref{e:projfactor} is fully faithful.
\end{prop}

\begin{proof}
By \eqref{e:projfactordef}, the faithfulness follows from Lemma~\ref{l:projreg}
and the fullness from Lemma~\ref{l:atens}. 
\end{proof}

Let $\sC$ be an essentially small rigid tensor category.
Since $\I$ is projective of finite type in $\widehat{\sC}$, so also is
any dualisable object of $\widehat{\sC}$. 
Thus by \eqref{e:fintype}, $(\widehat{\sC})_{\mathrm{rig}}$ is the pseudo-abelian hull 
of $\sC$ in $\widehat{\sC}$. 
It follows that $\widehat{\sC}$ is a well-dualled Grothendieck tensor category.
We write $\widetilde{\sC}$ for $\widehat{\sC}$ modulo torsion: 
\begin{equation*}
\widetilde{\sC} = \overline{\widehat{\sC}}.
\end{equation*}
The composite $\sC \to \widetilde{\sC}$ of the projection 
$\widehat{\sC} \to \widetilde{\sC}$ with $h_-:\sC \to \widehat{\sC}$ factors 
through a regular tensor functor $\sC \to (\widetilde{\sC})_\mathrm{rig}$, 
which in turn factors uniquely through a tensor functor
\begin{equation}\label{e:Cprojfactor}
\sC_\mathrm{fr} \to (\widetilde{\sC})_\mathrm{rig}
\end{equation}
by Proposition~\ref{p:fracclos}.

\begin{cor}\label{c:FrCff}
Suppose that either direct sums exist in $\sC$ or that $\sC$ is integral.
Then the tensor functor \eqref{e:Cprojfactor} is fully faithful.
\end{cor}

\begin{proof}
Since the tensor functor $T:\sC \to (\widehat{\sC})_\mathrm{rig}$ defined by $h_-$ is the embedding of
$\sC$ into its pseudo-abelian hull, \eqref{e:FrT} with $\sC' = (\widehat{\sC})_\mathrm{rig}$
is fully faithful.
The required result thus follows from Proposition~\ref{p:Frff}, because \eqref{e:Cprojfactor}
factors as $T_\mathrm{fr}$ followed by \eqref{e:projfactor} with $\sA = \widehat{\sC}$.
\end{proof}

Let $\sC'$ be an essentially small rigid tensor category and $F:\sC \to \sC'$ be a tensor functor.
As in Section~\ref{s:fun}, we have a tensor functor $\widehat{F}:\widehat{\sC} \to \widehat{\sC'}$ and a 
cocontinuous lax tensor functor  $F_\wedge:\widehat{\sC'} \to \widehat{\sC}$ right adjoint to $\widehat{F}$.
By \eqref{e:pres} and the fact \eqref{e:proj} is an isomorphism for $M$ dualisable, we have an isomorphism
\begin{equation}\label{e:projhat}
F_\wedge(M') \otimes M \iso F_\wedge(M' \otimes \widehat{F}(M))
\end{equation}
natural in $M$ in $\widehat{\sC}$ and $M'$ in $\widehat{\sC'}$.
Similarly, with $\sA = \widehat{\sC}$, $\sA' = \widehat{\sC'}$, $H = \widehat{F}$ and $H' = F_\wedge$,
\eqref{e:EMtens} is an isomorphism for every $N'$ in the essential image of $\widehat{F}$, 
and \eqref{e:EMhom} is an isomorphism for every $M'$ in the essential image of $\widehat{F}$.

\begin{lem}\label{l:isotors}
Let $\sC$ and $\sC'$ be essentially small tensor categories with $\sC$ rigid, 
and $F:\sC \to \sC'$ be a regular tensor functor.
Suppose that either direct sums exist in $\sC$ or that $\sC$ is integral.
Then $\widehat{F}:\widehat{\sC} \to \widehat{\sC'}$ sends isomorphisms up to torsion to isomorphisms up to torsion.
\end{lem}

\begin{proof}
The tensor functor $(\widehat{\sC})_{\mathrm{rig}} \to (\widehat{\sC'})_{\mathrm{rig}}$ induced by
$\widehat{F}$ between the pseudo-abelian hulls of $\sC$ and $\sC'$ is regular, because 
$\Reg(\sC) \to \Reg((\widehat{\sC})_{\mathrm{rig}})$ is cofinal.
Since $\widehat{F}$ is cocontinuous, it sends epimorphisms with torsion kernel to epimorphisms with torsion kernel.
It remains to show that $\widehat{F}$ sends monomorphism with torsion cokernel to morphisms
with torsion kernel and cokernel.

Suppose given a short exact sequence \eqref{e:sex} in $\widehat{\sC}$ with $M''$ torsion.
Then 
\begin{equation*}
\widehat{F}(M') \to \widehat{F}(M)
\end{equation*}
has torsion cokernel.
If $P'$ is its kernel, it is to be shown that for any $N'$ in $(\widehat{\sC'})_{\mathrm{rig}}$
and morphism $b':\I \to N' \otimes P'$ there is a regular $a':\I \to L'$ in 
$(\widehat{\sC'})_{\mathrm{rig}}$ with $a' \otimes b' = 0$.

The morphism $M' \to M$ induces the horizontal arrows of a commutative square
\begin{equation*}
\xymatrix{
\widehat{\sC'}(\I, N' \otimes \widehat{F}(M')) \ar[r] & \widehat{\sC'}(\I,N' \otimes \widehat{F}(M)) \\
\widehat{\sC}(\I, F_\wedge(N') \otimes M') \ar^{\wr}[u] \ar[r] &  \widehat{\sC}(\I, F_\wedge(N') \otimes M) \ar^{\wr}[u]
}
\end{equation*}
with the vertical isomorphisms given by the isomorphisms induced by those of the form 
\eqref{e:projhat} followed by the adjunction isomorphisms.
Since $N' \otimes -$ is exact, 
we may identify $b'$ with an element in the kernel of the top arrow of the square.
It is enough to show that there is 
a regular $a':\I \to L'$ in $(\widehat{\sC'})_{\mathrm{rig}}$
such that $a' \otimes b'$ in $\widehat{\sC'}(\I, L' \otimes N' \otimes \widehat{F}(M'))$ is $0$.
If $b'$ is the image under the left isomorphism of $b$, then $b:\I \to F_\wedge(N') \otimes M'$ factors through
the kernel $P$ of
\begin{equation*}
F_\wedge(N') \otimes M' \to F_\wedge(N') \otimes M.
\end{equation*}
Since $P$ is a torsion object by Lemma~\ref{l:tensisotors}, there is a regular $a:\I \to L$ in 
$(\widehat{\sC})_{\mathrm{rig}}$ such that $a \otimes b = 0$.
Now by \eqref{e:projadj}, the left isomorphism of the square is given by applying $\widehat{F}$ and then using the counit
$\widehat{F}F_\wedge(N') \to N'$.
Thus $a' \otimes b' = 0$ with $a' = \widehat{F}(a)$.
\end{proof}

Let $\sC$ and $\sC'$ be essentially small rigid tensor categories.
Suppose that either direct sums exist in $\sC$ or that $\sC$ is integral.
Then it follows from Lemma~\ref{l:isotors} that for any 
regular tensor functor $F:\sC \to \sC'$ there is a unique tensor functor 
\begin{equation*}
\widetilde{F}:\widetilde{\sC} \to \widetilde{\sC'}
\end{equation*}
which is compatible with $\widehat{F}$ and the projections 
$\widehat{\sC} \to \widetilde{\sC}$
and $\widehat{\sC'} \to \widetilde{\sC'}$.
Further for any tensor isomorphism $\varphi:F \iso F'$ of regular tensor functors $\sC \to \sC'$
there is a unique tensor isomorphism 
$\widetilde{\varphi}:\widetilde{F} \iso \widetilde{F'}$ 
which is compatible with $\widehat{\varphi}$ and the projections. 
There are the usual canonical tensor isomorphisms, satisfying the usual compatibilities,
for composable regular tensor functors.

\section{Modules}\label{s:mod}

In this section we study the behaviour of categories of modules in a tensor category
under passage to quotients modulo torsion.
The goal is to prove Theorem~\ref{t:Ftildeequiv}, which gives a geometric description
of the functor categories modulo torsion which will be used to define super Tannakian hulls.

Let $\sA$ be a well-dualled Grothendieck tensor category and $R$ be a commutative algebra in $\sA$.
Since the forgetful functor from $\MOD_{\sA}(R)$ to $\sA$ creates limits and colimits,
$\MOD_{\sA}(R)$ is a Grothendieck category.
Further the unit $R$ is of finite type in $\MOD_{\sA}(R)$ because $\Hom_R(R,-)$ is isomorphic
to $\sA(\I,-)$, and $\otimes_R$ is cocontinuous in $\MOD_{\sA}(R)$ by its definition using a coequaliser.
The free $R$\nd modules $R \otimes M$ with $M$ dualisable form a set of generators for $\MOD_{\sA}(R)$.
Thus any dualisable object $N$ of $\MOD_{\sA}(R)$ is a quotient of $R \otimes M$ with $M$ dualisable, 
because $N$ is of finite type in $\MOD_{\sA}(R)$.
It follows that $\MOD_{\sA}(R)_{\mathrm{rig}}$ is essentially small.
Thus $\MOD_{\sA}(R)$ is a well-dualled Grothendieck tensor category.

\begin{lem}\label{l:regtorsfree}
The tensor functor $R \otimes -:\sA_{\mathrm{rig}} \to \Mod_{\sA}(R)$ is regular
if and only if $R$ is a torsion free object of $\sA$.
\end{lem}

\begin{proof}
The tensor functor $R \otimes -$ is regular if and only if for every
regular $a:A \to \I$ in $\sA_{\mathrm{rig}}$ and non-zero 
$b':R \otimes B \to R$ in $\Mod_{\sA}(R)$ with $B$ in $\sA_{\mathrm{rig}}$,
the morphism $(R \otimes a) \otimes_R b'$ is non-zero.
If $b'$ corresponds under the adjunction isomorphism
\begin{equation*}
\Hom_{\sA}(-,R) \iso \Hom_R(R \otimes -,R) 
\end{equation*} 
to $b:B \to R$ in $\sA$,
then $(R \otimes a) \otimes_R b'$ corresponds, modulo the tensor structural 
isomorphism of $R \otimes -$, to $a \otimes b$.
The result follows.
\end{proof}

\begin{lem}\label{l:Rextpres}
Suppose that $R$ is a torsion free object of $\sA$.
Then the tensor functor $R \otimes -:\sA \to \MOD_{\sA}(R)$ preserves torsion objects and the
forgetful lax tensor functor $\MOD_{\sA}(R) \to \sA$ preserves torsion free objects and reflects
torsion objects.
\end{lem}

\begin{proof}
That $R \otimes -$ preserves torsion objects and the forgetful functor preserves
torsion free objects follows from Lemma~\ref{l:adjtorspres}  
with $T = R \otimes -$ and $T'$ the forgetful functor together with Lemma~\ref{l:regtorsfree}.
If $M$ in $\MOD_{\sA}(R)$ is a torsion object in $\sA$, then $M$ is a torsion object in
$\MOD_{\sA}(R)$ because it is a quotient of $R \otimes M$.
\end{proof}

The projection $\sA \to \overline{\sA}$ defines a canonical tensor functor
\begin{equation}\label{e:MODfun}
\MOD_{\sA}(R) \to \MOD_{\overline{\sA}}(\overline{R}).
\end{equation}
which is compatible with the projection and the forgetful lax tensor functors to 
$\sC$ and $\sC'$.
It is exact and cocontinuous.

\begin{lem}\label{l:MODfunsurj}
The tensor functor \eqref{e:MODfun} is essentially surjective.
\end{lem}

\begin{proof}
Any $\overline{R}$\nd module is isomorphic to $\overline{M}/L$
for some free $R$\nd module $M$ and $\overline{R}$\nd submodule $L$ of $\overline{M}$.
There exists a subobject $N$ of $M$ in $\sA$ with $\overline{N} = L$, 
and the image $RN$ of $R \otimes N \to L$ is an $R$\nd submodule of $M$
with $\overline{RN} = L$.
We then have $\overline{M/RN} = \overline{M}/L$.
\end{proof}

\begin{lem}\label{l:barfac}
Suppose that $R$ is a torsion free object of $\sA$.
Then for any morphism $a:M \to \overline{R}$ in $\Mod_{\overline{\sA}}(\overline{R})$
there exists a morphism $a_0:M_0 \to R$ in $\Mod_{\sA}(R)$ such that 
$\overline{a_0} = a \circ p$ with
$p:\overline{M_0} \to M$ an epimorphism in 
$\MOD_{\overline{\sA}}(\overline{R})$.
\end{lem}

\begin{proof}
After composing with an appropriate epimorphism, we may suppose that 
$M = \overline{R} \otimes \overline{B}$ with $B$ in $\sA_{\mathrm{rig}}$.
If $b:\overline{B} \to \overline{R}$ in $\overline{\sA}$
corresponds under adjunction to $a$,
we show that for some $b_0:A \to R$ in $\sA$ with $A$ in $\sA_{\mathrm{rig}}$ 
\begin{equation}\label{e:barfac}
\overline{b_0} = b \circ q
\end{equation}
with $q:\overline{A} \to \overline{B}$ an epimorphism in $\overline{\sA}$.
We may then take for $a_0$ the morphism $R \otimes A \to R$ 
corresponding under adjunction to $b_0$, and $p = \overline{R} \otimes q$.
Since $R$ is torsion free, there exist by \eqref{e:torscolim} a subobject 
$L$ of $B$ and a morphism $f:L \to R$ in $\sA$ such that the embedding of $L$ 
induces an isomorphism 
$i:\overline{L} \iso \overline{B}$ with
\begin{equation*}
b \circ i = \overline{f}.
\end{equation*}
Let $r$ be an epimorphism to $L$ in $\sA$ from a coproduct of objects $A_\lambda$
in $\sA_{\mathrm{rig}}$. 
Then $i \circ \overline{r}$ is an epimorphism in $\overline{\sA}$.
Since $\overline{B}$ is of finite type in $\overline{\sA}$,
there is a finite coproduct $A$ of the $A_\lambda$ such that
$i \circ \overline{r_0}$ with $r_0:A \to L$ the restriction of $r$ to $A$
is an epimorphism in $\overline{\sA}$.
Then \eqref{e:barfac} holds with $b_0 = f \circ r_0$ and $q = i \circ \overline{r_0}$. 
\end{proof}

\begin{lem}\label{l:Rmodbar}
Suppose that $R$ is a torsion free object of $\sA$.
\begin{enumerate}
\item\label{i:Rmodbarreg}
The tensor functor 
$\Mod_{\sA}(R) \to \Mod_{\overline{\sA}}(\overline{R})$ induced by \eqref{e:MODfun} 
is faithful and regular.
\item\label{i:Rmodbarpres}
The tensor functor \eqref{e:MODfun} preserves and reflects torsion objects, 
and preserves torsion free objects.
\end{enumerate}
\end{lem}

\begin{proof}
\ref{i:Rmodbarreg} 
Since $R$ is torsion free, \eqref{e:torsfreeinj} shows using \eqref{e:torscolim}
that the projection $\sA \to \overline{\sA}$ is injective on hom groups with target $R$.
Thus \eqref{e:MODfun} is injective on hom groups
\begin{equation*}
\Hom_R(R \otimes M,R)
\end{equation*}
for every $M$ in $\sA$.
Since every $R$\nd module is a quotient 
of a free $R$\nd module, this gives the faithfulness.

For the regularity it is enough to show that if $c:C \to R$ is regular in $\Mod_{\sA}(R)$ then 
$\overline{c} \otimes_{\overline{R}} d$ is non-zero for every non-zero $d:D \to \overline{R}$ 
in $\Mod_{\overline{\sA}}(\overline{R})$. 
By Lemma~\ref{l:barfac}, we may after composing $d$ with an appropriate epimorphism of 
$\overline{R}$\nd modules assume that $d = \overline{d_0}$ for some $d_0:D_0 \to R$ in 
$\Mod_{\sA}(R)$.
The required result then follows from the faithfulness. 

\ref{i:Rmodbarpres}
That \eqref{e:MODfun} preserves torsion objects follows from \ref{i:Rmodbarreg} 
and Lemma~\ref{l:adjtorspres} applied to \eqref{e:MODfun}.
It follows from \ref{i:Rmodbarreg} that 
$\Mod_{\sA}(R) \to \Mod_{\overline{\sA}}(\overline{R})$ reflects regular morphisms.
Thus by Lemma~\ref{l:regtorscok} applied to $\MOD_{\overline{\sA}}(\overline{R})$ 
and Lemma~\ref{l:barfac},
the composite of any regular morphism $a:A \to \overline{R}$ in 
$\Mod_{\overline{\sA}}(\overline{R})$
with an appropriate epimorphism in $\MOD_{\overline{\sA}}(\overline{R})$ is of the form
$\overline{a_0}$ for some regular morphism $a_0:A_0 \to R$ in $\Mod_{\sA}(R)$.
Now let $M$ be an $R$\nd module with $\overline{M}$ a torsion object in $\MOD_{\sA}(R)$.
Then if $b:B \to M$ is a morphism in $\MOD_{\sA}(R)$ with $B$ in $\Mod_{\sA}(R)$,
there exists a regular morphism $a_0:A_0 \to R$ in $\Mod_{\sA}(R)$ with 
$\overline{a_0} \otimes_{\overline{R}} \overline{b} = 0$ and hence
$a_0 \otimes_R b = 0$ by \ref{i:Rmodbarreg}.
Thus \eqref{e:MODfun} reflects torsion objects.
Similarly \eqref{e:MODfun} preserves torsion free objects.
\end{proof}

Suppose that $R$ is a torsion free object of $\sA$.
Then \eqref{e:MODfun} induces an exact tensor functor
\begin{equation}\label{e:MODfuntors}
\overline{\MOD_{\sA}(R)} \to 
\overline{\MOD_{\overline{\sA}}(\overline{R})}
\end{equation}
because \eqref{e:MODfun} is an exact tensor functor which by 
Lemma~\ref{l:Rmodbar}\ref{i:Rmodbarpres} preserves torsion objects.

\begin{prop}\label{p:torsequ}
Suppose that $R$ is a torsion free object of $\sA$.
Then \eqref{e:MODfuntors} is a tensor equivalence.
\end{prop}

\begin{proof}
Since \eqref{e:MODfun} is essentially surjective by Lemma~\ref{l:MODfunsurj}, 
so also is \eqref{e:MODfuntors}.
To show that  \eqref{e:MODfuntors} is fully faithful, it is enough by \eqref{e:torscolim}
to show for every pair of $R$\nd modules
$M$ and $N$ with $N$ torsion free in $\MOD_{\sA}(R)$ that the homomorphism
\begin{equation}\label{e:torsequff}
\colim_{M' \subset M, \, M/M' \, \text{torsion}}\Hom_R(M',N) \to 
\colim_{L \subset \overline{M}, \, \overline{M}/L \, \text{torsion}}\Hom_{\overline{R}}(L,\overline{N})
\end{equation}
is an isomorphism, where $M'$ runs over the $R$\nd submodules of $M$ with $M/M'$ a torsion $R$\nd module,
$\overline{N}$ is torsion free by Lemma~\ref{l:Rmodbar}\ref{i:Rmodbarpres},
$L$ runs over the $\overline{R}$\nd submodules of $\overline{M}$ with $\overline{M}/L$ a torsion 
$\overline{R}$\nd module, 
and the class of $i:M' \to N$ is sent to the class of $\overline{\imath}:\overline{M'} \to \overline{N}$.

Let $M'$ be an $R$\nd submodule of $M$ with $M/M'$ a torsion $R$\nd module, and
\begin{equation*}
i:M' \to N
\end{equation*}
be a morphism of $R$\nd modules whose class in the source of \eqref{e:torsequff} lies 
in the kernel of \eqref{e:torsequff}.
Since the transition homomorphisms of the target of \eqref{e:torsequff} are injective 
by the analogue of \eqref{e:torsfreeinj} for $\overline{R}$\nd modules, 
$\overline{\imath}:\overline{M'} \to \overline{N}$ is $0$ in $\overline{\sA}$. 
Thus $\Img i$ is a torsion object in $\sA$.
On the other hand, $\Img i$ is a torsion free object in $\sA$, because it is a subobject of $N$
with $N$ torsion free in $\sA$ by Lemma~\ref{l:Rextpres}.
Hence $\Img i = 0$ and $i = 0$.
Thus \eqref{e:torsequff} is injective.

Let $L$ be an $\overline{R}$\nd submodule of $\overline{M}$ with $\overline{M}/L$ a torsion
$\overline{R}$\nd module, and 
\begin{equation*}
j:L \to \overline{N}
\end{equation*}
be a morphism of $\overline{R}$\nd modules.
Then $L$ lifts to a subobject $M_0$ of $M$ in $\sA$.
Replacing if necessary $M_0$ by a smaller subobject $M_1$ of $M$ in $\sA$ with $M_0/M_1$
a torsion object in $\sA$, we may assume further that $j$ lifts to a morphism
\begin{equation*}
i_0:M_0 \to N
\end{equation*}
in $\sA$.
The image $M' = RM_0$ of the morphism
\begin{equation*}
m:R \otimes M_0 \to M
\end{equation*}
of $R$\nd modules corresponding to the embedding $M_0 \to M$ is an $R$\nd submodule of $M$ above $L$.
By Lemma~\ref{l:Rmodbar}\ref{i:Rmodbarpres}, $M/M'$ is a torsion $R$\nd module because 
$\overline{M/M'} = \overline{M}/L$ is a torsion $\overline{R}$\nd module.
Write
\begin{equation*}
i_1:R \otimes M_0 \to N
\end{equation*}
for the morphism of $R$\nd modules corresponding to the morphism $i_0:M_0 \to N$ in $\sA$.
Since $i_0$ lies above the morphism $j$ of $\overline{R}$\nd modules, 
we have $\overline{i_1(\Ker m)} = 0$ in $\overline{\sA}$. 
Hence $i_1(\Ker m)$ is a torsion object in $\sA$.
On the other hand $i_1(\Ker m)$ is a subobject of $N$ in $\sA$, and hence
by  Lemma~\ref{l:Rextpres} is torsion free in $\sA$.
Thus $i_1(\Ker m) = 0$, so that $i_1$ factors as
\begin{equation*}
R \otimes M_0 \to M' \xrightarrow{i} N
\end{equation*} 
where the first arrow is the epimorphism defined by $m$
and $i$ is a morphism of $R$\nd modules.
Thus $M'$ and $i$ lie above $L$ and $j$, so that the image of 
the class of $i$ under \eqref{e:torsequff} is the class of $j$.
Hence \eqref{e:torsequff} is surjective.
\end{proof}

Let $\sC$ be a rigid tensor category, $\sD$ be a cocomplete tensor category
with every object of $\sD_\mathrm{rig}$ of finite 
presentation in $\sD$, and 
\begin{equation*}
T:\sC \to \sD
\end{equation*}
be a fully faithful tensor functor.
As in Section~\ref{s:fun}, we have by additive Kan extension along $h_-:\sC \to \widehat{\sC}$
a tensor functor $T^*:\widehat{\sC} \to \sD$
with the universal natural transformation a tensor isomorphism \eqref{e:Th},
and there exists a lax tensor functor
$T_*:\sD \to \widehat{\sC}$ right adjoint to $T$.
By \eqref{e:pres}, a morphism
$j$ in $\widehat{\sC}$ is an isomorphism if and only if $\widehat{\sC}(h_A,j)$
is an isomorphism for every $A$.
Thus $T_*$ preserves filtered colimits: 
take for $j$ the canonical morphism from $\colim T_*(V_\lambda)$ to
$T_*(\colim V_\lambda)$ and use the fact that the $h_A$ and the objects of 
$\sD_\mathrm{rig}$ are of finite presentation.
Similarly composing the unit for the adjunction with $h_-$ gives an isomorphism
\begin{equation}\label{e:unithiso}
h_- \iso T_*T^*h_-,
\end{equation}
as can be seen by taking for $j$ the components of \eqref{e:unithiso} and using
the fact that by \eqref{e:Th} $T^*h_-$ is fully faithful.
Taking the component of \eqref{e:unithiso} at $\I$ shows using the compatibility of \eqref{e:unithiso}
with the tensor structures that the unit
for $T$ is an isomorphism
\begin{equation}\label{e:Tunitiso}
\I \iso T_*(\I).
\end{equation} 
Thus the forgetful lax tensor functor 
$\MOD_{\widehat{\sC}}(T_*(\I)) \to \widehat{\sC}$ is a tensor equivalence,
so that by \eqref{e:EMtens} with $H = T^*$ and $H' = T_*$ the tensor structural 
morphism
\begin{equation}\label{e:Ttensiso}
T_*(V) \otimes T_*(W) \to T_*(V \otimes W)
\end{equation}
is an isomorphism for $W$ in the essential image of $T$. 

Let $k$ be a field of characteristic $0$.
If $\bm|\bn$ is a family of pairs of integers $\ge 0$, 
we may take $\sC = \sF_{\bm|\bn}$ and $\sD = \MOD_{\GL_{\bm|\bn},\varepsilon}(k)$
above, and for
\begin{equation*}
T:\sF_{\bm|\bn} \to \MOD_{\GL_{\bm|\bn},\varepsilon}(k)
\end{equation*}
the $k$\nd tensor functor obtained by composing the embedding of 
$\Mod_{\GL_{\bm|\bn},\varepsilon}(k)$ into $\MOD_{\GL_{\bm|\bn},\varepsilon}(k)$
with \eqref{e:freeGL}:
that $T$ is fully faithful follows from Lemma~\ref{l:freeGLff}.
We define a lax tensor functor 
\begin{equation}\label{e:Hdef}
\MOD_{\GL_{\bm|\bn},\varepsilon}(k) \to \widetilde{\sF_{\bm|\bn}}
\end{equation}
by composing the projection $\widehat{\sF_{\bm|\bn}} \to \widetilde{\sF_{\bm|\bn}}$
with $T_*$.

\begin{prop}\label{p:MODGLFequiv}
The lax tensor functor \eqref{e:Hdef} is a tensor equivalence.
\end{prop}

\begin{proof}
Write $H$ for \eqref{e:Hdef}, $G$ for $\GL_{\bm|\bn}$ and $\sF$ for $\sF_{\bm|\bn}$.
By Lemma~\ref{l:freeGLff}, $\sF$ is integral.
The functor $H$ is left exact and preserves filtered colimits.
By \eqref{e:Tunitiso}, the unit for the lax tensor functor $H$ is an isomorphism
\begin{equation}\label{e:Hunitiso}
\I \iso H(k)
\end{equation}
and by \eqref{e:Ttensiso} the tensor structural morphism
\begin{equation}\label{e:Htensiso}
H(V) \otimes H(W) \to H(V \otimes_k W)
\end{equation}
of $H$ is an isomorphism for $W$ in the essential image of $T$.
By Lemma~\ref{l:Fmnfracclose} and Corollary~\ref{c:FrCff}, the composite 
$\sF \to \widetilde{\sF}$ of the projection 
$\widehat{\sF} \to \widetilde{\sF}$ with $h_-$ is fully faithful.
Thus by \eqref{e:Th} and \eqref{e:unithiso}, $HT$ is fully faithful.
The homomorphism
\begin{equation}\label{e:Hhomiso}
\Hom_G(V,W) \to \Hom_{\widetilde{\sF}}(H(V),H(W))
\end{equation}
induced by $H$ is then an isomorphism for $V$ and $W$ in the essential image of $T$. 

Given a short exact sequence
\begin{equation}\label{e:Vsex}
0 \to V' \to V \to V'' \to 0
\end{equation}
in $\Mod_{G,\varepsilon}(k)$, there exists by Theorem~\ref{t:VtensW} and 
Proposition~\ref{p:Msummand}\ref{i:Msummandss} a representation $W_0 \ne 0$ of 
$(\rM_{\bm|\bn},\varepsilon)$ such that
\begin{equation*}
0 \to V' \otimes_k W_0 \to V \otimes_k W_0 \to V'' \otimes_k W_0 \to 0
\end{equation*}
is a split short exact sequence.
By Proposition~\ref{p:Msummand}\ref{i:Msummand}, 
$W_0$ lies in the pseudo-abelian hull of the essential image of $T$,
so that \eqref{e:Htensiso} is an isomorphism for $W = W_0$ and any $V$,
and \eqref{e:Hhomiso} is an isomorphism
for $V = W = W_0$.
The naturality of \eqref{e:Htensiso} then shows that applying $H$ to \eqref{e:Vsex} and tensoring 
with $H(W_0)$ gives split short exact sequence
\begin{equation*}
0 \to H(V') \otimes H(W_0) \to H(V) \otimes H(W_0) \to H(V'') \otimes H(W_0) \to 0.
\end{equation*}
Since \eqref{e:Htensiso} with $W_0$ or $W_0{}\!^\vee$ for $W$ is an isomorphism,
$H(W_0)$ is dualisable.
Hence $- \otimes H(W_0)$ is exact.
If $M$ is the cokernel of $H(V') \to H(V)$ and
\begin{equation*}
i:M \to H(V'')
\end{equation*}
is the unique factorisation of $H(V) \to H(V'')$ through $H(V) \to M$, 
it follows that $i \otimes H(W_0)$ is an isomorphism.
Thus
\begin{equation*}
(\Ker i) \otimes H(W_0) = 0 = (\Coker i) \otimes H(W_0)
\end{equation*}
by exactness of $- \otimes H(W_0)$.
Now $H(W_0) \ne 0$ because \eqref{e:Hhomiso} is an isomorphism with $V = W = W_0$,
so that $1_{H(W_0)}$ is regular in the pseudo-abelian hull 
$(\widetilde{\sF})_\mathrm{rig}$ of the integral $k$\nd tensor
category $\sF$.
Since $\widetilde{\sF}$ has no non-zero torsion objects 
by Proposition~\ref{p:notors} with $\sA = \widehat{\sF}$, it follows that 
\begin{equation*}
\Ker i = 0 = \Coker i.
\end{equation*}
Thus $i$ is an isomorphism, so that the restriction of $H$ to 
$\Mod_{G,\varepsilon}(k)$ is right exact.
Since $H$ preserves filtered colimits 
and every morphism in $\MOD_{G,\varepsilon}(k)$ is a filtered colimit of morphisms 
in $\Mod_{G,\varepsilon}(k)$,
it follows that $H$ is right exact, and hence exact and cocontinuous.

By Corollary~\ref{c:quotsub}, every object in $\Mod_{G,\varepsilon}(k)$ is a kernel 
(resp.\ a cokernel) of a morphism in the pseudo-abelian hull of the essential image of $T$.
Thus \eqref{e:Htensiso} is an isomorphism for $W$ in $\Mod_{G,\varepsilon}(k)$ and every $V$.
In particular $H(V)$ is dualisable and hence of finite type for $V$ in 
$\Mod_{G,\varepsilon}(k)$.
Similarly \eqref{e:Hhomiso} is an isomorphism for $V$ and $W$ in $\Mod_{G,\varepsilon}(k)$.
Since every object of $\MOD_{G,\varepsilon}(k)$ is the filtered colimit of its subobjects
in $\Mod_{G,\varepsilon}(k)$, it follows that \eqref{e:Htensiso} is an isomorphism for
every $V$ and $W$, and that \eqref{e:Hhomiso} is an isomorphism for $V$ in 
$\Mod_{G,\varepsilon}(k)$ and every $W$, and hence for every $V$ and $W$.

That \eqref{e:Htensiso} and \eqref{e:Hhomiso} are isomorphisms for every $V$ and $W$ shows together with
\eqref{e:Hunitiso} that $T$ is a fully faithful tensor functor.
By \eqref{e:Th} and \eqref{e:unithiso} together with \eqref{e:pres}, every object of $\widehat{\sF}$ is a 
cokernel of a morphism between objects in the essential image of $T_*$.
Hence by exactness of the projection onto $\widetilde{\sF}$,
every object of $\widetilde{\sF}$
is a cokernel of a morphism between objects in the essential image of $H$.
The essential surjectivity of $H$ thus follows from its exactness and full faithfulness.
\end{proof}

\begin{thm}\label{t:Ftildeequiv}
Let $k$ be a field of characteristic $0$ and $\sC$ be as in Proposition~\textnormal{\ref{p:FmnCesssurj}}.
Then for some $\bm|\bn$ there exists an affine super $(\GL_{\bm|\bn},\varepsilon)$\nd scheme $X$ with 
$\Mod_{\GL_{\bm|\bn},\varepsilon}(X)$ integral such that $\widetilde{\sC}$ is $k$\nd tensor
equivalent to $\overline{\MOD_{\GL_{\bm|\bn},\varepsilon}(X)}$.
\end{thm}

\begin{proof}
Let $\bm|\bn$ and $R$ be as in Theorem~\ref{t:Fhatequiv}.
Since $\sC$ and $\sF_{\bm|\bn}$ are integral, so are their pseudo-abelian hulls
$(\widehat{\sC})_\mathrm{rig}$ and $(\widehat{\sF_{\bm|\bn}})_\mathrm{rig}$,
and hence also $\Mod_{\widehat{\sF_{\bm|\bn}}}(R)$.
Thus $R \otimes -$ from $(\widehat{\sF_{\bm|\bn}})_\mathrm{rig}$ to 
$\Mod_{\widehat{\sF_{\bm|\bn}}}(R)$ is regular, so that by Lemma~\ref{l:regtorsfree}
$R$ is a torsion free object of $\widehat{\sF_{\bm|\bn}}$.
By Propositions~\ref{p:torsequ} and \ref{p:MODGLFequiv}, if $R'$ is the image of
$\overline{R}$ under a quasi-inverse of the $k$\nd tensor equivalence \eqref{e:Hdef},
we have $k$\nd tensor equivalences 
\begin{equation*}
\widetilde{\sC} \to \overline{\MOD_{\widetilde{\sF_{\bm|\bn}}}(\overline{R})} \to
\overline{\MOD_{\GL_{\bm|\bn},\varepsilon}(R')}.
\end{equation*}
Applying $(-)_\mathrm{rig}$ and using Lemma~\ref{l:projreg} shows that 
$\Mod_{\GL_{\bm|\bn},\varepsilon}(R')$ is integral.
Thus we may take $X = \Spec(R')$.
\end{proof}

\section{Equivariant sheaves}\label{s:equ}

In this section we consider tensor categories of equivariant quasi-coherent sheaves
over a super  scheme with an action of an affine super group over a field of characteristic $0$.
The main results, Theorems~\ref{t:transaff} and \ref{t:GKMODequiv},
show that, under appropriate restrictions, the quotient of such a tensor category
modulo torsion is tensor equivalent to the category of modules over a transitive affine groupoid.

Throughout this section, $k$ is a field of characteristic $0$ and $(G,\varepsilon)$ 
is an affine super $k$\nd group with involution.

A morphism $Y \to X$ of super $k$\nd schemes will be called \emph{quasi-compact}
if the inverse image of every quasi-compact open super subscheme of $X$ is quasi-compact
in $Y$, \emph{quasi-separated} if the diagonal $Y \to Y \times_X Y$ is quasi-compact,
and \emph{quasi-affine} if it is quasi-compact and the inverse image of any affine open
super subscheme of $X$ is isomorphic to an open super subscheme of an affine
super $k$\nd scheme. 
A super $k$\nd scheme will be called quasi-compact (quasi-separated, quasi-affine) if its
structural morphism is.

Let $f:Y \to X$ be a  quasi-compact and quasi-separated morphism of super $k$\nd schemes and $\sV$ be a 
quasi-coherent $\sO_Y$\nd module.
Then $f_*\sV$ is a quasi-coherent $\sO_X$\nd module:
reduce first to the case where $X$ is affine, then by taking a finite affine open 
cover of $Y$ to the case where $Y$ is quasi-affine, and finally to the case where $Y$ is affine.
Similarly the base change morphism for pullback of $f$ and $\sV$ along a flat morphism $X' \to X$ of super schemes
is an isomorphism.
When $X = \Spec(k)$, the push forward $f_*\sV$ of $\sV$ may be identified with the 
super $k$\nd vector space
\begin{equation*}
H^0(X,\sV)
\end{equation*}
of global sections of $\sV$.

A morphism $f:Y \to X$ will be called \emph{super schematically dominant} if for
every open super subscheme $X'$ of $X$ the morphism $Y' \to X'$ induced by $f$ on the 
inverse image $Y'$ of $X'$ in $Y$ factors through no closed super subscheme of $X'$
strictly contained in $X'$. 
For $f$ quasi-compact and quasi-separated, it is equivalent to require that the
canonical morphism $\sO_X \to f_*\sO_Y$ in $\MOD(X)$ be a monomorphism. 
A super subscheme of $X$ will be called \emph{super schematically dense} if the embedding is
super schematically dominant

Suppose that $f$ is quasi-affine.
The canonical morphism
\begin{equation*}
Y \to \Spec(f_*\sO_Y)
\end{equation*}
is an open immersion.
For any quasi-coherent $\sO_Y$\nd module $\sV$ the counit $f^*f_*\sV \to \sV$ is an epimorphism:
reduce to the cases where $f$ is an open immersion or where $f$ is affine.

Let $f:Y \to X$ be a  quasi-compact and quasi-separated morphism of super $(G,\varepsilon)$\nd schemes 
and $\sV$ be a $(G,\varepsilon)$\nd equivariant quasi-coherent $\sO_Y$\nd module.
Then $f_*\sV$ has a 
canonical structure of $(G,\varepsilon)$\nd equivariant $\sO_X$\nd module.
We have a lax tensor functor $f_*$ from $\MOD_{G,\varepsilon}(Y)$ to $\MOD_{G,\varepsilon}(Y)$
right adjoint to $f^*$, where the unit and counit for the adjunction have the same
components as those for the underlying modules.

Let $Y$ be a quasi-compact quasi-separated super $(G,\varepsilon)$\nd scheme. 
Then $H^0(Y,\sV)$ for $\sV$ in $\MOD_{G,\varepsilon}(Y)$ has a canonical structure of 
$(G,\varepsilon)$\nd module.
We denote by
\begin{equation*}
H^0_G(Y,\sV)
\end{equation*}
the $k$\nd vector subspace of $H^0(Y,\sV)$ of invariants under $G$.
It consists of those global sections $v$ of $\sV$, necessarily lying in the even
part $\sV_0$ of $\sV$, such that the action of $G$ on $\sV$ sends the pullback
of $v$ along the projection from $G \times_k Y$ to $Y$ to its pullback along 
the action of $G$ on $Y$. 
Thus
\begin{equation*}
H^0_G(Y,\sV) = \Hom_{G,\sO_Y}(\sO_Y,\sV).
\end{equation*}
The super $k$\nd algebra $H^0(Y,\sO_Y)$ is a commutative $(G,\varepsilon)$\nd algebra 
with the canonical morphism
\begin{equation*}
Y \to \Spec(H^0(Y,\sO_Y)) 
\end{equation*}
a morphism of $(G,\varepsilon)$\nd schemes.

\begin{lem}\label{l:quotsub}
Let $X$ be a quasi-affine super $(G,\varepsilon)$\nd scheme.
\begin{enumerate}
\item\label{i:quotsubmod}
Every object of $\MOD_{G,\varepsilon}(X)$ is a quotient of one of the form $V \otimes_k \sO_X$
for some $(G,\varepsilon)$\nd module $V$.
\item\label{i:quotsubrep}
Every object of $\Mod_{G,\varepsilon}(X)$ is a quotient (resp.\ subobject) of one of the form 
$V \otimes_k \sO_X$ for some representation $V$ of $(G,\varepsilon)$. 
\end{enumerate}
\end{lem}

\begin{proof}
Let $V$ be the push forward of $\sV$ in $\MOD_{G,\varepsilon}(X)$ along the structural morphism of $X$.
Then the counit $V \otimes_k \sO_X \to \sV$ is an epimorphism in $\MOD(X)$,
and hence in $\MOD_{G,\varepsilon}(X)$.
This gives \ref{i:quotsubmod}.
If $\sV$ lies in $\Mod_{G,\varepsilon}(X)$, writing 
$V$ as the filtered colimit of its $(G,\varepsilon)$\nd submodules of finite type
and using the quasi-compactness of $X$ gives the quotient case of \ref{i:quotsubrep}.
The subobject case follows by taking duals.
\end{proof}

It follows from Lemma~\ref{l:quotsub}\ref{i:quotsubmod} that if $X$ is a quasi-affine 
super $(G,\varepsilon)$\nd scheme then $\MOD_{G,\varepsilon}(X)$ is a well-dualled
Grothendieck tensor category.

\begin{lem}\label{l:domfaith}
Let $f:X' \to X$ be a morphism of super $(G,\varepsilon)$\nd schemes, with $X$ quasi-affine and $X'$
quasi-compact and quasi-separated.
Then $f$ is super schematically dominant if and only if 
$f^*:\Mod_{G,\varepsilon}(X) \to \Mod_{G,\varepsilon}(X')$ is faithful.
\end{lem}

\begin{proof}
Since $f$ is quasi-compact and quasi-separated, the unit $\sO_X \to f_*\sO_{X'}$ is a morphism in 
$\MOD_{G,\varepsilon}(X)$.
By Lemma~\ref{l:quotsub}\ref{i:quotsubmod} applied to its kernel, it is a monomorphism
if and only if its composite with every non-zero morphism  $V \otimes_k \sO_X \to \sO_X$
in $\Mod_{G,\varepsilon}(X)$ with $V$ in $\Mod_{G,\varepsilon}(k)$ is non-zero. 
\end{proof}

A tensor category will be called \emph{reduced} if $f^{\otimes n} = 0$ implies $f = 0$
for every morphism $f$.

\begin{lem}\label{l:rednil}
Let $X$ be a quasi-affine super $(G,\varepsilon)$\nd scheme with $\Mod_{G,\varepsilon}(X)$ reduced.
Then there are no non-zero $(G,\varepsilon)$\nd ideals of $\sO_X$ contained 
in the nilradical of $\sO_X$.
\end{lem}

\begin{proof}
Let $\sJ$ be a non-zero $(G,\varepsilon)$\nd ideal of $\sO_X$.
By Lemma~\ref{l:quotsub}\ref{i:quotsubmod}, there is for some representation $V$ of $(G,\varepsilon)$ a non-zero
morphism $\varphi$ from $V \otimes_k \sO_X$ to $\sO_X$ which factors through $\sJ$.
Thus $\sJ$ cannot be contained in the nilradical of $\sO_X$, because otherwise the quasi-compactness of $X$ 
would imply that $\varphi^{\otimes n} = 0$ for some $n$.
\end{proof}

\begin{lem}\label{l:pullepi}
Let $X$ be a quasi-affine super $(G,\varepsilon)$\nd scheme with $\Mod_{G,\varepsilon}(X)$ reduced, 
and $j:Z \to X$ be a morphism of super $k$\nd schemes which factors through every
non-empty open super $G$\nd subscheme of $X$.
Then $j^*$ sends every non-zero morphism in $\MOD_{G,\varepsilon}(X)$ with target $\sO_X$ 
to an epimorphism in $\MOD(Z)$. 
\end{lem}

\begin{proof}
The image of a non-zero morphism $\sV \to \sO_X$ in $\MOD_{G,\varepsilon}(X)$ is a non-zero 
$G$\nd ideal $\sJ$ of $\sO_X$.
By Lemma~\ref{l:rednil}, the complement $Y$ of the closed super $G$\nd subscheme of $X$ defined
by $\sJ$ is a non-empty open $G$\nd subscheme of $X$.
Thus $j$ factors through $Y$, so that $j^*$ sends $\sO_X/\sJ$ to $0$, and hence $\sV \to \sO_X$
to an epimorphism.
\end{proof}

Let $X$ be a quasi-compact and quasi-separated super $k$\nd scheme.
Then $H^0(X,-)$ preserves filtered colimits of quasi-coherent $\sO_X$\nd modules: 
reduce by taking a finite affine
open cover of $X$ first to the case where $X$ is quasi-affine and hence separated, 
and then to the case where $X$ is affine.
If $X$ is the limit of a filtered system $(X_\lambda)$ of quasi-compact and quasi-separated super $k$\nd schemes 
with affine transition morphisms, pushing forward the structure sheaves onto some $X_{\lambda_0}$ shows that
$H^0(X,\sO_X)$ is the colimit of the $H^0(X_\lambda,\sO_{X_\lambda})$.

\begin{lem}\label{l:limopensub}
Let $X$ be a filtered limit $\lim_{\lambda \in \Lambda}X_\lambda$ of quasi-affine super $(G.\varepsilon)$\nd schemes
with affine transition morphisms, and $Y$ be an open super $G$\nd subscheme of $X$.
Then for each point $y$ of $Y$ there exists for some $\lambda \in \Lambda$ a quasi-compact 
open super $G$\nd subscheme 
$Y'$ of $X_\lambda$ containing the image of $y$ in $X_\lambda$ 
such that the inverse image of $Y'$ in $X$ is contained in $Y$.
\end{lem}

\begin{proof}
We may suppose that $G$, $X$ and the $X_\lambda$ are reduced.
Since $X$ is quasi-affine, $y \in X_f \subset Y$ for some $f$ in $H^0(X,\sO_X)$.
If $f_1, \dots , f_n$ is a basis of the $G$\nd submodule of $H^0(X,\sO_X)$ generated by $f$, then
\begin{equation*}
y \in X_{f_1} \cup \dots \cup X_{f_n} \subset Y.
\end{equation*}
For some $\lambda \in \Lambda$, each $f_i$ comes from an $f'{}\!_i$ in 
$H^0(X_\lambda,\sO_{X_\lambda})$ with the $f'{}\!_i$ a 
basis of a $G$\nd submodule of $H^0(X_\lambda,\sO_{X_\lambda})$.
Then we may take $X_\lambda{}_{f'{}\!_1} \cup \dots \cup X_\lambda{}_{f'{}\!_n}$ 
for $Y'$.
\end{proof}

\begin{lem}\label{l:qafflim}
Let $X$ be a quasi-affine super $(G,\varepsilon)$\nd scheme.
Then $X$ is the limit of a filtered inverse system $(X_\lambda)_{\lambda \in \Lambda}$ of 
quasi-affine super $(G.\varepsilon)$\nd schemes of finite type with affine transition morphisms
and super schematically dominant projections.
If $X$ is affine then the $X_\lambda$ may be taken to be affine.
\end{lem}

\begin{proof}
If $X = \Spec(R)$ is affine, writing $R$ as the filtered colimit of its finitely generated 
super $G$\nd subalgebras gives the required result, with the $X_\lambda$ affine.
In general, we may regard $X$ as an open super $G$\nd subscheme of the spectrum $Z$ of
$H^0(X,\sO_X)$.
By the affine case, $Z$ is the limit of a filtered inverse system $(Z_\lambda)$ of affine 
$(G,\varepsilon)$\nd schemes of finite type with super schematically dominant projections.
By quasi-compactness of $X$ and Lemma~\ref{l:limopensub}, there exists a $\lambda_0$
such that $X$ is the inverse image of an open super $G$\nd subscheme $X'$ of $Z_{\lambda_0}$.
Then the system $(X_\lambda)_{\lambda \ge \lambda_0}$ with $X_\lambda$ the inverse image 
of $X'$ in $Z_\lambda$ has the required properties.
\end{proof}

\begin{lem}\label{l:schdense}
Let $X$ be a quasi-affine super $(G,\varepsilon)$\nd scheme with $\Mod_{G,\varepsilon}(X)$ integral.
Then any non-empty open super $G$\nd subscheme of $X$ is super schematically dense.
\end{lem}

\begin{proof}
Suppose first that $X$ is of finite type.
Let $U_1$ be a non-empty open super $G$\nd subscheme of $X$.
The open super subscheme $U_2$ of $X$ on the complement of the closure of $U_1$
is a super $G$\nd subscheme which is disjoint from $U_1$, and $U_1 \cup U_2$
is dense in $X$.
If $j_i:U_i \to X$ is the embedding, then the canonical morphism
\begin{equation*}
\sO_X \to j_i{}_*\sO_{U_i}
\end{equation*}
has kernel a quasi-coherent $G$\nd ideal $\sJ_i$ of $\sO_X$.
Since $\sJ_1 \cap \sJ_2$ is $0$ on $U_1 \cup U_2$, it is contained in the 
nilradical of $\sO_X$, and hence is $0$ by Lemma~\ref{l:rednil}.
Thus $\sJ_1\sJ_2 = 0$.
Now $\sJ_2 \ne 0$ because $U_1$ is non-empty and $U_1 \cap U_2 = \emptyset$.
If $\sJ_1$ were $\ne 0$, then by Lemma~\ref{l:quotsub}\ref{i:quotsubmod} there would for each 
$i$ be a non-zero morphism $V_i \otimes_k \sO_X \to \sO_X$ in $\Mod_{G,\varepsilon}(X)$ 
which factors through $\sJ_i$, 
contradicting the integrality of $\Mod_{G,\varepsilon}(X)$.
Thus $\sJ_1 = 0$, and $U_1$ is super schematically dense in $X$.

To prove the general case, write $X$ as the limit of a filtered inverse system 
$(X_\lambda)_{\lambda \in \Lambda}$ as in Lemma~\ref{l:qafflim}.
Let $U$ be a non-empty open super $G$\nd subscheme of $X$.
There exists by Lemma~\ref{l:limopensub} a $\lambda_0 \in \Lambda$ such that
\begin{equation*}
\emptyset \ne \pr_{\lambda_0}^{-1}(U_0) \subset U
\end{equation*}
for an open super $G$\nd subscheme $U_0$ of $X_{\lambda_0}$.
Now by Lemma~\ref{l:domfaith}
the $\Mod_{G,\varepsilon}(X_\lambda)$ are integral.
Thus by the case where $X$ is of finite type, the inverse image of $U_0$ in $X_\lambda$ 
is super schematically dense for each $\lambda \ge \lambda_0$.
Covering $X_{\lambda_0}$ with affine open subsets then shows that
$\pr_{\lambda_0}^{-1}(U_0)$ and hence $U$ is super schematically dense in $X$.
\end{proof}

Let $X$ be a super $(G,\varepsilon)$\nd scheme and $x$ be a point of $X$.
We write
\begin{equation*}
\omega_{X,x}:\Mod_{G,\varepsilon}(X) \to \Mod(\kappa(x))
\end{equation*}
for the $k$\nd tensor functor given by taking the fibre at $x$.

\begin{lem}\label{l:pointfaith}
Let $X$ be a quasi-affine super $(G,\varepsilon)$\nd scheme with $\Mod_{G,\varepsilon}(X)$ integral.
Then $X$ has a point which lies in every non-empty open super $G$\nd subscheme of $X$.
For any such point $x$, the functor $\omega_{X,x}$ from $\Mod_{G.\varepsilon}(X)$ to
$\Mod(\kappa(x))$ is faithful.
\end{lem}

\begin{proof}
Write $X$ as the limit of a filtered inverse system 
$(X_\lambda)_{\lambda \in \Lambda}$ as in Lemma~\ref{l:qafflim}.
Let $\overline{k}$ be an algebraic closure of $k$, and for each $\lambda$
write $\overline{\sM}_\lambda$ for the set of maximal points of 
$(X_\lambda)_{\overline{k}}$.
The group
\begin{equation*}
\Gamma = G(\overline{k}) \rtimes \Gal(\overline{k}/k)
\end{equation*}
acts on $\overline{\sM}_\lambda$,
and the quotient $\overline{\sM}_\lambda/\Gal(\overline{k}/k)$ by the subgroup 
$\Gal(\overline{k}/k)$ may be identified with the set $\sM_\lambda$
of maximal points of $X_\lambda$.
Let $\overline{\sM}_1$ be a non-empty $\Gamma$\nd subset of $\overline{\sM}_\lambda$.
If $\overline{\sM}_2$ is the complement of $\overline{\sM}_1$ in $\overline{\sM}_\lambda$,
then the complement of the closure of $\overline{\sM}_2$ in $(X_\lambda)_{\overline{k}}$
is an open super subscheme $\overline{U}_1$  of $(X_\lambda)_{\overline{k}}$ 
which contains $\overline{\sM}_1$ but no point of $\overline{\sM}_2$.
Further $\overline{U}_1$ is  stable under $\Gamma$, and hence descends to an open super 
$G$\nd subscheme $U_1$ of $X_\lambda$ which contains the image $\sM_1$ of $\overline{\sM}_1$
in $\sM_\lambda$ but no point of the image $\sM_2$ of $\overline{\sM}_2$. 
Since $\Mod_{G,\varepsilon}(X_\lambda)$ is integral by Lemma~\ref{l:domfaith},
$U_1$ is dense in $X_\lambda$ by Lemma~\ref{l:schdense}.
Thus $\sM_2$ and hence $\overline{\sM}_2$ is empty, so that $\sM_1 = \overline{\sM}_\lambda$.
This shows that $\Gamma$ acts transitively on $\overline{\sM}_\lambda$.
Since $(X_\lambda)_{\overline{k}} \to (X_\mu)_{\overline{k}}$ 
is dominant and compatible with the action of $\Gamma$ for $\lambda \ge \mu$, 
it follows that it sends $\overline{\sM}_\lambda$ to $\overline{\sM}_\mu$.
Hence $X_\lambda \to X_\mu$ sends $\sM_\lambda$ to $\sM_\mu$, and the $\sM_\lambda$
form a filtered inverse system of finite non-empty subsets of the $X_\lambda$.

By Tychonoff's theorem, the set $\lim_\lambda \sM_\lambda$ is non-empty.
Let $(x_\lambda)$ be an element.
Then there is a (unique) point $x$ of $X$ which lies above $x_\lambda$ in $X_\lambda$
for each $\lambda$.
By Lemma~\ref{l:limopensub}, any non-empty open super $G$\nd subscheme $U$ of $X$
contains the inverse image in $X$ of a non-empty open super $G$\nd subscheme $U_\lambda$
of some $X_\lambda$.
Since $U_\lambda$ contains $x_\lambda$ by Lemma~\ref{l:schdense},
it follows that $U$ contains $x$.

The final statement follows from Lemma~\ref{l:pullepi} by taking $Z = \Spec(\kappa(x))$.
\end{proof}

\begin{lem}\label{l:openGsubint}
Let $X$ be a quasi-affine super $(G,\varepsilon)$\nd scheme with $\Mod_{G,\varepsilon}(X)$ integral,
and $Y$ be a non-empty quasi-compact open super $G$\nd subscheme of $X$.
Then $\Mod_{G,\varepsilon}(Y)$ is integral, the restriction functor from 
$\Mod_{G,\varepsilon}(X)$ to $\Mod_{G,\varepsilon}(Y)$ is faithful, and the induced homomorphism
from $\kappa(\Mod_{G,\varepsilon}(X))$ to $\kappa(\Mod_{G,\varepsilon}(Y))$
is an isomorphism. 
\end{lem}

\begin{proof}
The faithfulness is clear from Lemmas~\ref{l:domfaith} and \ref{l:schdense}.
Write $j:Y \to X$ for the embedding.
By Lemma~\ref{l:quotsub}\ref{i:quotsubrep}, any morphism in $\Mod_{G,\varepsilon}(Y)$ 
with source $\sO_Y$ may,
after composing with an appropriate monomorphism, be put into the form 
\begin{equation*}
f:\sO_Y \to j^*\sV
\end{equation*}
with $\sV$ in $\Mod_{G,\varepsilon}(X)$.
To prove the integrality and isomorphism statements,
it is then enough by the faithfulness to show that if $f \ne 0$ there exist 
morphisms $h$ and $f' \ne 0$ in $\Mod_{G,\varepsilon}(X)$ such that
\begin{equation}\label{e:fjh}
f \otimes j^*(h) = j^*(f')
\end{equation}
We have $f = j^*(f_0)$, where $f_0:\sO_X \to j_*j^*\sV$ corresponds under adjunction to $f$.
Pulling back $f_0$ along the unit $\eta_{\sV}$ gives a cartesian square
\begin{equation*}
\xymatrix{
\sW_1 \ar_{u}[d] \ar^{f_1}[r] & \sV \ar^{\eta_{\sV}}[d] \\
\sO_X \ar^{f_0}[r] & j_*j^*\sV
}
\end{equation*}
in $\MOD_{G,\varepsilon}(X)$, with $j^*(\eta_{\sV})$ the identity and hence $j^*(u)$ an isomorphism.
Then $j^*(f_1) \ne 0$ and hence $f_1 \ne 0$, so that by Lemma~\ref{l:quotsub}\ref{i:quotsubmod} 
there is a $v:\sW \to \sW_1$ with $\sW$ in $\Mod_{G,\varepsilon}(X)$ such that $f_1 \circ v \ne 0$ in 
$\Mod_{G,\varepsilon}(X)$.
If $h = u \circ v$ and $f' = f_1 \circ v$, then \eqref{e:fjh} is satisfied because 
$f \otimes j^*(h) = f \circ j^*(h)$.
\end{proof}

Let $X$ be a quasi-affine super $(G,\varepsilon)$\nd scheme
with $\Mod_{G,\varepsilon}(X)$ integral.
The endomorphism $k$\nd algebra of $\I$ in $\Mod_{G,\varepsilon}(X)$
is given by
\begin{equation*}
\End_{G,\sO_X}(\sO_X) = H^0_G(X,\sO_X).
\end{equation*}
For $Y$ as in Lemma~\ref{l:openGsubint} we then have a commutative square
\begin{equation*}
\xymatrix{
H^0_G(X,\sO_X) \ar[d] \ar[r] & H^0_G(Y,\sO_Y) \ar[d] \\
\kappa(\Mod_{G,\varepsilon}(X)) \ar^{\sim}[r] & \kappa(\Mod_{G,\varepsilon}(Y))
}
\end{equation*}
with the arrows injective.
As $Y$ varies, such squares form a filtered system with the left arrow fixed.
It can be seen as follows that the homomorphism
\begin{equation}\label{e:funfieldiso}
\colim_{Y \subset X, \; Y \ne \emptyset}H^0_G(Y,\sO_Y) \to \kappa(\Mod_{G,\varepsilon}(X))
\end{equation}
it defines is an isomorphism, where the colimit is over the filtered system of non-empty quasi-compact open super 
$G$\nd subschemes $Y$ of $X$.
Given $\alpha$ in $\kappa(\Mod_{G,\varepsilon}(X))$, it is to be shown that
there exists a $Y$ such that the image of $\alpha$ in 
$\kappa(\Mod_{G,\varepsilon}(Y))$ is the image of an element of $H^0_G(Y,\sO_Y)$.
We have $\alpha = h/f$ for
morphisms $f,h:\sV \to \sO_X$ in $\Mod_{G,\varepsilon}(X)$ with $f \ne 0$ and 
$f \otimes h = h \otimes f$.
The closed super subscheme of $X$ defined by the ideal $\Img f$ of $\sO_X$ is a $G$\nd subscheme,
and we may take for $Y$ its complement.
Indeed $Y$ so defined is quasi-compact and by Lemma~\ref{l:rednil} it is non-empty, 
and since the restriction of $f$ to $Y$
is an epimorphism, the image of $\alpha$ in $\kappa(\Mod_{G,\varepsilon}(Y))$ has the required property 
by Lemma~\ref{l:torsfrac} with $\sA = \MOD_{G,\varepsilon}(Y)$.

\begin{lem}\label{l:redmono}
Let $f:X' \to X$ be a super schematically dominant morphism of super $k$\nd schemes of finite type
which induces an isomorphism from $f^{-1}(X_{\mathrm{red}})$ to $X_{\mathrm{red}}$.
Then $f$ is an isomorphism.
\end{lem}

\begin{proof}
Write $\sN$ and $\sN'$ for the nilradicals of $X$ and $X'$,
and 
\[
\alpha:\sO_X \to f_*\sO_{X'}
\]
for the monomorphism induced by $f$.
Then $\sN'$ contains the inverse image of $\sN$ in $\sO_{X'}$,
and hence coincides with it, because $f^{-1}(X_{\mathrm{red}})$ is reduced.
Thus $X'{}\!_{\mathrm{red}} = f^{-1}(X_{\mathrm{red}})$, so that 
$f_{\mathrm{red}}:X'{}\!_{\mathrm{red}} \to X_{\mathrm{red}}$ is an isomorphism and $f$ is a
homeomorphism.
It remains to show that $\alpha$ is an isomorphism.
We show by induction on $n$ that for $n \ge 1$
\begin{align}
\label{e:fNfNn}
f_*\sN' & = \alpha(\sN) + (f_*\sN')^n \\ 
\label{e:fOfNn}
f_*\sO_{X'} & = \alpha(\sO_X) + (f_*\sN')^n. 
\end{align}
Since $f_*\sN'$ is a nilpotent ideal of $f_*\sO_{X'}$, it will follow from \eqref{e:fOfNn} for $n$ large
that $f_*\sO_{X'} = \alpha(\sO_X)$, so that $\alpha$ is indeed an isomorphism.
That \eqref{e:fNfNn} holds for $n=1$ is immediate.
That \eqref{e:fOfNn} holds for $n=1$ follows from the fact that
since $f_{\mathrm{red}}$ is an isomorphism,
$\alpha$ induces an isomorphism from $\sO_X/\sN$ to $f_*\sO_{X'}/f_*\sN'$.
Suppose that \eqref{e:fNfNn} and \eqref{e:fOfNn} hold for $n=r$.
Then inserting \eqref{e:fOfNn} with $n=r$ into
\begin{equation*}
f_*\sN' = \alpha(\sN)f_*\sO_{X'}.
\end{equation*}
shows that \eqref{e:fNfNn} holds for $n = r+1$,
and inserting \eqref{e:fNfNn} with $n=r$ into \eqref{e:fOfNn} with $n=r$ shows that \eqref{e:fOfNn}
holds for $n = r+1$.
\end{proof}

By an equivalence relation on a super $k$\nd scheme  $X$ we mean a super subscheme $E$
of $X \times_k X$ with faithfully flat quasi-compact projections $E \to X$ such that
$E(S)$ is an equivalence relation on the set $X(S)$ for any super $k$\nd scheme $S$. 
By a quotient of $X$ by $E$ we mean a super $k$\nd scheme $Y$ together with a
faithfully flat quasi-compact $k$\nd morphism $X \to Y$ which coequalises the 
projections $E \to X$, such that the square with two sides $X \to Y$ and the other two sides the projections 
$E \to X$ is cartesian.
Such an $X \to Y$ is the coequaliser in the category of super local $k$\nd ringed spaces 
of the projections $E \to X$, and hence is unique up to unique isomorphism when it exists.
We write $X/E$ for the quotient of $X$ by $E$ when it exists.
Formation of quotients is compatible with extension of scalars.
The quotient $(T \times_k X)/(T \times_k E)$  exists if $X/E$ does, and may be identified with $T \times_k (X/E)$.
Suppose that $X/E$ exists.
If $E'$ is an equivalence relation on $X$ coarser than $E$, and if 
$E' \to X \times_k X$ is a closed immersion, then by faithfully flat descent 
of  super $k$\nd subschemes, $E$ descends along 
\begin{equation*}
X \times_k X \to (X/E) \times_k (X/E) 
\end{equation*}
to an equivalence relation $\overline{E'}$ on $X/E$.
The quotient $X/E'$ exists if and only if $(X/E)/\overline{E'}$ does, and they then coincide.

A super subscheme $T$ of a super $k$\nd scheme $X$ will be called a \emph{transversal} to an equivalence relation
$E$ on $X$ if the restriction of the first projection $\pr_1:E \to X$ to $\pr_2{}\!^{-1}(T)$ is an isomorphism.
With $X$ identified with $\pr_2{}\!^{-1}(T)$ by this isomorphism, the projection $X \to T$ is that onto the 
quotient of $X$ by $E$.

For $G$ a super $k$\nd group and $H$ a super $k$\nd subgroup of $G$, we denote as usual by $G/H$ the quotient,
when it exists, of $G$ by the equivalence relation defined by right translation by $H$.
The left action of $G$ on itself by left translation defines a structure of super $G$\nd scheme on $G/H$.
If $X$ is a super $G$\nd scheme and $x$ is a $k$\nd point of $X$ which is fixed by $H$, there is a unique morphism
$G/H \to X$ of $G$\nd schemes which sends the image in $G/H$ of the identity of $G$ to $x$.

\begin{lem}\label{l:superrediso}
Let $X$ and $X'$ be smooth super $k$\nd schemes, $f$ be a $k$\nd morphism from $X'$ to $X$,
and $x'$ be a $k$\nd point of $X'$.
Suppose that $f$ induces an isomorphism from 
$X'{}\!_{\mathrm{red}}$ to $X_{\mathrm{red}}$,
and an isomorphism from the tangent space of $X'$ at $x'$ to that of $X$ at $f(x')$. 
Then there exists an open super subscheme $X_0$ of $X$ containing $f(x')$ such that $f$
induces an isomorphism from $f^{-1}(X_0)$ to $X_0$.
\end{lem}

\begin{proof}
It is enough to show that the canonical morphism $\sO_X \to f_*\sO_{X'}$ is an isomorphism
in some neighbourhood of $f(x')$.
We may suppose that $X$ and $X'$ are pure of the same dimension $m|n$.
Write $\sN$ and $\sN'$ for the nilradicals of $\sO_X$ and $\sO_{X'}$.
Both $\sN^{n+1}$ and $\sN'{}^{n+1}$ are $0$.
Since $f_*$ is exact because $f$ is a homeomorphism,
it is thus enough to show that for every $i$ the canonical morphism
\begin{equation*}
h_i:\sN^i/\sN^{i+1} \to f_*(\sN'{}^i/\sN'{}^{i+1})
\end{equation*} 
is an isomorphism in a neighbourhood of $f(x')$.
Now $\sN/\sN^2$ is a locally free $\sO_{X_\mathrm{red}}$\nd module and $\sN'/\sN'{}^2$ a locally free 
$\sO_{X'{}\!_\mathrm{red}}$\nd module of rank $0|n$,
and $\sN^i/\sN^{i+1}$ and $\sN'{}^i/\sN'{}^{i+1}$ are the $i$th symmetric powers of 
$\sN/\sN^2$ and $\sN'/\sN'{}^2$ over $\sO_{X_\mathrm{red}}$ and $\sO_{X'{}\!_\mathrm{red}}$.
Since by hypothesis $f$ induces an isomorphism from $X'{}\!_\mathrm{red}$ to $X_\mathrm{red}$, 
we may suppose that $i=1$.
The morphism of super $k$\nd vector spaces induced by $h_1$ at $f(x')$ is the dual of that
induced by $f$ on the degree $1$ part of the tangent spaces at $x'$ and $f(x')$, and hence is an isomorphism.
The required result follows. 
\end{proof}


\begin{lem}
Let $G$ be an affine super $k$\nd group of finite type and $G_0$ be a closed super $k$\nd subgroup of $G$.
Then the quotient $G/G_0$ exists and is smooth over $k$.
Further the projection from $G$ to $G/G_0$ is smooth.
\end{lem}

\begin{proof}
Both $G$ and $G_0$ are smooth over $k$, of respective super dimensions $m|n$ and $m_0|n_0$, say.
There then exists a closed immersion $\bA^{0|n} \to G$.
Composing the multiplication $G \times_k G \to G$ of $G$ with the product of the embeddings of $\bA^{0|n}$ 
and $G_{\mathrm{red}}$ into $G$, we obtain a morphism
\begin{equation*}
\bA^{0|n} \times_k G_{\mathrm{red}} \to G
\end{equation*}
of right $G_{\mathrm{red}}$\nd schemes.
It is an isomorphism, as can be seen by passing to an algebraic closure and using Lemma~\ref{l:superrediso}.
Thus
\begin{equation*}
\overline{G} = G/G_0{}_{\mathrm{red}} = \bA^{0|n} \times_k (G_{\mathrm{red}}/G_0{}_{\mathrm{red}})
\end{equation*}
exists and is smooth over $k$ of dimension $(m-m_0)|n$, and $G \to \overline{G}$ is smooth
because it is the product of $\bA^{0|n}$ with a smooth morphism.
The equivalence relation $E$ on $G$ defined by the right action of $G_0$ on $G$ then descends to an equivalence
relation $\overline{E}$ on $\overline{G}$.
It is enough to show that 
$\overline{G}/\overline{E}$ exists and is smooth over $k$, and that the projection
$\overline{G} \to \overline{G}/\overline{E}$ is smooth.

Every open super subscheme of $\overline{G}$ is $\overline{E}$\nd saturated.
The isomorphism $G \times_k G_0 \iso E$ that sends $(g,g_0)$ to $(g,gg_0)$ is compatible with the
first projections onto $G$, so that $E$ is smooth over $G$ and hence over $k$.
Since $E$ is the inverse image of $\overline{E}$ under the surjective smooth morphism 
$G \times_k G \to \overline{G} \times_k \overline{G}$, it follows that $\overline{E}$ 
is smooth over $k$ of dimension $(m-m_0)|(n+n_0)$.
Further the projections $\overline{E} \to \overline{G}$ are smooth because their composites with 
the surjective morphism $E \to \overline{E}$ factor as the composite of smooth morphisms
$G \to \overline{G}$ and $E \to G$.

Given $n' \le n$, there corresponds to each $n'$\nd dimensional subspace $k$\nd vector subspace of $k^n$, or equivalently to each $k$\nd point $t$ of the Grassmannian $\mathbf{Gr}(n',n)$, a 
linearly embedded super subscheme $Z_t$ of $\bA^{0|n}$ isomorphic to $\bA^{0|n'}$.
Let $\sS$ be a finite set of closed points of $G_{\mathrm{red}}/G_0{}_{\mathrm{red}}$.
We now show that there exists 
a non-empty open subscheme $Y$ of $\mathbf{Gr}(n-n_0,n)$ with the following property:
for every $t$ in $Y(k)$ there is an open subscheme $U_t$ of $G_{\mathrm{red}}/G_0{}_{\mathrm{red}}$ containing 
$\sS$ such that $Z_t \times_k U_t$ is a transversal for the restriction 
of $\overline{E}$ to $\bA^{0|n} \times_k U_t$. 

Suppose first that $k$ is algebraically closed.
Then closed points of $G_{\mathrm{red}}/G_0{}_{\mathrm{red}}$ may be identified with $k$\nd points,
and also with $k$\nd points of $\overline{G}$ or $\overline{E}$.
If $x$ is such a point, then $\overline{E}$ induces a linear equivalence relation on the tangent space 
$V_x$ of $\overline{G}$, given by the tangent space $T_x$ of $\overline{E}$ at $x$, regarded
as a super $k$\nd vector subspace of $V_x \oplus V_x$.
Thus $T_x$ is the inverse image of a super subspace $W_x$ of $V_x$ of dimension $0|n_0$
under the subtraction homomorphism $V_x \oplus V_x \to V_x$.
Then $|W_x|$ is a subspace of dimension $n_0$ of the odd part
$k^n$ of $V_x$.  
Now take for $Y$ the open subscheme of $\mathbf{Gr}(n-n_0,n)$ parametrising the $(n-n_0)$\nd dimensional 
subspaces of $k^d$ complementary to $|W_x|$ at each $x$ in $\sS$.
Then for $t$ in $Y(k)$, the morphism
\begin{equation*}
f:\pr_2{}\!^{-1}(Z_t \times (G_{\mathrm{red}}/G_0{}_{\mathrm{red}})) \to 
\bA^{0|n} \times (G_{\mathrm{red}}/G_0{}_{\mathrm{red}}) 
\end{equation*}
defined by restricting the projection $\pr_1:\overline{E} \to \overline{G}$ to the inverse image
along $\pr_2$ induces an isomorphism on the tangent space at each $x$ in $\sS$.
The required $U_t$ thus exists by Lemma~\ref{l:superrediso} applied to $f$.

For arbitrary $k$ with algebraic closure $\overline{k}$,
applying the algebraically closed case with $k$, $G$, $G_0$, and $\sS$ replaced by 
$\overline{k}$, $G_{\overline{k}}$, $G_0{}_{\overline{k}}$, and the inverse image $\sS'$ of $\sS$
in $G_{\mathrm{red}}/G_0{}_{\mathrm{red}}$,
we obtain a non-empty open subscheme $Y'$ of $\mathbf{Gr}(n-n_0,n)_{\overline{k}}$ and for each
$\overline{k}$\nd point $t'$ of $Y'$ over $\overline{k}$ an open subscheme $U'{}\!_{t'}$ of
$(G_{\mathrm{red}}/G_0{}_{\mathrm{red}})_{\overline{k}}$ containing $\sS'$, with properties 
similar to the above. 
Replacing $Y'$ by the intersection of its finite set of conjugates under $\Gal(\overline{k}/k)$,
we may assume that $Y' = Y_{\overline{k}}$ for some $Y$.
For $t$ in $Y(k)$, the intersection of the conjugates of $U'{}\!_t$ 
descends to the required $U_t$.

Since every non-empty open subscheme of a Grassmannian has a $k$\nd point,
taking sets $\sS $ consisting of a single closed point shows  that
$G_{\mathrm{red}}/G_0{}_{\mathrm{red}}$ may be covered by open subschemes $U$ such that the restriction 
of $\overline{E}$ to $\bA^{0|n} \times_k U$ has a transversal $Z \times_k U$ with $Z$
a closed super subscheme of $\bA^{0|n}$ isomorphic to $\bA^{0|n-n_0}$. 
For such a $U$ and $Z$ and any open subscheme $U_0$ of $U$, the restriction of  
$\overline{E}$ to $\bA^{0|n} \times_k U_0$ then has a transversal $Z \times_k U_0$.
Thus the quotients $Z \times_k U$ of the $\bA^{0|n} \times_k U$ patch together to give the required quotient
$\overline{G}/\overline{E}$.
The smoothness over $k$ of $\overline{G}/\overline{E}$ follows from that of the $Z \times_k U$, 
and the smoothness of $\overline{G} \to \overline{G}/\overline{E}$ from the fact that the projections
$\bA^{0|n} \times_k U \to Z \times_k U$ are isomorphic to pullbacks of $\pr_2:\overline{E} \to \overline{G}$.
\end{proof}

\begin{lem}\label{l:homogGsub}
Let $X$ be a quasi-affine super $(G,\varepsilon)$\nd scheme of finite type.
Suppose that $\Mod_{G,\varepsilon}(X)$ is integral with $\kappa(\Mod_{G,\varepsilon}(X)) = k$.
Then $X$ has a unique non-empty open homogeneous super $G$\nd subscheme $X_0$.
Every non-empty open super $G$\nd subscheme of $X$ contains $X_0$.
\end{lem}

\begin{proof}
It is enough to prove that a non-empty open homogeneous super $G$\nd subscheme $X_0$ of $X$ exists:
for every non-empty open super $G$\nd subscheme $Y$ of $X$ the  open super $G$\nd subscheme 
$X_0 \cap Y$ of $X_0$ will be non-empty by Lemma~\ref{l:pointfaith}, 
and the final statement and hence the uniqueness of $X_0$ will follow.
After replacing $G$ by a quotient, we may suppose that $G$ is of finite type.

Let $k'$ be an extension of $k$.
By Lemma~\ref{l:quotsub}\ref{i:quotsubrep} and Lemma~\ref{l:ext}\ref{i:extint}, the hypotheses
on $X$ hold with $X$ and $k$ replaced by $X_{k'}$ and $k'$.
Suppose that $X_{k'}$ has a non-empty open homogeneous super $G_{k'}$\nd subscheme $X'{}\!_0$.
Then by the uniqueness, for any extension $k''$ of $k$ the inverse images of $X'{}\!_0$ under the
$k$\nd morphisms $X_{k''} \to X_{k'}$ induced by any two $k$\nd homomorphisms $k' \to k''$ coincide.  
Thus $X'{}\!_0$ descends to an open super subscheme $X_0$ of $X$, necessarily a 
homogeneous super $G$\nd subscheme.

Let $x$ be a point of $X$.
If $x'$ is the $\kappa(x)$\nd rational point of $X_{\kappa(x)}$ above $x$, 
then $\omega_{X,x}$ factors as
\[
\Mod_{G,\varepsilon}(X) \to \kappa(x) \otimes_k \Mod_{G,\varepsilon}(X) \to
\Mod_{G_{\kappa(x)},\varepsilon}(X_{\kappa(x)}) \xrightarrow{\omega_{X_{\kappa(x)},x'}} \Mod(\kappa(x))
\]
with the second arrow fully faithful. 
By Lemma~\ref{l:pointfaith}, an $x$ exists with $\omega_{X,x}$ faithful.
The composite of the second two arrows is then faithful by 
Lemma~\ref{l:ext}\ref{i:extfaith}, and hence $\omega_{X_{\kappa(x)},x'}$ is faithful 
by Lemma~\ref{l:quotsub}\ref{i:quotsubrep}.
Replacing $k$, $X$ and $x$ by $\kappa(x)$, $X_{\kappa(x)}$ and $x'$, we may thus suppose that $x$
is $k$\nd rational and $\omega_{X,x}$ is faithful.

Write $G_0$ for the stabiliser of $x$ under $G$. 
The unique morphism $(G,\varepsilon)$\nd schemes
\[
\varphi:G/G_0  \to X
\] 
that sends the base point $z$ of $G/G_0$ to $x$
defines a factorisation
\[
\Mod_{G,\varepsilon}(X) \xrightarrow{\varphi^*} \Mod_{G,\varepsilon}(G/G_0) \xrightarrow{\omega_{G/G_0,z}} \Mod(k).
\]
of $\omega_{X,x}$.
Thus $\varphi^*$ is faithful,
so that by Lemma~\ref{l:domfaith} $\varphi$ is super schematically dominant.

The stabiliser of the $k$\nd point $x$ of $X_{\mathrm{red}}$ under the action 
of $G_{\mathrm{red}}$ is $G_0{}_{\mathrm{red}}$, and 
\[
\varphi_{\mathrm{red}}:G_{\mathrm{red}}/G_0{}_{\mathrm{red}} = (G/G_0)_{\mathrm{red}} \to X_{\mathrm{red}}
\]
is the morphism of $G_{\mathrm{red}}$\nd schemes that sends the base point to $x$.
Since $\varphi_{\mathrm{red}}$ is dominant, it factors by homogeneity of $G_{\mathrm{red}}/G_0{}_{\mathrm{red}}$
as an isomorphism onto an open super $G_{\mathrm{red}}$\nd subscheme $X_1$ of $X_{\mathrm{red}}$ 
followed by the embedding.
Then 
\[
X_1 = X_0{}_{\mathrm{red}}
\]
for an open super $G$\nd subscheme $X_0$ of $X$,
and $\varphi$ factors as
\[
\varphi_0:G/G_0 \to X_0
\]
followed by the embedding.
The morphism from $\varphi_0{}\!^{-1}(X_0{}_{\mathrm{red}})$ to $X_0{}_{\mathrm{red}}$
induced by $\varphi_0$ is a morphism of $G_{\mathrm{red}}$\nd schemes, and hence is an isomorphism
because $X_0{}_{\mathrm{red}} = X_1$ is homogeneous while the fibre $G_0/G_0$ of $\varphi_0$ above $x$
is $\Spec(k)$.
Since $\varphi_0$ is super schematically dominant,
it is thus an isomorphism by Lemma~\ref{l:redmono} with $f = \varphi_0$.
\end{proof}

Let $X$ be a super $(G,\varepsilon)$\nd scheme of finite type, $S$ be a $k$\nd scheme on which $G$ acts trivially, 
$X \to S$ be a morphism of super $G$\nd schemes, and $s$ be a point of $S$.
Write $k'$ for $\kappa(s)$ and $X'$ for the fibre $X_s$ of $X$ above $s$.
Then $X'$ is a super $(G_{k'},\varepsilon)$\nd scheme of finite type.
If $U'$ is an open super $G_{k'}$\nd subscheme of $X'$, 
then the reduced scheme on the complement $Z'$ of $U'$ is a closed 
$(G_{k'})_\mathrm{red}$\nd subscheme, and hence a closed $G_\mathrm{red}$\nd subscheme, of $X'$.
The reduced subscheme of $X$ on the closure $Z$ of $Z'$ is then a closed $G_\mathrm{red}$\nd subscheme of $X$,
and its complement is an open super $G$\nd subscheme of $X$ with $U \cap X' = U'$.
Thus every open super $G_{k'}$\nd subscheme of $X'$ is the intersection of $X'$ 
with an open super $G$\nd subscheme of $X$.

Suppose now that $S$ is integral and that $s$ is its generic point.
Then $X'$ is the intersection of the family $(X_\lambda)_{\lambda \in \Lambda}$ of
inverse images under $X \to S$ of the non-empty open subschemes of $S$.
If we write $j_\lambda$ for the embedding $X_\lambda \to X$ and $j$ for $X' \to X$, 
it can be seen as follows that the canonical homomorphism
\begin{equation*}
\colim_{\lambda \in \Lambda} H^0_G(X_\lambda,j_\lambda{}\!^*\sV) \to H^0_G(X',j^*\sV) = H^0_{G_{k'}}(X',j^*\sV)
\end{equation*}
is an isomorphism for every $\sV$ in $\MOD_{G,\varepsilon}(X)$.
It is enough to prove that the corresponding homomorphism where there is no group acting and $H^0_G$ 
is replaced by $H^0$ is an isomorphism: taking $G \times_k X$ and $X$  for $X$ and considering the pullbacks
along the projection and the action from $G \times_k X$ to $X$ will then give the required result for $H^0_G$.
We may assume that the $j_\lambda$ are affine.
Covering $X$ with affine open sets, we reduce first to the case where $X$ is quasi-affine, and finally to the case 
where $X$ is affine, which is clear.
By Lemmas~\ref{l:quotsub}\ref{i:quotsubrep} and \ref{l:openGsubint},
it follows that if $X$ is a quasi-affine $(G,\varepsilon)$\nd scheme of finite type 
with $\Mod_{G,\varepsilon}(X)$ integral,
then $\Mod_{G_{k'},\varepsilon}(X')$ is integral, the functor 
\begin{equation*}
\Mod_{G,\varepsilon}(X) \to \Mod_{G,\varepsilon}(X') = \Mod_{G_{k'},\varepsilon}(X')
\end{equation*}
induced by $X' \to X$ is faithful, and the induced homomorphism
\begin{equation}\label{e:genfieldiso}
\kappa(\Mod_{G,\varepsilon}(X)) \to \kappa(\Mod_{G_{k'},\varepsilon}(X'))
\end{equation}
an isomorphism.

\begin{lem}\label{l:smallequiv}
Let $X$ be a quasi-affine super $(G,\varepsilon)$\nd scheme of finite type
with $\Mod_{G,\varepsilon}(X)$ integral.
Denote by $A$ the $k$\nd algebra $H^0_G(X,\sO_X)$, 
and by $k'$ its field of fractions.
Then the following conditions are equivalent:
\begin{enumerate}
\renewcommand{\theenumi}{(\alph{enumi})}
\item\label{i:smallsub}
every non-empty open super $G$\nd subscheme of $X$ contains one of the form $X_f$ for some $f \ne 0$ in $A$;
\item\label{i:smallhomog}
the generic fibre of $X \to \Spec(A)$ is homogeneous as a super $G_{k'}$\nd scheme.
\end{enumerate}
When these conditions hold, the canonical homomorphism $k' \to \kappa(\Mod_{G,\varepsilon}(X))$ is an isomorphism.
\end{lem}

\begin{proof}
The generic fibre $X'$ of $X \to \Spec(A)$
is the intersection of the open super $G$\nd subschemes $X_f$ of $X$ for $f \ne 0$ in $A$.
Writing the push forward of $\sO_{X_f}$ along the embedding $X_f \to X$ as the filtered colimit
of copies of $\sO_X$ with transition morphisms given by powers of $f$ shows that
\begin{equation*}
H^0_G(X_f,\sO_{X_f}) = A_f,
\end{equation*}
where we have used the fact that $H^0_G(X,-)$ commutes with filtered colimits 
in $\MOD_{G,\varepsilon}(X)$, because $X$ is of finite type.

Suppose that \ref{i:smallsub} holds.
Then $k' \to \kappa(\Mod_{G,\varepsilon}(X))$ is an isomorphism because $\kappa(\Mod_{G,\varepsilon}(X))$
is by the isomorphism \eqref{e:funfieldiso} the filtered colimit of the $H^0_G(X_f,\sO_{X_f})$ for $f \ne 0$.
Thus
\begin{equation*}
k' \to \kappa(\Mod_{G_{k'},\varepsilon}(X'))
\end{equation*} 
is an isomorphism because \eqref{e:genfieldiso} is.
Further any non-empty open $G_{k'}$\nd subscheme of $X'$ coincides with $X'$, because as above
it contains one of the form $X_f \cap X' = X'$ for $f \ne 0$.  
Thus \ref{i:smallhomog} holds by Lemma~\ref{l:homogGsub}.

Conversely suppose that \ref{i:smallhomog} holds.
Then every non-empty open super $G_{k'}$\nd subscheme of $X'$ coincides with $X'$.
Let $Y$ be a non-empty open super $G$\nd subscheme of $X$.
Then $Y \cap X$ is non-empty by Lemma~\ref{l:pointfaith}, so that $Y$ contains $X'$.
Since $X'$ is the intersection of the $X_f$ for $f \ne 0$, there thus exists 
for each point $z$ of the complement $Z$ of $Y$ an $f \ne 0$ for which $X_f$ 
does not contain $z$.
It follows that there exists an $f \ne 0$ such that $X_f$ does not contain any maximal point of $Z$.
Then $X_f \cap Z = \emptyset$, so that $Y$ contains $X_f$. 
\end{proof}

\begin{lem}\label{l:smallexist}
Let $X$ be a quasi-affine super $(G,\varepsilon)$\nd scheme of finite type
with $\Mod_{G,\varepsilon}(X)$ integral.
Then $X$ has an open super $G$\nd subscheme $X_1$ such that the equivalent conditions of 
Lemma~\textnormal{\ref{l:smallequiv}} are satisfied with $X$ replaced by $X_1$.
\end{lem}

\begin{proof}
Let $A$ and $k'$ be as in Lemma~\ref{l:smallequiv},
and write $k''$ for $\kappa(\Mod_{G,\varepsilon}(X))$. 
By Lemma~\ref{l:pointfaith}, $\omega_{X,x}$ for some point $x$ of $X$
defines an embedding of $k''$
as a subextension of the extension $\kappa(x)$ of $k$.
Thus $k''$ is a finitely generated extension of $k$.
By Lemma~\ref{l:openGsubint}, we may after replacing $X$ by a non-empty open super $G$\nd subscheme
suppose that $A$ contains a set of generators for $k''$ over $k$.
Thus $k' = k''$,

The generic fibre $X'$ of $X \to \Spec(A)$ is a quasi-affine super $(G_{k'},\varepsilon)$\nd scheme
of finite type with $\Mod_{G_{k'},\varepsilon}(X')$ integral,
and since \eqref{e:genfieldiso} is an isomorphism we have $\kappa(\Mod_{G_{k'},\varepsilon}(X')) = k'$.
By Lemma~\ref{l:homogGsub}, $X'$ has thus a non-empty open homogeneous super $G_{k'}$\nd subscheme $X'{}\!_1$.
By the remarks preceding Lemma~\ref{l:smallequiv}, there is an open super $G$\nd subscheme $X_1$ of $X$
with $X_1 \cap X' = X'{}\!_1$. 
If we write $A_1$ for $H^0_G(X_1,\sO_{X_1})$, then by Lemma~\ref{l:openGsubint}, 
$A \to A_1$ is injective and $A_1$ has field of fractions $k'$.
Thus $X'{}\!_1$ is the generic fibre of $X_1 \to \Spec(A_1)$.
It follows that \ref{i:smallhomog} of Lemma~\ref{l:smallequiv} is satisfied with $X_1$ for $X$.
\end{proof}

Filtered limits exist in the category of local super $k$\nd ringed spaces: they are given by taking the
filtered limit of the underlying topological spaces, equipped with the filtered colimit of the of the 
pullbacks of the structure sheaves. 
Any filtered limit of super $k$\nd schemes with affine transition morphisms is a super $k$\nd scheme.
Since the forgetful functor from the category of super $(G,\varepsilon)$\nd schemes to the category of 
super $k$\nd schemes creates limits, any limit of super $(G,\varepsilon)$\nd schemes which exists 
has a canonical structure of super $(G,\varepsilon)$\nd scheme.

Let $X$ be a local super $k$\nd ringed space and $X_0$ be a topological subspace of $X$ with structure
sheaf $\sO_{X_0}$ the restriction of $\sO_X$ to $X_0$.
Then pullback of $\sO_X$\nd modules along the embedding $j:X_0 \to X$ 
coincides with pullback of the underlying sheaves of abelian groups.
In particular $j$ is flat. 
This applies in particular when $X_0$ is the intersection of a family of open
super $k$\nd ringed subspaces of $X$.

\begin{lem}\label{l:homog}
Let $X$ be a quasi-affine super $(G,\varepsilon)$\nd scheme of finite type
with $\Mod_{G,\varepsilon}(X)$ integral.
Then the intersection $X_0$ in the category of local super ringed $k$\nd spaces of the non-empty open 
super $G$\nd subschemes of $X$ is a super $(G,\varepsilon)$\nd scheme.
If $k_0$ denotes the extension $\kappa(\Mod_{G,\varepsilon}(X))$ of $k$ and $X_0$ is given the 
structure of $k_0$\nd scheme defined by the isomorphism \eqref{e:funfieldiso}, 
then $X_0$ is a non-empty quasi-affine homogeneous 
super $(G_{k_0},\varepsilon)$\nd scheme of finite type over $k_0$. 
\end{lem}

\begin{proof}
By Lemma~\ref{l:smallexist}, we may suppose after replacing $X$ by a non-empty open
super $G$\nd subscheme that the equivalent conditions \ref{i:smallsub} and \ref{i:smallhomog}
of Lemma~\ref{l:smallequiv} are satisfied.
With notation as in Lemma~\ref{l:smallequiv}, we then have $k' = k_0$,
and the required intersection is the generic fibre of $X \to \Spec(A)$ by \ref{i:smallsub}, 
while the final statement holds by \ref{i:smallhomog}.
\end{proof}

Let $X$ be the limit of a filtered inverse system  $(X_\lambda)_{\lambda \in \Lambda}$ of 
super $k$\nd schemes with affine transition morphisms. 
The assignment
\begin{equation}\label{e:colimfunctor}
(\sV_\lambda)_{\lambda \in \Lambda} \mapsto \colim_{\lambda \in \Lambda}
\pr_\lambda{}\!^*\sV_\lambda
\end{equation}
defines a functor to $\MOD(X)$ from the category of systems 
$(\sV_\lambda)_{\lambda \in \Lambda}$ above $(X_\lambda)_{\Lambda \in \Lambda}$
with $\sV_\lambda$ in $\MOD(X_\lambda)$.
The functor \eqref{e:colimfunctor} is exact: we may suppose that $X = \Spec(R)$ and the 
$X_\lambda = \Spec(R_\lambda)$ are affine, and if $\sV_\lambda$ is the 
$\sO_{X_\lambda}$\nd module associated to the $R_\lambda$\nd module $V_\lambda$, we have an 
isomorphism 
\begin{equation*}
\colim_{\lambda \in \Lambda}V_\lambda \iso 
\colim_{\lambda \in \Lambda} R \otimes_{R_\lambda} V_\lambda
\end{equation*}
natural in $(V_\lambda)_{\lambda \in \Lambda}$.

Suppose that the $X_\lambda$ are quasi-compact and quasi-separated, and that
$\Lambda$ has an initial object $\lambda_0$.
Write $q_\lambda:X_\lambda \to X_{\lambda_0}$ for the transition morphism.
Then for $\sV_0$ in $\Mod(X_{\lambda_0})$ and $\sW_0$ in $\MOD(X_{\lambda_0})$,
the pullback functors define an isomorphism
\begin{equation}\label{e:colimHomVW}
\colim_\lambda\Hom_{\sO_{X_\lambda}}(q_\lambda{}\!^*\sV_0,q_\lambda{}\!^*\sW_0) 
\iso \Hom_{\sO_X}(\pr_{\lambda_0}{}\!^*\sV_0,\pr_{\lambda_0}{}\!^*\sW_0).
\end{equation}
This can be seen by reducing after taking a finite affine open cover of $X_{\lambda_0}$ to the case 
where $X_{\lambda_0}$ is affine.
If further $(X_\lambda)_{\lambda \in \Lambda}$ is a system of $(G,\varepsilon)$\nd schemes and 
$\sV_0$ and $\sW_0$ are equivariant $(G,\varepsilon)$\nd modules we have an isomorphism
\begin{equation}\label{e:colimHomGVW}
\colim_\lambda\Hom_{G,\sO_{X_\lambda}}(q_\lambda{}\!^*\sV_0,q_\lambda{}\!^*\sW_0) 
\iso \Hom_{G,\sO_X}(\pr_{\lambda_0}{}\!^*\sV_0,\pr_{\lambda_0}{}\!^*\sW_0).
\end{equation}
This follows from \eqref{e:colimHomVW} and the similar isomorphism for 
$(G \times_k X_\lambda)_{\lambda \in \Lambda}$ 
because for example $\Hom_{G,\sO_X}(\sV,\sW)$ is the equaliser of two
appropriately defined homomorphisms from $\Hom_{\sO_X}(\sV,\sW)$ to 
$\Hom_{\sO_{G \times_k X}}(\pr_2{}\!^*\sV,\alpha{}\!^*\sW)$, where $\alpha$ is the action of $G$
on $X$.
Similarly, taking an appropriate equaliser shows that for $\sV$ in $\Mod_G(X)$
the functor $\Hom_{G,\sO_X}(\sV,-)$ preserves filtered colimits in $\MOD_G(X)$.
It follows that for $\sV_0$ in $\Mod_G(X_{\lambda_0})$ and 
$(\sW_\lambda)_{\lambda \in \Lambda}$ a system above
$(X_\lambda)_{\lambda \in \Lambda}$ with $\sW_\lambda$ in $\MOD_G(X_\lambda)$,
the pullback functors define an isomorphism
\begin{equation}\label{e:colimHomGVWl}
\colim_\lambda\Hom_{G,\sO_{X_\lambda}}(q_\lambda{}\!^*\sV_0,\sW_\lambda) 
\iso \Hom_{G,\sO_X}(\pr_{\lambda_0}{}\!^*\sV_0,\colim_\lambda \pr_\lambda{}\!^*\sW_\lambda).
\end{equation}
Indeed if $q_{\mu\lambda}:X_\mu \to X_\lambda$ is the transition morphism,
then since $\lambda \mapsto (\lambda,\lambda)$ is cofinal in the set of $(\lambda,\mu)$ with $\mu \ge \lambda$,
\eqref{e:colimHomGVWl} factors as an isomorphism to
\begin{equation*}
\colim_\lambda \colim_{\mu \ge \lambda} 
\Hom_{G,\sO_{X_\mu}}(q_\mu{}\!^*\sV_0,q_{\mu\lambda}{}\!^*\sW_\lambda)
\end{equation*}
followed by a colimit of isomorphisms of the form \eqref{e:colimHomGVW}.

Suppose now that the $X_\lambda$ are quasi-affine $(G,\varepsilon)$\nd schemes.
Let $\sV$ be an object in $\MOD_{G,\varepsilon}(X)$.
By Lemma~\ref{l:quotsub}\ref{i:quotsubmod} we may write $\sV$ as a cokernel
\begin{equation}\label{e:Vcoker}
V'{}\!_X \to V_X \to \sV \to 0
\end{equation}
for some $V$ and $V'$ in $\MOD_{G,\varepsilon}(k)$.
By \eqref{e:colimHomGVW}, for each pair of subobjects $W$ of $V$ and $W'$ of $V'$
in $\Mod_{G,\varepsilon}(k)$ such that $V'{}\!_X \to V_X$ sends $W'{}\!_X$ into $W_X$, 
there exist a $\lambda \in \Lambda$ and an $h:W'{}\!_{X_\lambda} \to W_{X_\lambda}$ 
in $\Mod_{G,\varepsilon}(X_\lambda)$ for which $\pr_\lambda{}\!^*(h)$ coincides modulo
the pullback isomorphisms with $W'{}\!_X \to W_X$.
The category $\Lambda'$ of quadruples $(W,W',\lambda,h)$ is then filtered,
$\Lambda' \to \Lambda$ given by $(W,W',\lambda,h) \mapsto \lambda$ is cofinal, 
and $\sV$ is the colimit over $\Lambda'$ of the $\pr_\lambda{}\!^*(\Coker h)$.
Thus $\sV$ is isomorphic to an object in the image of 
a functor \eqref{e:colimfunctor} with $\Lambda$ replaced by $\Lambda'$.
Similarly if $l$ is a morphism in $\MOD_{G,\varepsilon}(X)$, then starting from a
commutative diagram with exact rows of the form \eqref{e:Vcoker} 
and right vertical arrow $l$ shows that
$l$ is isomorphic to a morphism in the image of a functor \eqref{e:colimfunctor}
for an appropriate $\Lambda$.

Let $X$ be a homogeneous super $(G,\varepsilon)$\nd scheme and $j:Z \to X$ be a morphism of 
super $k$\nd schemes with $Z$ non-empty.
Then the functor $j^*$ from $\MOD_{G,\varepsilon}(X)$ to $\MOD(Z)$
is faithful and exact.
Indeed if $j_0$ and $j_1$ are the morphisms from $G \times_k Z$ to $X$
that send $(g,z)$ respectively to $j(z)$ and $gj(z)$, then the equivariant action of $G$
defines a natural isomorphism from $j_0{}\!^*$ to $j_1{}\!^*$, 
and $j_1$ is faithfully flat while $j_0 = j \circ \pr_2$ with $\pr_2$ faithfully flat.

\begin{lem}\label{l:pointexact}
Let $X$ be a quasi-affine super $(G,\varepsilon)$\nd scheme with $\Mod_{G,\varepsilon}(X)$ integral,
and $x$ be a point of $X$ which lies in every non-empty open super $G$\nd subscheme of $X$.
Then the functor from $\MOD_{G,\varepsilon}(X)$ to $\MOD(\kappa(x))$ defined by passing to the fibre 
at $x$ is exact.
\end{lem}

\begin{proof}
Suppose first that $X$ is of finite type.
Then $x$ is a point of the super scheme $X_0$ over $k_0$ of Lemma~\ref{l:homog},
and passage to the fibre at $x$ factors through pullback onto $X_0$.
Pullback onto $X_0$ is exact because the monomorphism $X_0 \to X$ is flat.
Since $X_0$ is a homogeneous $G_{k_0}$\nd scheme, passage to the fibre
at the point $x$ of $X_0$ is as above also exact.

To prove the general case, write $X$ as the limit of a filtered inverse system 
$(X_\lambda)_{\lambda \in \Lambda}$ as in Lemma~\ref{l:qafflim}.
If $j:\Spec(\kappa(x)) \to X$ is the morphism defined by $x$, it is enough to show
that $j^*$ preserves the kernel of any morphism $l$ in $\MOD_{G,\varepsilon}(X)$.
As above, we may after replacing $\Lambda$ if necessary suppose that $l$ is the
image under \eqref{e:colimfunctor} of a morphism of systems
\begin{equation*}
(l_\lambda)_{\lambda \in \Lambda}:(\sV'{}\!_\lambda)_{\lambda \in \Lambda}
\to (\sV_\lambda)_{\lambda \in \Lambda}
\end{equation*}
above $(X_\lambda)_{\lambda \in \Lambda}$.
Since $j^*$ is cocontinuous, the left arrow of the commutative square 
\begin{equation*}
\xymatrix{
j^*\Ker l \ar[r] & \Ker j^*(l) \\
\colim_\lambda j^*\pr_\lambda{}\!^* \Ker l_\lambda \ar[u] \ar[r] &
\colim_\lambda \Ker j^*\pr_\lambda{}\!^*(l_\lambda) \ar[u]
}
\end{equation*}
is an isomorphism by exactness of \eqref{e:colimfunctor},
and the right arrow by exactness of filtered colimits.
The bottom arrow is an isomorphism because
$j^*\pr_\lambda{}\!^* \simeq (\pr_\lambda \circ j)^*$ is exact 
by the case where $X$ is of finite type.
Thus the top arrow is an isomorphism.
\end{proof}

Let $X$ be a super $k$\nd scheme.
If $\sV$ and $\sW$ are $\sO_X$\nd modules, then the internal hom 
\begin{equation*}
\underline{\Hom}_{\sO_X}(\sV,\sW)
\end{equation*}
is the $\sO_X$\nd module with sections of degree $i$ above the open subset $U$ of $X$
the $\sO_U$\nd homomorphisms of degree $i$ from $\sV|U$ to $\sW|U$. 
Its formation commutes with restriction to open super subschemes.
When $\sV = \sW$, the internal hom
\begin{equation*}
\underline{\End}_{\sO_X}(\sV) = \underline{\Hom}_{\sO_X}(\sV,\sV)
\end{equation*}
is an $\sO_X$\nd algebra, with unit
\begin{equation}\label{e:endunit}
\sO_X \to \underline{\End}_{\sO_X}(\sV)
\end{equation}
given by $1_{\sV}$ and composition by that of $\sO_U$\nd endomorphisms of $\sV|U$.
The kernel of \eqref{e:endunit} is the annihilator of $\sV$, i.e.\ the ideal of $\sO_X$
with sections above $U$ those of $\sO_X$ that annihilate $\sV|U$. 
If $u:\sU \to \sO_X$ is a morphism of $\sO_X$\nd modules, then
\begin{equation}\label{e:utensVzero}
u \otimes \sV = 0:\sU \otimes_{\sO_X} \sV \to \sV
\end{equation}
if and only if $u$ factors through the annihilator of $\sV$.

We may identify $\underline{\Hom}_{\sO_X}(\sO_X,\sW)$ with $\sW$, and 
$\underline{\Hom}_{\sO_X}(\sO_X{}^{0|1},\sW)$ with the $\sO_X$\nd module $\Pi \sW$ 
given by interchanging the parities of $\sW$.

Let $h:X' \to X$ be a morphism of super $k$\nd schemes. 
Then there is a morphism
\begin{equation}\label{e:inthomlower}
\underline{\Hom}_{\sO_X}(\sV,\sW) \to h_*\underline{\Hom}_{\sO_{X'}}(h^*\sV,h^*\sW),
\end{equation}
natural in $\sV$ and $\sW$, which sends $u$ from $\sV|U$ to $\sW|U$ to 
$h^*(u)$ from $h^*\sV|h^{-1}(U)$ to $h^*\sW|h^{-1}(U)$.
By adjunction, we then have a morphism
\begin{equation}\label{e:inthomupper}
h^*\underline{\Hom}_{\sO_X}(\sV,\sW) \to \underline{\Hom}_{\sO_{X'}}(h^*\sV,h^*\sW)
\end{equation}
which is natural in $\sV$ and $\sW$. 
It is an isomorphism for $\sV = \sO_X{}^{m|n}$ and any $\sW$,
because when $\sV = \sO_X$, \eqref{e:inthomlower} is the unit $\sW \to h_*h^*\sW$ and hence
\eqref{e:inthomupper} is the identity of $h^*\sW$,
and when $\sV = \sO_X{}^{0|1}$, \eqref{e:inthomlower} is the unit $\Pi\sW \to h_*h^*\Pi\sW$ 
and \eqref{e:inthomupper} is the identity of $h^*\Pi\sW$.

If $\sW$ is a quasi-coherent $\sO_X$\nd module, and $\sV$ is locally on $X$ 
a cokernel of a morphism of $\sO_X$\nd modules $\sO_X{}^{m'|n'} \to \sO_X{}^{m|n}$,
then $\underline{\Hom}_{\sO_X}(\sV,\sW)$ is a quasi-coherent $\sO_X$\nd module,
and \eqref{e:inthomupper} is an isomorphism for $h$ flat.

Suppose now that $X$ has a structure of $(G,\varepsilon)$\nd scheme.
Let $\sV$ and $\sW$ be objects of $\MOD_{G,\varepsilon}(X)$, with the underlying
$\sO_X$\nd module of $\sV$ locally on $X$ the cokernel
of a morphism $\sO_X{}^{m'|n'} \to \sO_X{}^{m|n}$.
We define a $(G,\varepsilon)$\nd equivariant structure on the quasi-coherent
$\sO_X$\nd module $\underline{\Hom}_{\sO_X}(\sV,\sW)$ as follows:
if $\alpha$ and $\beta$ are the isomorphisms between
the pullbacks respectively of $\sV$ and $\sW$ along the projection and the action
of $G$ from $G \times_k X$ to $X$ defining the actions of $G$ on $\sV$ and $\sW$,
then the action of $G$ on $\underline{\Hom}_{\sO_X}(\sV,\sW)$ is given, modulo isomorphisms
of the form \eqref{e:inthomupper}, by $\Hom_{\sO_{G \times_k X}}(\alpha^{-1},\beta)$.
If $\sV = \sW$, then \eqref{e:endunit} is a morphism in $\MOD_{G,\varepsilon}(X)$,
and the annihilator of $\sV$ is a subobject of $\sO_X$ in $\MOD_{G,\varepsilon}(X)$.

\begin{lem}\label{l:pointtors}
Let $X$ be a quasi-affine super $(G,\varepsilon)$\nd scheme with $\Mod_{G,\varepsilon}(X)$ integral,
and $x$ be a point of $X$ which lies in every non-empty open super $G$\nd subscheme of $X$.
Then an object of $\MOD_{G,\varepsilon}(X)$ is a torsion object if and only if 
its fibre at $x$ is $0$.
\end{lem}

\begin{proof}
Write $j:\Spec(\kappa(x)) \to X$ for the morphism defined by $x$,
and let $\sV$ be an object of $\MOD_{G,\varepsilon}(X)$.
If $\sV$ is a torsion object, then $j^*\sV$ is a torsion object of $\MOD(\kappa(x))$
by Lemmas~\ref{l:adjtorspres} and \ref{l:pointfaith}, so that $j^*\sV = 0$.

Conversely suppose that $j^*\sV = 0$.
By Lemma~\ref{l:pointexact}, $j^*$ is exact.
Hence $j^*\sW = 0$ for every subobject $\sW$ of $\sV$.
By  Lemma~\ref{l:quotsub}\ref{i:quotsubmod}, $\sV$ is the filtered colimit
of subobjects which are quotients of objects $V_X$ with $V$ in $\Mod_{G,\varepsilon}(k)$.
To prove that $\sV$ is a torsion object, we may thus suppose that $\sV$ is such a quotient.
Then by Lemma~\ref{l:quotsub}\ref{i:quotsubmod}, $\sV$ is a cokernel \eqref{e:Vcoker}
with $V$ in $\Mod_{G,\varepsilon}(k)$ and $V'$ in $\MOD_{G,\varepsilon}(k)$.
If we write $V'$ as the filtered colimit of its subobjects $W'$ in $\Mod_{G,\varepsilon}(k)$,
then by cocontinuity $j^*$ sends the cokernel of some $W'{}\!_X \to V_X$ to $0$.
Thus we may suppose that $\sV$ is a cokernel \eqref{e:Vcoker} with both $V$ and $V'$
in $\Mod_{G,\varepsilon}(k)$.
Then $\underline{\End}_{\sO_X}(\sV)$ and the annihilator $\sJ$ of $\sV$ exist in 
$\MOD_{G,\varepsilon}(X)$.
If $h = j$ then \eqref{e:inthomupper} is an isomorphism because it 
is an isomorphism with $V_X$ or $V'{}\!_X$ for $\sV$ and $j^*$ is exact.
Thus $j^*\underline{\End}_{\sO_X}(\sV) = 0$.
By the exactness of $j^*$, it follows that $j^*\sJ \ne 0$ and hence $\sJ \ne 0$.
There is then by Lemma~\ref{l:quotsub}\ref{i:quotsubmod} a non-zero
$u:\sU \to \sO_X$ in $\Mod_{G,\varepsilon}(X)$ which factors through $\sJ$,
so that \eqref{e:utensVzero} holds.
Thus $\sV$ is a torsion object.
\end{proof}

Let $f:X \to X'$ be a super schematically dominant morphism of quasi-affine
super $(G,\varepsilon)$\nd schemes with $\Mod_{G,\varepsilon}(X)$ integral.
Then by Lemma~\ref{l:domfaith}, $f^*$ from $\Mod_{G,\varepsilon}(X')$ to $\Mod_{G,\varepsilon}(X)$ is faithful 
and hence regular.
By Lemmas~\ref{l:pointfaith}, \ref{l:pointexact} and \ref{l:pointtors}, 
$f^*$ from $\MOD_{G,\varepsilon}(X')$ to $\MOD_{G,\varepsilon}(X)$ sends 
isomorphisms up to torsion to isomorphisms up to torsion.
There is thus a unique tensor functor
\begin{equation*}
\overline{\MOD_{G,\varepsilon}(X')} \to \overline{\MOD_{G,\varepsilon}(X)}
\end{equation*}
compatible with $f^*$ and the projections.
If $f$ is \emph{affine}, it can also be seen as follows that $f^*$ sends isomorphisms
up to torsion to isomorphisms up to torsion.
We may identify $\MOD_{G,\varepsilon}(X)$ with the tensor category of $f_*\sO_X$\nd modules
in $\MOD_{G,\varepsilon}(X')$, and $f^*$ with $f_*\sO_X \otimes_{\sO_{X'}} -$.
By Lemma~\ref{l:regtorsfree}, $f_*\sO_X$ is torsion free in
$\MOD_{G,\varepsilon}(X')$, so that by Lemma~\ref{l:Rextpres} an $f_*\sO_X$\nd module 
is a torsion $f_*\sO_X$\nd module if it is a torsion object of $\MOD_{G,\varepsilon}(X')$.
The required result now follows from Lemma~\ref{l:tensisotors}.

Let $X$ be a quasi-affine super $(G,\varepsilon)$\nd scheme with $\Mod_{G,\varepsilon}(X)$ integral,
and $k'$ be an extension of $k$.
By the isomorphism \eqref{e:funfieldiso} there is associated 
to any $k'$\nd point of $X$ which lies in every 
non-empty open super $G$\nd subscheme of $X$ a $k$\nd homomorphism
\begin{equation*}
\kappa(\Mod_{G,\varepsilon}(X)) \to k'.
\end{equation*}
This is the same $k$\nd homomorphism as that associated to the $k$\nd tensor functor 
from $\Mod_{G,\varepsilon}(X)$ to $\Mod(k)$ which is faithful by Lemma~\ref{l:pointfaith}.
We write
\begin{equation*}
X(k')_\rho \subset X(k')
\end{equation*}
for the set of $k'$\nd points lying in every 
non-empty open super $G$\nd subscheme of $X$ with associated $k$\nd homomorphism $\rho$.
The action of $G(k')$ on $X(k')$ sends $X(k')_\rho$ to itself.

The isomorphisms of the form \eqref{e:funfieldiso}, as well as passage to the associated $k$\nd homomorphism,
are compatible with super schematically dominant morphisms $X' \to X$ of quasi-affine super 
$(G,\varepsilon)$\nd schemes
with $\Mod_{G,\varepsilon}(X')$ integral.

\begin{lem}\label{l:kbarpoints}
Let $X$ be a quasi-affine super $(G,\varepsilon)$\nd scheme with $\Mod_{G,\varepsilon}(X)$ integral,
$k'$ be an algebraically closed extension of $k$,
and $\rho$ be a $k$\nd homomorphism from $\kappa(\Mod_{G,\varepsilon}(X))$ to $k'$.
Then $X(k')_\rho$ is non-empty,
and $G(k')$ acts transitively on it.
\end{lem}

\begin{proof}
Write $X$ as a the limit of a filtered inverse system $(X_\lambda)_{\lambda \in \Lambda}$
of quasi-affine $(G,\varepsilon)$\nd schemes of finite type as in Lemma~\ref{l:qafflim}.
For each $\lambda$ there is a $k$\nd quotient of finite type of $G$ through which it acts on $X_\lambda$.
If $\Lambda'$ is the set of pairs $(\lambda,G')$ with $\lambda$ in $\Lambda$ and 
$G'$ a $k$\nd quotient of $G$ of finite type
through which $G$ acts on $X_\lambda$, where $(\lambda,G') \le (\lambda',G'')$ when $\lambda \le \lambda'$
and  $G \to G'$ factors through $G \to G''$, then $(\lambda,G') \mapsto \lambda$ from 
$\Lambda'$ to $\Lambda$ is cofinal, as is $(\lambda,G') \mapsto G'$ from $\Lambda'$ to 
$k$\nd quotients of $G$ of finite type with the reverse of its natural order.
Replacing $\Lambda$ by $\Lambda'$, we may assume that there exists an inverse system 
$(G_\lambda)_{\lambda \in \Lambda}$
of $k$\nd quotients of $G$ of finite type with limit $G$ such that the action of $G$ on
$X_\lambda$ factors through $G_\lambda$.

Write $k_0$ and $k_\lambda$ for $\kappa(\Mod_{G,\varepsilon}(X))$ and 
$\kappa(\Mod_{G,\varepsilon}(X_\lambda))$,
and $\rho_\lambda:k_\lambda \to k'$ for the composite of $\rho$ with $k_\lambda \to k_0$.
By Lemma~\ref{l:limopensub}, a $k'$\nd point $(x_\lambda)$ of $X$ lies in every non-empty
open super $(G,\varepsilon)$\nd subscheme of $X$ if and only if
$x_\lambda$ lies in every non-empty open super $(G,\varepsilon)$\nd subscheme of $X_\lambda$ for
each $\lambda$.
We thus have an equality
\begin{equation*}
X(k')_\rho = \lim_\lambda X_\lambda(k')_{\rho_\lambda}
\end{equation*}
of $G(k')$\nd sets,
because the isomorphisms of the form \eqref{e:funfieldiso} are compatible with the projections  
and by isomorphisms of the form \eqref{e:colimHomGVW}, $k_0$ is the filtered 
colimit of the $k_\lambda$.
By Lemma~\ref{l:homog}, the intersection $X_\lambda{}_0$ of the non-empty open super 
$(G,\varepsilon)$\nd subschemes of $X_\lambda$, regarded as super $k_\lambda$\nd scheme
by means of the isomorphism \eqref{e:funfieldiso} with $X_\lambda$ for $X$, is a non-empty homogeneous
super $(G_{k_\lambda},\varepsilon)$\nd scheme.
If $X'{}\!_\lambda$ is obtained from $X_\lambda{}_0$ by extension of scalars along
$\rho_\lambda$, we then have an equality 
\begin{equation*}
X_\lambda(k')_{\rho_\lambda} = X'{}\!_\lambda(k')_{k'}
\end{equation*}
of $G(k')$\nd sets, compatible with the transition maps, where 
$X'{}\!_\lambda(k')_{k'}$ denotes the set of $k'$ points over $k$.
It is thus to be shown that $\lim_\lambda X'{}\!_\lambda(k')_{k'}$ is non-empty, and that
$G(k') = G_{k'}(k')_{k'}$ acts transitively on it.

Since the action of $G$ on $X_\lambda$ factors through $G_\lambda$, the action of $G_{k'}$ on
$X'{}\!_\lambda$ factors through $G_{\lambda}{}_{k'}$, and $X'{}\!_\lambda$ is a non-empty
homogeneous super $G_{\lambda}{}_{k'}$\nd scheme.
If
\begin{equation*}
U_\lambda = (G_{\lambda}{}_{k'})_{\mathrm{red}},
\end{equation*}
then $U_\lambda(k')_{k'} = G_{\lambda}{}_{k'}(k')_{k'}$
acts transitively on the non-empty set $X'{}\!_\lambda(k')_{k'}$.
Further if $H$ is the stabiliser of the $k'$\nd point $z$ of $X'{}\!_\lambda$ over $k'$,
then the stabiliser of the element $z$ of $X'{}\!_\lambda(k')_{k'}$ is the group of
$k'$\nd points over $k'$ of the $k'$\nd subgroup $H_{\mathrm{red}}$ of $U_\lambda$.
The required result now follows from \cite[1.1.1]{O10} with $k'$ for $k$ and 
$X'{}\!_\lambda(k')_{k'}$ for $X_\lambda$.
\end{proof}

Let $Z$ be a super $k$\nd scheme.
By a \emph{super groupoid over $Z$} we mean a super $k$\nd scheme $K$ together
with a source morphism $d_1$ and a target morphism $d_0$ from $K$ to $Z$, an identity morphism $s_0$ from $Z$ to $K$,
and a composition morphism $\circ$ from $K \times_{{}^{d_1}Z^{d_0}} K$ to $K$, such that $\circ$ is associative with
$s_0$ its left and right identity, and has inverses (necessarily unique).
The morphism
\begin{equation*}
(d_0,d_1):K \to Z \times_k Z
\end{equation*}
defines a structure of super scheme over $Z \times_k Z$ on $K$.
The super groupoid $K$ over $Z$ will be called \emph{affine} if it is affine 
as super scheme over $Z \times_k Z$,
and \emph{transitive} if it is faithfully flat as a super scheme over $Z \times_k Z$. 

By a \emph{super groupoid with involution over $Z$} we mean a pair $(K,\varepsilon)$ with
$K$ a super groupoid over $Z$ and $\varepsilon:Z \to K$ a lifting of $(\iota_Z,1_Z):Z \to Z \times_k Z$ to $K$
with $\varepsilon$ and $\varepsilon \circ \iota_Z$ inverse to one another, 
such that conjugation by $\varepsilon$ acts as $\iota_K$ on $K$.

An \emph{action} of a super groupoid $K$ over $Z$ on a quasi-coherent $\sO_Z$\nd module $\sV$
is an isomorphism of $\sO_K$\nd modules
\begin{equation*}
\alpha:d_1{}\!^*\sV \iso d_0{}\!^*\sV
\end{equation*}
such that if $\alpha_v:\sV_{z_0} \iso \sV_{z_1}$ is the fibre of $\alpha$ at the point $v$
of $K$ above $(z_1,z_0)$, then 
\begin{equation*}
\alpha_{w \circ v} = \alpha_w \circ \alpha_v 
\end{equation*}
for $w$ above $(z_2,z_1)$.

Let $(K,\varepsilon)$ be a super groupoid with involution over $Z$.
We define a \emph{$(K,\varepsilon)$\nd module} as a quasi-coherent $\sO_Z$\nd module $\sV$ together with 
an action $\alpha$ of $K$ on $\sV$ such that 
\begin{equation*}
\alpha_\varepsilon = \iota_{\sV}:\sV \iso \iota_Z{}\!^*\sV.
\end{equation*}
A morphism $\sV \to \sW$ of $(K,\varepsilon)$\nd modules is a morphism $f:\sV \to \sW$ of the 
undelying $\sO_Z$\nd modules for which the square formed by $d_1{}\!^*(f)$, $d_0{}\!^*(f)$, and the actions 
commutes. 
With tensor product that of the underlying $\sO_Z$\nd modules, $(K,\varepsilon)$\nd modules form a tensor 
category $\MOD_{K,\varepsilon}(Z)$.
A $(K,\varepsilon)$\nd module with underlying $\sO_Z$\nd module a vector bundle over $Z$
will also be called a \emph{representation of $(K,\varepsilon)$}.
We write $\Mod_{K,\varepsilon}(Z)$ for the full rigid tensor subcategory of $\MOD_{K,\varepsilon}(Z)$ 
consisting of the representations of $(K,\varepsilon)$.

Suppose that $Z$ has a structure of super $k_1$\nd scheme for an extension $k_1$ of $k$.
If $K$ is a groupoid over $Z$ such that $(d_0,d_1)$ factors through the super subscheme $Z \times_{k_1} Z$ of
$Z \times_k Z$, then $K$ may be regarded as a groupoid in the category of super $k_1$\nd schemes.
We then say that $K$ is a groupoid over $Z/k_1$.
In that case any lifting $\varepsilon$ as above is at the same time a lifting of $(\iota_Z,1_Z):Z \to Z \times_{k_1} Z$,
and the category $\MOD_{K,\varepsilon}(Z)$ is the same whether $K$ is
regarded as a groupoid over $Z$ or over $Z/k_1$.
If $Z = \Spec(k')$ for an extension $k'$ of $k$,
then a groupoid over $Z$ will also be called a groupoid over $k'$, and a groupoid over $Z/k_1$ 
a groupoid over $k'/k_1$.

Let $h:Z' \to Z$ be a morphism of super $k$\nd schemes.
A morphism over $h$ from a super groupoid $K'$ over $Z'$ to a super groupoid $K$ over $Z$ 
is a morphism $l:K' \to K$ such that
$h$ and $l$ are compatible with the source, target, identity and composition for $K$ and $K'$.
When $Z' = Z$ and $h = 1_Z$ we speak of a morphism over $Z$.
The \emph{pullback of $K$ along $h$} is the super groupoid 
\begin{equation*}
K \times_{Z \times_k Z} Z' \times_k Z'
\end{equation*}
over $Z'$ where the identity sends $z'$ to $(s_0(h(z')),z',z')$ and the composition sends 
$((w,z'{}\!_2,z'{}\!_1),(v,z'{}\!_1,z'{}\!_0))$ to $(w \circ v,z'{}\!_2,z'{}\!_0)$.
Together with the projection to $K$, it is universal among super groupoids over $Z'$ 
equipped with a morphism to $K$ over $h$.
The pullback along $h$ of a super groupoid with involution $(K,\varepsilon)$ over $Z$
is defined as $(K',\varepsilon')$ where $K'$ is the pullback of $K$ along $h$ and
$\varepsilon'$ sends $z'$ to $(\varepsilon(h(z')),\iota_{Z'}(z'),z')$.

Let $(K,\varepsilon)$ be a super groupoid with involution over $Z$ and   
$(K',\varepsilon')$ be a super groupoid with involution over $Z'$.
Given a morphism $l:K' \to K$ be of super groupoids over $h:Z' \to Z$ with 
$l \circ \varepsilon' = \varepsilon \circ h$,
we define as follows a tensor functor
\begin{equation*}
l^*:\MOD_{K,\varepsilon}(Z) \to \MOD_{K',\varepsilon'}(Z').
\end{equation*}
On the underlying $\sO_Z$\nd modules, $l^*$ is $h^*$.
If $\alpha$ is the action of $K$ on $\sV$, then modulo the pullback isomorphisms, 
the action of $K'$ on $h^*\sV$ is $l^*(\alpha)$.
When $K$ is a transitive affine groupoid over $Z$ and $(K',\varepsilon')$
is the pullback of $(K,\varepsilon)$ along $h$, it follows
from faithfully flat descent for quasi-coherent modules over a super $k$\nd scheme that
for $Z'$ non-empty, $l^*$ is an equivalence, and that it induces an equivalence from 
$\Mod_{K,\varepsilon}(Z)$ to $\Mod_{K',\varepsilon'}(Z')$.

Let $X$ be a super $(G,\varepsilon)$\nd scheme.
Then we have a super groupoid with involution
\begin{equation*}
(G \times_k X,\varepsilon \times_k X)
\end{equation*}
over $X$, where $d_0$ is the action, 
$d_1$ is the projection, the identity sends $x$ to $(1,x)$, and the composition
sends $((g',gx),(g,x))$ to $(g'g,x)$.
Further
\begin{equation*}
\MOD_{G,\varepsilon}(X) = \MOD_{G \times_k X,\varepsilon \times_k X}(X)
\end{equation*}
with a similar identification for $\Mod_{G,\varepsilon}(X)$ and $\Mod_{G \times_k X,\varepsilon \times_k X}(X)$.

Let $X$ be a quasi-affine super $(G,\varepsilon)$\nd scheme, $Z$ be a super $k$\nd scheme 
and $j:Z \to X$ be a $k$\nd morphism.
Write $X$ as the limit $\lim_\lambda X_\lambda$ with $(X_\lambda)_{\lambda \in \Lambda}$ as in 
Lemma~\ref{l:qafflim}. 
Denote by $(K,\varepsilon_0)$ the pullback of $(G \times_k X,G \times_k \varepsilon)$ along $j$ and by 
$(K_\lambda,\varepsilon_\lambda)$ the pullback of $(G \times_k X_\lambda,\varepsilon \times_k X_\lambda)$
along $j_\lambda = \pr_\lambda \circ j$.
We have a commutative square
\begin{equation}\label{e:GXKlambda}
\begin{gathered}
\xymatrix{
G \times_k X \ar_{G \times_k \pr_\lambda}[d] & K \ar_-{l}[l] \ar^{q_\lambda}[d] \\
G \times_k X_\lambda & K_\lambda \ar_-{l_\lambda}[l] 
}
\end{gathered}
\end{equation}
of groupoids, compatible with $\varepsilon \times_k X$, $\varepsilon \times_k X_\lambda$,
$\varepsilon_0$ and $\varepsilon_\lambda$, where $l$ and $l_\lambda$ are the projections and 
$q_\lambda$ is the fibre product of  
$G \times_k \pr_\lambda$ and $Z \times_k Z$ over $j_\lambda \times_k j_\lambda$.
Then
\begin{equation}\label{e:KlimKlambda}
K = \lim_\lambda K_\lambda.
\end{equation}
in the category of groupoids over $Z$,
with projections $q_\lambda$.

\begin{thm}\label{t:transaff}
Let $X$ be a quasi-affine super $(G,\varepsilon)$\nd scheme with $\Mod_{G,\varepsilon}(X)$ 
integral, $Z$ be a non-empty super $k$\nd scheme and $j:Z \to X$ be a morphism of
super $k$\nd schemes which factors through every non-empty open super $G$\nd subscheme of $X$.
Then if $Z$ is given the structure of $\kappa(\Mod_{G,\varepsilon}(X))$\nd scheme defined by
the isomorphism \eqref{e:funfieldiso}, the pullback along $j$ of the super groupoid 
$G \times_k X$ over $X/k$ is a
transitive affine super groupoid over $Z/\kappa(\Mod_{G,\varepsilon}(X))$.
\end{thm}

\begin{proof}
Write $k_0$ for $\kappa(\Mod_{G,\varepsilon}(X))$.
Suppose first that $X$ is of finite type.
Then by Lemma~\ref{l:homog} the intersection $X_0$ of the non-empty 
open super $G$\nd subschemes of $X$ with its structure $k_0$\nd scheme 
defined by the isomorphism \eqref{e:funfieldiso} is a homogeneous $G_{k_0}$\nd scheme.
Since the projection  $e:X_0 \to X$ is a monomorphism of super $k$\nd schemes 
underlying a morphism of super $G$\nd schemes, the pullback of $G \times_k X$ along $e$ is
\begin{equation*}
G \times_k X_0 = G_{k_0} \times_{k_0} X_0,
\end{equation*}
and hence is transitive affine over $X_0/k_0$.
The result follows, because $Z$ factors through $X_0$.

To prove the general case, write $X$ as the limit of a filtered inverse system
$(X_\lambda)_{\lambda \in \Lambda}$ as above. 
Let $K_\lambda$ be as above, and write $k_\lambda$ for 
$\kappa(\Mod_{G,\varepsilon}(X_\lambda))$. 
By the case where $X$ is of finite type, each $K_\lambda$ is transitive affine
over $X_\lambda/k_\lambda$.
We have
\begin{equation*}
k_0 = \colim_\lambda k_\lambda
\end{equation*}
by isomorphisms of the form \eqref{e:colimHomGVW}.
Thus $Z \times_{k_0} Z = \lim_\lambda Z \times_{k_\lambda} Z$.
The result now follows from \eqref{e:KlimKlambda}.
\end{proof}

Let $X$, $Z$ and $j$ be as in Theorem~\ref{t:transaff}.
Denote by $(K,\varepsilon_0)$ the pullback of $(G \times_k X,G \times_k \varepsilon)$ along $j$,
and by $l:K \to G \times_k X$ the projection.
Since pullback onto a point of $Z$
defines an equivalence on $\MOD_{K,\varepsilon_0}(Z)$,
it follows from Lemmas~\ref{l:pointexact} and \ref{l:pointtors} that the pullback tensor functor
\begin{equation*}
l^*:\MOD_{G,\varepsilon}(X) \to \MOD_{K,\varepsilon_0}(Z)
\end{equation*}
is exact, and that $\sV$ is a torsion object in $\MOD_{G,\varepsilon}(X)$ if and only if
$l^*\sV = 0$.
Thus $l^*$ factors uniquely through the projection onto $\overline{\MOD_{G,\varepsilon}(X)}$ 
as a $k$\nd tensor functor
\begin{equation}\label{e:MODbarK}
\overline{\MOD_{G,\varepsilon}(X)} \to \MOD_{K,\varepsilon_0}(Z),
\end{equation}
and \eqref{e:MODbarK} is faithful, exact and cocontinuous.

\begin{thm}\label{t:GKMODequiv}
The $k$\nd tensor functor \eqref{e:MODbarK} is an equivalence.
\end{thm}

\begin{proof}
Since \eqref{e:MODbarK} is faithful, it remains to prove that it is full 
and essentially surjective.
Suppose first that $X$ is of finite type.
Write $X_0$ for the intersection of the non-empty open super $G$\nd subschemes of $X$
as in  Lemma~\ref{l:homog},
and $e$ for the monomorphism $X_0 \to X$.
By Lemma~\ref{l:homog}, $l^*$ factors as 
\begin{equation*}
e^*:\MOD_{G,\varepsilon}(X) \to \MOD_{G,\varepsilon}(X_0)
\end{equation*} 
followed by a tensor equivalence.
Thus $e^*$ is exact, and $e^*\sV = 0$ if and only if $\sV$ is a torsion object.
Further $e^*$ factors as the projection onto $\overline{\MOD_{G,\varepsilon}(X)}$ followed 
by a $k$\nd tensor functor 
\begin{equation}\label{e:MODbarG}
\overline{\MOD_{G,\varepsilon}(X)} \to \MOD_{G,\varepsilon}(X_0),
\end{equation}
and \eqref{e:MODbarK} is \eqref{e:MODbarG} followed by a tensor equivalence.
Since $X_0$ is a topological subspace of $X$
with structure sheaf the restriction of $\sO_X$,
the counit for $e^*$ and $e_*$ is an isomorphism
\begin{equation*}
e^*e_* \iso \Id.
\end{equation*}
Thus \eqref{e:MODbarG} and \eqref{e:MODbarK} are essentially surjective, and surjective on
hom groups between objects of the form $\overline{e_*\sV_0}$.
If $\eta$ is the unit, then $e^*\eta$ is an isomorphism by the triangular identity.
Thus $\eta_{\sV}$ is an isomorphism up to torsion for every $\sV$,
and $\overline{\eta_{\sV}}$ is an isomorphism $\overline{\sV} \iso \overline{e_*e^*\sV}$.
It follows that \eqref{e:MODbarG} and \eqref{e:MODbarK} are full.

To prove that \eqref{e:MODbarK} is full and essentially surjective for arbitrary $X$,
write $X$ as the limit of a filtered system $(X_\lambda)_{\lambda \in \Lambda}$ as in 
Lemma~\ref{l:qafflim}, and let $K_\lambda$, $\varepsilon_\lambda$,
$l_\lambda$ and $q_\lambda$ be as in \eqref{e:GXKlambda}. 
Let $\sV$ and $\sW$ be objects of $\MOD_{G,\varepsilon}(X)$.
To prove that \eqref{e:MODbarK} defines a surjection
\begin{equation}\label{e:HomVWbar}
\Hom(\overline{\sV},\overline{\sW}) \to \Hom_{K,\sO_X}(l^*\sV,l^*\sW)
\end{equation}
we may suppose by \eqref{e:Vcoker} and the faithfulness and cocontinuity of 
\eqref{e:MODbarK} that $\sV = V_X$ with $V$ in $\Mod_{G,\varepsilon}(k)$.
Since $l^*(V_X)$ is of finite type in $\MOD_{K,\varepsilon_0}(Z)$,
and $\sW$ is by Lemma~\ref{l:quotsub}\ref{i:quotsubmod} the filtered colimit of 
its subobjects which are quotients of objects $W_X$ with $W$ in $\Mod_{G,\varepsilon}(k)$,
we may further suppose that $\sW$ is such a quotient.
Then $\sW$ is the cokernel of a morphism $f':W'{}\!_X \to W_X$ for some $W'$ in 
$\MOD_{G,\varepsilon}(k)$. 
Writing $W'$ as the filtered colimit of its subobjects $W''$ in $\Mod_{G,\varepsilon}(k)$
shows that $l^*\sW$ is for some $W''$ the cokernel of 
$l^*f''$ with
\begin{equation*}
f'':W''{}\!_X \to W_X
\end{equation*}
the restriction of $f'$ to $W''{}\!_X$.
If $p$ is the canonical morphism from $\Coker f''$ to $\sW = \Coker f'$, then $l^*p$ is
an isomorphism, so that by Lemmas~\ref{l:pointexact} and \ref{l:pointtors} 
$p$ is an isomorphism up to torsion and $\overline{p}$ is an isomorphism.
Thus we may suppose further that $\sW = \Coker f''$.
By \eqref{e:colimHomGVW}, $f''$ descends to a morphism 
\begin{equation*}
f_0:W''{}\!_{X_{\lambda_0}} \to W_{X_{\lambda_0}}
\end{equation*}
for some $\lambda_0 \in \Lambda$.
If we replace $\Lambda$ by $\lambda_0/\Lambda$, we may suppose finally by taking 
$\sV_0 = V_{X_{\lambda_0}}$ and $\sW_0 = \Coker f_0$ that $\Lambda$ has an
initial object $\lambda_0$ and $\sV = (\pr_{\lambda_0})^*\sV_0$ and
$\sW = (\pr_{\lambda_0})^*\sW_0$ for $\sV_0$ and $\sW_0$ in $\MOD_{G,\varepsilon}(X_{\lambda_0})$
with $(l_{\lambda_0})^*\sV_0$ and $(l_{\lambda_0})^*\sW_0$ in 
$\Mod_{K_{\lambda_0},\varepsilon}(X_{\lambda_0})$.

Let $h:l^*\sV \to l^*\sW$ be an element of the target of \eqref{e:HomVWbar}.
Then if $\sV_\lambda$ and $\sW_\lambda$ are the pullbacks of $\sV_0$ and $\sW_0$
along $X_\lambda \to X_{\lambda_0}$, restricting
to a non-empty affine open subset of $Z$ shows that there exists a $\lambda \in \Lambda$
and a morphism
\begin{equation*}
h_1:l_\lambda{}\!^*\sV_\lambda \to l_\lambda{}\!^*\sW_\lambda
\end{equation*}
such that $q_\lambda{}\!^*h_1$ coincides, modulo pullback isomorphisms, with $h$. 
By the case where $X$ is of finite type, $h_1$ is the image of a morphism
$\overline{\sV_\lambda} \to \overline{\sW_\lambda}$.
Thus there exist isomorphisms up to torsion $s:\sV'{}\!_\lambda \to \sV_\lambda$
and $r:\sW_\lambda \to \sW'{}\!_\lambda$ and a morphism 
$h_0:\sV'{}\!_\lambda \to \sW'{}\!_\lambda$ such that 
$l_\lambda{}\!^*h_0 = l_\lambda{}\!^*r \circ h_1 \circ l_\lambda{}\!^*s$.
We then have a diagram
\begin{equation*}
\xymatrix{
l^*\sV \ar_h[d] \ar^-{l^*a}[r] & l^*(\pr_\lambda)^*\sV_\lambda \ar[d] & 
l^*(\pr_\lambda)^*\sV'{}\!_\lambda \ar_-{l^*(\pr_\lambda)^*s}[l] 
\ar^{l^*(\pr_\lambda)^*h_0}[d] \\
l^*\sW \ar_-{l^*b}[r] & l^*(\pr_\lambda)^*\sW_\lambda \ar_-{l^*(\pr_\lambda)^*r}[r] &
l^*(\pr_\lambda)^*\sW'{}\!_\lambda
}
\end{equation*}
where $a$ and $b$ are pullback isomorphisms, the middle arrow coincides modulo
pullback isomorphisms with $h$ and $q_\lambda{}\!^*h_1$, 
and the right square commutes by naturality of the pullback
isomorphism $l^*(\pr_\lambda)^* \iso q_\lambda{}\!^*l_\lambda{}\!^*$.
By Lemmas~\ref{l:pointexact} and \ref{l:pointtors}, $(\pr_\lambda)^*r$ and
$(\pr_\lambda)^*s$ are isomorphisms up to torsion.
The exterior of the diagram thus shows that $h$ is in the image of
\eqref{e:HomVWbar}.
This proves that \eqref{e:MODbarK} is full.

Let $\sU$ be an object of $\Mod_{K,\varepsilon_0}(Z)$.
Restricting to a non-empty affine open subset of $Z$ shows that there exists 
a $\lambda \in \Lambda$ such that $\sU = q_\lambda{}\!^*\sU_1$ for some $\sU_1$
in $\Mod_{K_\lambda,\varepsilon_\lambda}(Z)$.
By the case where $X$ is of finite type, $\sU_1$ is isomorphic to $l_\lambda{}\!^*\sU_0$
for some $\sU_0$ in $\MOD_{G,\varepsilon}(X_\lambda)$.
Then $\sU$ is isomorphic to $q_\lambda{}\!^*l_\lambda{}\!^*\sU_0$ and hence
to $l^*(\pr_\lambda)^*\sU_0$.
Since every object of $\MOD_{K,\varepsilon_0}(Z)$ is the filtered colimit
of its subobjects in $\Mod_{K,\varepsilon_0}(Z)$, and since  \eqref{e:MODbarK}
is fully faithful and cocontinuous, the essential surjectivity of \eqref{e:MODbarK}
follows.  
\end{proof}

\section{Super Tannakian hulls}\label{s:supTann}

In this section we combine the results of the preceding two sections to construct a 
super Tannakian hull for any pseudo-Tannakian category as the full tensor subcategory
of dualisable objects of a functor category modulo torsion.
 
Let $\sC$ be an essentially small integral rigid tensor category.
Recall that $\widetilde{\sC}$ is the quotient $\overline{\widehat{\sC}}$
of $\widehat{\sC}$ by the torsion objects.
The canonical tensor functor 
$\sC \to (\widetilde{\sC})_\mathrm{rig}$ factors uniquely through the
faithful strict tensor functor
$E_{\sC}$ as 
\begin{equation}\label{e:canfac}
\sC \xrightarrow{E_{\sC}} \sC_\mathrm{fr} \to (\widetilde{\sC})_\mathrm{rig},
\end{equation}
and by Corollary~\ref{c:FrCff} the second arrow is fully faithful.
In particular $\sC \to \widetilde{\sC}$ defines an isomorphism
\begin{equation}\label{e:kappCEnd}
\kappa(\sC) \iso \End_{\widetilde{\sC}}(\I).
\end{equation}
Thus we may regard $\widetilde{\sC}$ as a $\kappa(\sC)$\nd tensor category.
Further $(\widetilde{\sC})_\mathrm{rig}$
is integral because $\sC_\mathrm{fr}$ is integral and any morphism $A \to \I$ 
in $(\widetilde{\sC})_\mathrm{rig}$,
after composing with an appropriate epimorphism $A' \to A$ in $\widetilde{\sC}$, 
lies in the image of the additive hull of $\sC_\mathrm{fr}$. 

An integral tensor category $\sC$ will be said to be \emph{of characteristic $0$}
if the integral domain $\End_{\sC}(\I)$ is of characteristic $0$.
When this is so, the canonical tensor functor from $\sC$ to $\Q \otimes_{\Z} \sC$ is
faithful.

\begin{lem}\label{l:CCQequiv}
Let $\sC$ be an essentially small, integral, rigid tensor category of characteristic $0$.
Then the tensor functor $\widetilde{\sC} \to (\Q \otimes_{\Z} \sC)^\sim$
induced by the canonical tensor functor $\sC \to \Q \otimes_{\Z} \sC$ is a tensor equivalence.
\end{lem}

\begin{proof}
Write $\sC'$ for $\Q \otimes_{\Z} \sC$ and $I:\sC \to \sC'$ for the canonical tensor functor.
We may identify the additive category $\widehat{\sC'}$ with the category 
of additive functors from $\sC$ to $\Q$\nd vector spaces.
Composing with the extension of scalars functor $\Q \otimes_{\Z} -$ from abelian groups to 
$\Q$\nd vector spaces then defines a functor $H:\widehat{\sC} \to \widehat{\sC'}$.
Further $H$ is cocontinuous and $Hh_- = h_-I$ is isomorphic to $\widehat{I}h_-$, so that
$H$ is isomorphic to $\widehat{I}$.
Since the forgetful functor from $\Q$\nd vector spaces to abelian groups is right adjoint
to $\Q \otimes_{\Z} -$, composing with it defines a forgetful functor
$H':\widehat{\sC'} \to \widehat{\sC}$ right adjoint to $H$.
The counit for the adjunction is an isomorphism, while the component $\eta_M:M \to H'H(M)$
of the unit at $M$ in $\widehat{\sC}$ is given by the canonical homomorphisms
\begin{equation}\label{e:unitMAQ}
M(A) \to \Q \otimes_{\Z} M(A) = H'H(M)(A) 
\end{equation}
for $A$ in $\sC$.
For every $M$, the kernel and cokernel of $\eta_M$ are torsion objects in $\widehat{\sC}$, 
because the kernel and cokernel of \eqref{e:unitMAQ} are torsion abelian groups, 
so that each of their elements is annulled by some non-zero element of $\Z \subset \sC(\I,\I)$.

Since the hom groups $\sC(A,\I)$ are torsion free and every element of $\sC'(A,\I)$
is a rational multiple of one in $\sC(A,\I)$, the torsion objects in $\widehat{\sC'}$
are those which are torsion as objects of $\widehat{\sC}$.
The forgetful functor $H'$ thus preserves torsion objects.
Similarly $H$ preserves torsion objects.
Since $H$ and $H'$ are exact, they thus define functors 
$\overline{H}:\widetilde{\sC} \to \widetilde{\sC'}$ and
$\overline{H'}:\widetilde{\sC'} \to \widetilde{\sC}$.
Further $\overline{H}$ and $\overline{H'}$ are adjoint to one another with the counit of the adjunction 
an isomorphism.
The unit, with components $\overline{\eta_M}$, is also an isomorphism because each
$\eta_M$ is an isomorphism up to torsion.
Thus $\overline{H}$ and hence $\widetilde{I}$ is an equivalence.
\end{proof}

\begin{defn}\label{d:pseudoTann}
By a \emph{pseudo-Tannakian category} we mean
an essentially small integral rigid tensor category of characteristic $0$ in which 
each object $M$ has a tensor power $M^{\otimes n}$ on which some non-zero element of 
$\Z[\mathfrak{S}_n]$ acts as $0$.
An abelian pseudo-Tannakian category will be called a \emph{super Tannakian category}.
\end{defn}

By Lemma~\ref{l:abfracclose},
$\End(\I)$ is a field of characteristic $0$ in any super Tannakian category.
A super Tannakian category in the sense of Definition~\ref{d:pseudoTann}
is thus the same as a tensor category $\sC$ such that $\End_{\sC}(\I)$ is
a field of characteristic $0$ and $\sC$ is a super Tannakian category over $\End_{\sC}(\I)$ 
in  the usual sense.

If $\sC$ is a pseudo-Tannakian category then $\sC_\mathrm{fr}$ is a pseudo-Tannakian category.
By Lemma~\ref{l:abfracclose}, any super Tannakian category is fractionally closed.

\begin{thm}\label{t:Tannequiv}
Let $\sC$ be a pseudo-Tannakian category 
and  $\overline{\kappa(\sC)}$ be an algebraic closure of $\kappa(\sC)$.
Then $\widetilde{\sC}$ is $\kappa(\sC)$\nd tensor equivalent to 
$\MOD_{K,\varepsilon}(\overline{\kappa(\sC)})$ for some 
transitive affine super groupoid with involution 
$(K,\varepsilon)$ over $\overline{\kappa(\sC)}/\kappa(\sC)$.
\end{thm}

\begin{proof}
We may suppose by Lemma~\ref{l:CCQequiv} that $\sC$ is a $\Q$\nd tensor category.
By Theorem~\ref{t:Ftildeequiv}  there exist an affine $\Q$\nd group
with involution $(G,\varepsilon)$ and an affine super $(G,\varepsilon)$\nd scheme $X$  
with $\Mod_{G,\varepsilon}(X)$ 
integral such that we have tensor equivalence
\begin{equation*}
\widetilde{\sC} \to \overline{\MOD_{G,\varepsilon}(X)}.
\end{equation*}
It induces by \eqref{e:kappCEnd} and Proposition~\ref{p:Frff} an isomorphism
\begin{equation*}
i:\kappa(\sC) \iso \kappa(\Mod_{G,\varepsilon}(X)).
\end{equation*}
With $\rho:\kappa(\Mod_{G,\varepsilon}(X)) \to \overline{\kappa(\sC)}$ the
composite of the embedding $\kappa(\sC) \to \overline{\kappa(\sC)}$ and $i^{-1}$,
there then exists by Lemma~\ref{l:kbarpoints} a $\overline{\kappa(\sC)}$\nd point $z$ of $X$
lying in every non-empty open super $G$\nd subscheme of $X$, such that the homomorphism
from $\kappa(\Mod_{G,\varepsilon}(X))$ to $\overline{\kappa(\sC)}$ defined by $z$
is $\rho$.
If $(K,\varepsilon_0)$ is the pullback of $(G \times_k X,G \times_k \varepsilon)$
along $z$,
then $K$ is transitive affine over $\overline{\kappa(\sC)}/\kappa(\Mod_{G,\varepsilon}(X))$
by Theorem~\ref{t:transaff} with $Z = \Spec(\overline{\kappa(\sC)})$ and $j = z$,
and pullback along $z$ defines by Theorem~\ref{t:GKMODequiv} a 
$\kappa(\Mod_{G,\varepsilon}(X))$\nd tensor equivalence
\begin{equation*}
\overline{\MOD_{G,\varepsilon}(X)} \to \MOD_{K,\varepsilon_0}(\overline{\kappa(\sC)}).
\end{equation*}
Regarding $K$ as a transitive affine groupoid over $\overline{\kappa(\sC)}/\kappa(\sC)$
using $i$ and writing $\varepsilon$ instead of $\varepsilon_0$
thus gives the required $\kappa(\sC)$\nd tensor equivalence.
\end{proof}

Let $\sC$ be a pseudo-Tannakian category.
By Theorem~\ref{t:Tannequiv}, 
the full subcategory $(\widetilde{\sC})_\mathrm{rig}$
of $\widetilde{\sC}$ is abelian and exactly embedded, and closed under the formation 
of subquotients and extensions. 
Further every object of $(\widetilde{\sC})_\mathrm{rig}$ is of finite presentation
in $\widetilde{\sC}$, and every object of $\widetilde{\sC}$ is the filtered colimit
of its subobjects in $(\widetilde{\sC})_\mathrm{rig}$.

Write $\sC_0$ for the full subcategory of $(\widetilde{\sC})_\mathrm{rig}$ 
whose objects are direct sums of those in the image of $\sC \to \widetilde{\sC}$.
Then the second arrow of \eqref{e:canfac} factors uniquely through
the embedding of $\sC_0$ as
\begin{equation}\label{e:addhull}
\sC_\mathrm{fr} \to \sC_0 \to (\widetilde{\sC})_\mathrm{rig}
\end{equation}
where the first arrow is the embedding of $\sC_\mathrm{fr}$ into its additive hull.
Every object of $(\widetilde{\sC})_\mathrm{rig}$ is a quotient of one
in $\sC_0$.
It follows that every object of 
$(\widetilde{\sC})_\mathrm{rig}$ is the cokernel of a morphism in $\sC_0$.
Taking duals then shows that every object of $(\widetilde{\sC})_\mathrm{rig}$ is 
the kernel of a morphism in $\sC_0$. 

\begin{cor}
A rigid tensor category is equivalent to a full tensor subcategory of a super Tannakian
category if and only if it is fractionally closed and pseudo-Tannakian.
\end{cor}

\begin{proof}
The ``only if'' is clear.
Conversely, if $\sC$ is a fractionally closed pseudo-Tannakian category, then
$\sC \to (\widetilde{\sC})_\mathrm{rig}$ is fully faithful by \eqref{e:canfac}
with $(\widetilde{\sC})_\mathrm{rig}$ super Tannakian by Theorem~\ref{t:Tannequiv}.
\end{proof}

\begin{lem}\label{l:Tannhullab}
Let $\sC$ be a super Tannakian category.
Then the canonical tensor functor $\sC \to (\widetilde{\sC})_\mathrm{rig}$ is 
a tensor equivalence.
\end{lem}

\begin{proof}
By Lemma~\ref{l:abfracclose}, $E_{\sC}$ is an equivalence, so that by \eqref{e:canfac}
$\sC \to (\widetilde{\sC})_\mathrm{rig}$ is fully faithful with essential image 
$\sC_0$ as in \eqref{e:addhull}.
Thus any object of $(\widetilde{\sC})_\mathrm{rig}$ is the kernel of a morphism
in the image of $\sC \to (\widetilde{\sC})_\mathrm{rig}$.
Since $\sC \to \widehat{\sC}$ is left exact, so also are $\sC \to \widetilde{\sC}$ and 
$\sC \to (\widetilde{\sC})_\mathrm{rig}$.
The essential surjectivity follows.  
\end{proof}

\begin{cor}\label{c:Tannequivab}
Let $\sC$ be a super Tannakian category.
Denote by $k$ the field $\End_{\sC}(\I)$, and let $\overline{k}$ be an 
algebraic closure of $k$.
Then $\sC$ is $k$\nd tensor equivalent to 
$\Mod_{K,\varepsilon}(\overline{k})$ for some 
transitive affine super groupoid with involution 
$(K,\varepsilon)$ over $\overline{k}/k$.
\end{cor}

\begin{proof}
Since $k = \kappa(\sC)$ by Lemma~\ref{l:abfracclose},
the result follows from Theorem~\ref{t:Tannequiv} and Lemma~\ref{l:Tannhullab}.
\end{proof}

Recall that if $V$ is a super vector space over a field of characteristic $0$
of super dimension $m|n$, then $S^\lambda V = 0$ 
for a partition $\lambda$ if and only if $[\lambda]$ contains the rectangular diagram with 
$n+1$ rows and $m+1$ columns.
The same therefore holds with $V$ replaced by a super linear map with 
image of super dimension $m|n$.

\begin{lem}\label{l:supertensexact}
Any faithful tensor functor between super Tannakian categories is exact. 
\end{lem}

\begin{proof}
By Corollary~\ref{c:Tannequivab}, it is enough to prove the exactness of any faithful
tensor functor $T$ between tensor categories of the form $\Mod_{K,\varepsilon}(k')$
for an extension $k'$ of a field $k$ of characteristic $0$
and a transitive affine groupoid with involution $(K,\varepsilon)$ over $k'/k$.
The property of Schur functors recalled above shows that
$T$ preserves super dimensions of representations, 
and also monomorphisms and epimorphisms.
Since super dimensions are additive for short exact sequences of representations, 
the result follows.
\end{proof}

\begin{lem}\label{l:Tannhullff}
Let $\sC$ be a pseudo-Tannakian category and $\sA$ be an abelian tensor category.
Then composition with the canonical tensor functor $\sC \to (\widetilde{\sC})_\mathrm{rig}$
defines a fully faithful functor from the groupoid of right exact regular tensor
functors $(\widetilde{\sC})_\mathrm{rig} \to \sA$ to the groupoid of regular tensor functors
$\sC \to \sA$.  
\end{lem}

\begin{proof}
Composition with $E_{\sC}$ defines a fully faithful functor from regular
tensor functors $\sC_\mathrm{fr} \to \sA$ to regular tensor functors $\sC \to \sA$,
and composition with the first arrow of \eqref{e:addhull} defines an equivalence
from tensor functors $\sC_0 \to \sA$ to tensor functors $\sC_\mathrm{fr} \to \sA$.  
It thus enough to show that given right exact tensor functors $T_1$ and $T_2$ 
from $(\widetilde{\sC})_\mathrm{rig}$ to $\sA$ and a tensor isomorphism
\begin{equation*}
\varphi_0:T_1|\sC_0 \iso T_2|\sC_0,
\end{equation*}
there is a unique tensor isomorphism
$\varphi:T_1 \iso T_2$ with $\varphi|\sC_0 = \varphi_0$.

As above, every object $A$ of $(\widetilde{\sC})_\mathrm{rig}$ is the target
of an epimorphism $p:B \to A$ in $(\widetilde{\sC})_\mathrm{rig}$ with $B$ in $\sC_0$,
and $p$ may be written as the cokernel of a morphism in $\sC_0$.
Thus $\varphi$ is unique if it exists, because $\varphi_A$ must render the square
\begin{equation*}
\xymatrix{
T_1(B) \ar_{T_1(p)}[d] \ar^{\varphi_0{}_B}[r] & T_2(B) \ar^{T_2(p)}[d] \\
T_1(A) \ar^{\varphi_A}[r] & T_2(A)
}
\end{equation*}
in $\sA$ commutative, with the left arrow an epimorphism.
That $\varphi_A$ as defined by the square exists and is an isomorphism can be seen
by writing $p$ as the cokernel of a morphism in $(\widetilde{\sC})_\mathrm{rig}$
and using the naturality of $\varphi_0$.
It is independent of the choice of $p$, because two epimorphisms $B_1 \to A$ and $B_2 \to A$
both factor through the same epimorphism $B_1 \oplus B_2 \to A$.

That $\varphi_A = \varphi_0{}_A$ for $A$ in $\sC_0$ is clear.
It remains to check that $\varphi_A$ is natural in $A$ and compatible with the  
tensor product.
Let $a:A' \to A$ be a morphism in $(\widetilde{\sC})_\mathrm{rig}$.
If $p:B \to A$
is an epimorphism with $B$ in $\sC_0$, 
composing the pullback of $p$ along $a$ with an appropriate epimorphism gives
a commutative square
\begin{equation*}
\xymatrix{
B' \ar_{p'}[d] \ar[r] & B \ar^{p}[d] \\
A' \ar^a[r] & A
}
\end{equation*}
in $(\widetilde{\sC})_\mathrm{rig}$ with $p$ and $p'$ epimorphisms and $B$ and $B'$ in $\sC_0$.
We obtain from it a cube with front and back faces given by applying $T_1$ and $T_2$,
left and right faces by the commutative diagrams defining $\varphi_{A'}$ and $\varphi_A$,
whose top face commutes by naturality of $\varphi_0$.
Since $T_1(p')$ is an epimorphism, the bottom face also commutes.
This gives the naturality of $\varphi$.
Similarly given $A$ and $A'$, we obtain from epimorphisms $p:B \to A$
and $p':B' \to A'$ with $B$ and $B'$ in $\sC_0$ a cube where the front
and back faces commute by naturality of the tensor structural isomorphisms of 
$T_1$ and $T_2$, the left and right faces by naturality of $\varphi$,
and the top face by compatibility of $\varphi_0$ with the tensor product.
The bottom face, expressing the compatibility of $\varphi$ with the tensor product,
thus also commutes, because $T_1(A)$ is dualisable and hence 
$T_1(A) \otimes -$ is exact and $T_1(p) \otimes T_1(p')$ is an epimorphism.
\end{proof}

\begin{defn}
Let $\sC$ be a pseudo-Tannakian category.
A faithful tensor functor $T:\sC \to \sC'$ will be called a \emph{super Tannakian hull of $\sC$} 
if for every super Tannakian category $\sD$, composition with $T$ defines an equivalence
from the groupoid of faithful tensor functors $\sC' \to \sD$ to the groupoid of 
faithful tensor functors $\sC \to \sD$.
\end{defn}

\begin{thm}\label{t:Tannhull}
For any pseudo-Tannakian category $\sC$,
the canonical tensor functor $\sC \to (\widetilde{\sC})_\mathrm{rig}$ is
a super Tannakian hull of $\sC$.
\end{thm}

\begin{proof}
Let $\sD$ be a super Tannakian category.
By Lemma~\ref{l:supertensexact} and Lemma~\ref{l:Tannhullff} with $\sA = \sD$, 
it is enough to show that every 
faithful tensor functor $T:\sC \to \sD$ factors up to tensor isomorphism through 
$\sC \to (\widetilde{\sC})_\mathrm{rig}$.
We have a diagram
\begin{equation*}
\xymatrix{
\sC \ar^{T}[d] \ar[r] & \widehat{\sC} \ar^{\widehat{T}}[d] \ar[r] & \widetilde{\sC} 
\ar^{\widetilde{T}}[d] \\
\sD \ar[r] & \widehat{\sD} \ar[r] & \widetilde{\sD}
}
\end{equation*}
of tensor functors, where the horizontal arrows are the canonical ones,
the left square commutes up to tensor isomorphism, and the right square commutes.
By Lemma~\ref{l:Tannhullab}, the composite if the bottom two arrows is fully faithful
with essential image $(\widetilde{\sD})_\mathrm{rig}$.
Since $\widetilde{T}$ sends $(\widetilde{\sC})_\mathrm{rig}$ into $(\widetilde{\sD})_\mathrm{rig}$, 
the required factorisation follows. 
\end{proof}

\begin{cor}\label{c:Tannhulleq}
Let $\sC$ be a pseudo-Tannakian category, $\sC'$ be a super Tannakian category,
and $T:\sC \to \sC'$ be a faithful tensor functor.
Denote by $T':\sC_\mathrm{fr} \to \sC'$ the tensor functor with $T = T'E_{\sC}$.
Then the following conditions are equivalent:
\begin{enumerate}
\renewcommand{\theenumi}{(\alph{enumi})}
\item\label{i:Tannhulldef}
$T$ is a super Tannakian hull of $\sC$.
\item\label{i:Tannhullffq}
$T'$ is fully faithful and every object of $\sC'$ is a quotient of 
a direct sum of objects in the essential image of $T'$.
\end{enumerate}
\end{cor}

\begin{proof}
Write $U:\sC \to (\widetilde{\sC})_\mathrm{rig}$ for the canonical tensor functor.
We may assume by Theorem~\ref{t:Tannhull} that $T = T''U$ for a faithful tensor functor 
\begin{equation*}
T'':(\widetilde{\sC})_\mathrm{rig} \to \sC'. 
\end{equation*}
It is then to be shown that \ref{i:Tannhullffq} holds if and only if $T''$ is an equivalence.

We have $T' = T''U'$ with $U'$ the second arrow of \eqref{e:canfac},
and it has been seen following \eqref{e:canfac} and \eqref{e:addhull} that
\ref{i:Tannhullffq} holds with $T$ and $T'$ replaced by $U$ and $U'$.
Thus if $T''$ is an equivalence then \ref{i:Tannhullffq} holds.

Conversely suppose that \ref{i:Tannhullffq} holds. 
Then the restriction of $T''$ to the essential image of the fully faithful
functor $U'$ is fully faithful, and hence so also is the restriction of $T''$
to the additive hull $\sC_0$ of this essential image in $(\widetilde{\sC})_\mathrm{rig}$.
Since every object of $(\widetilde{\sC})_\mathrm{rig}$ is a cokernel of a morphism
in $\sC_0$, and since $T''$ is exact by Lemma~\ref{l:supertensexact},
considering morphisms in $(\widetilde{\sC})_\mathrm{rig}$ with target $\I$ shows that
$T''$ is fully faithful.
The essential image of $T''$ contains the additive hull of the essential image of $T'$.
Thus by the hypothesis on $T'$, every object of $\sC$ is a cokernel of a morphism
between objects in the essential image of $T''$.
The essential surjectivity of $T''$ then follows from its full faithfulness and exactness. 
\end{proof}

\begin{cor}
Any faithful tensor functor from a fractionally closed rigid tensor category
to a category of super vector spaces over a field of characteristic $0$ is conservative.
\end{cor}

\begin{proof}
Let $\sC$ be a fractionally closed rigid tensor category and $T:\sC \to \Mod(k)$ 
be a faithful tensor functor with $k$ a field of characteristic $0$.
Then $\sC$ is integral.
To prove that $T$ is conservative, we may after replacing $\sC$ by a full rigid tensor
subcategory assume that $\sC$ is essentially small, and
hence pseudo-Tannakian.
By Theorem~\ref{t:Tannhull}, 
$T$ factors up to tensor isomorphism as the canonical tensor functor 
$U:\sC \to (\widetilde{\sC})_\mathrm{rig}$
followed by a faithful tensor functor $T'$.
By Lemma~\ref{l:supertensexact}, $T'$ is exact and hence conservative.
Since $U$ is fully faithful, $T$ is thus conservative.
\end{proof}

Let $k$ be a commutative ring, $k'$ be a commutative $k$\nd algebra, and $\sA$ be a cocomplete 
abelian $k$\nd tensor category with cocontinuous tensor product.
Then $k' \otimes_k \I$ has a canonical structure of commutative algebra in $\sA$, 
and $k'$ acts on this algebra through its action on $k'$. 
Thus we have a cocomplete
abelian $k'$\nd tensor category $\MOD_{\sA}(k' \otimes_k \I)$ with cocontinuous tensor product.
If $k$ is a field of characteristic $0$ and $k'$ is an extension of $k$,
then for $\sA = \MOD(X)$ with $X$ a super $k$\nd scheme
\begin{equation*}
\MOD_{\sA}(k' \otimes_k \I) = \MOD(X_{k'}),
\end{equation*}
and for $\sA = \MOD_{K,\varepsilon}(X)$ with $(K,\varepsilon)$ a super groupoid with 
involution over $X/k$
\begin{equation}\label{e:MODAMODK}
\MOD_{\sA}(k' \otimes_k \I) = \MOD_{K_{k'},\varepsilon}(X_{k'}).
\end{equation}
It thus follows from Theorem~\ref{t:Tannequiv} that if $\sC$ is a pseudo-Tannakian
$k$\nd tensor category with $\kappa(\sC) = k$, then $\Mod_{\widetilde{\sC}}(k' \otimes_k \I)$
is a super Tannakian category over $k'$.

\begin{cor}
Let $k$ be a field of characteristic $0$ and $k'$ be an extension of $k$.
Let $T:\sC \to \sC'$ be a faithful $k$\nd tensor functor with $\sC$ pseudo-Tannakian
and $\sC'$ super Tannakian, and $T':k' \otimes_k \sC \to \Mod_{\widetilde{\sC'}}(k' \otimes_k \I)$
be the $k'$\nd tensor functor induced by $T$.
Suppose that $\kappa(\sC) = k$.
Then $T$ is a super Tannakian hull of $\sC$ if and only if $T'$ is a super Tannakian hull of 
$k' \otimes_k \sC$.
\end{cor}

\begin{proof}
By Lemma~\ref{l:ext}, $k' \otimes_k \sC$ is an integral $k'$\nd tensor category
with $\kappa(k' \otimes_k \sC) = k'$, and $T'$ is faithful.
If $T = T_1E_{\sC}$ and $T' = T'{}\!_1E_{k' \otimes_k \sC}$,
we have a commutative square
\begin{equation*}
\xymatrix{
(k' \otimes_k \sC)_\mathrm{fr} \ar^-{T'{}\!_1}[r] & \Mod_{\widetilde{\sC'}}(k' \otimes_k \I) \\
k' \otimes_k \sC_\mathrm{fr} \ar[u] \ar^-{k' \otimes_k T_1}[r] & k' \otimes_k \sC' \ar[u]
}
\end{equation*}
with the left arrow defined by the factorisation of $\sC \to (k' \otimes_k \sC)_\mathrm{fr}$ 
through $\sC_\mathrm{fr}$.
The right arrow of the square is fully faithful, and every object $M$ of 
$\Mod_{\widetilde{\sC'}}(k' \otimes_k \I)$ is a quotient of an object in its image:
the $(k' \otimes_k \I)$\nd module structure $k' \otimes_k M \to M$ of $M$ restricts to an epimorphism
$k' \otimes_k M_0 \to M$ for some subobject $M_0$ of $M$ in $\sC'$.
Suppose that $T$ is a super Tannakian hull of $\sC$.
Then by Corollary~\ref{c:Tannhulleq}, $T_1$ is fully faithful, and every object $\sC'$ is a quotient of
a direct sum of objects in its image.
Since $T'{}\!_1$ is faithful, it follows from the square that the same holds with $T_1$ replaced
by $T'{}\!_1$.
Thus by Corollary~\ref{c:Tannhulleq}, $T'$ is a super Tannakian hull of $k' \otimes_k \sC$.

Conversely suppose that $T'$ is a super Tannakian hull of $k' \otimes_k \sC$.
Let $U:\sC \to \sC''$ be a super Tannakian hull of $\sC$.
Then $\sC''$ has a unique structure of $k$\nd tensor category such that $U$ is a $k$\nd tensor
functor, and $T$ factors up to tensor isomorphism as $U$ followed by a $k$\nd tensor functor
$\sC'' \to \sC'$.
The induced $k'$\nd tensor functor
\begin{equation*}
\Mod_{\widetilde{\sC''}}(k' \otimes_k \I) \to \Mod_{\widetilde{\sC'}}(k' \otimes_k \I)
\end{equation*}
is an equivalence by the ``only if'' with $T$ replaced by $U$. 
Its full faithfulness implies that of $\sC'' \to \sC'$.
Its essential surjectivity implies that every $k' \otimes_k M$ with $M$ in $\sC'$ is a
quotient in $\Mod_{\widetilde{\sC'}}(k' \otimes_k \I)$ of $k' \otimes_k M'$ for some $M'$ 
in the image of $\sC'' \to \sC'$,
and hence by choosing an epimorphism $k' \to k$ of $k$\nd vector spaces that
$M$ is the quotient in $\sC'$ of some $M'{}^n$.
It follows that $\sC'' \to \sC'$ is an equivalence, so that $T$ is a super Tannakian
hull of $\sC$. 
\end{proof}

Let $f$ be a morphism of super vector bundles over a super scheme $X$.
If the fibre $f_x$
of $f$ above a point $x$ of $X$ is an epimorphism of super vector spaces, 
then $f$ is an epimorphism in $\MOD(X')$ for some open super subscheme $X'$ of $X$
containing $x$, by Nakayama's lemma applied to the morphism induced on stalks.
In particular if $f_x$ is an isomorphism, then $f$ is an isomorphism 
in some neighbourhood of $x$.

\begin{lem}\label{l:vecbunhom}
Let $X$ be a super $\Q$\nd scheme and $f$ be a morphism 
of vector bundles over $X$.
For each point $x$ of $X$ and partition $\lambda$, suppose that $S^\lambda(f_x) = 0$
implies $S^\lambda(f) = 0$.
Then the kernel, cokernel and image of $f$ in $\MOD(X)$ are super vector bundles, 
and the image has constant super rank.
\end{lem}

\begin{proof}
Since $S^\lambda(f) = 0$ implies $S^\lambda(f_x) = 0$,
it follows from the property of Schur functors recalled before Lemma~\ref{l:supertensexact}
that the super dimension $m|n$ of the image of $f_x$ is constant.
It is thus enough to show that every point $x$ of $X$ is contained in an open super subscheme $X'$
of $X$ such that the kernel, cokernel and image of the restriction $f':\sV' \to \sW'$ of 
$f:\sV \to \sW$ above $X'$ are super vector bundles.

Suppose that $f_x$ has super rank $m|n$.
Then we may write $\sV_x = V_1 \oplus V_2$ and $\sW_x = W_1 \oplus W_2$
with $V_1$ and $W_1$ of super dimension $m|n$ so that the matrix entry
$V_1 \to W_1$ of $f_x$ is an isomorphism.  
For sufficiently small $X'$, we have decompositions $\sV' = \sV'{}\!_1 \oplus \sV'{}\!_2$ 
and $\sW' = \sW'{}\!_1 \oplus \sW'{}\!_2$ such that $(\sV'{}\!_i)_x = V_i$
and $(\sW'{}\!_i)_x = W_i$,
with $\sV'{}\!_1$ and $\sW'{}\!_1$ of constant rank $m|n$.
The entry $f'{}\!_{11}:\sV'{}\!_1 \to \sW'{}\!_1$
of the matrix $(f'{}\!_{ij})$ of $f'$ is then an isomorphism.
We have
\begin{gather*}
\begin{pmatrix}
1 & 0 \\
- f'{}\!_{21} \circ f'{}\!_{11}{}\!^{-1} & 1
\end{pmatrix}
\begin{pmatrix}
f'{}\!_{11} & f'{}\!_{12} \\
f'{}\!_{21} & f'{}\!_{22}
\end{pmatrix}
\begin{pmatrix}
1 & - f'{}\!_{11}{}\!^{-1} \circ f'{}\!_{12} \\
0 & 1
\end{pmatrix}
=
\begin{pmatrix}
f'{}\!_{11} & 0 \\
0 & h
\end{pmatrix}
\end{gather*}
with $h = f'{}\!_{22} - f'{}\!_{21} \circ f'{}\!_{11}{}\!^{-1} \circ f'{}\!_{12}$.
After modifying appropriately the direct sum decompositions of $\sV'$ and $\sW'$,
we may thus assume that $f'$ is the direct sum
\begin{equation*}
f' = f'{}\!_{11} \oplus h:\sV'{}\!_1 \oplus \sV'{}\!_2 \to \sW'{}\!_1 \oplus \sW'{}\!_2
\end{equation*}
with $f'{}\!_{11}$ an isomorphism of vector bundles of constant rank $m|n$.
If $\lambda$ is the partition whose diagram is rectangular with
$n+1$ rows and $m+1$ columns, then $S^\lambda(f_x) = 0$, so that by hypothesis
$S^\lambda(f) = 0$ and $S^\lambda(f') = 0$.
Thus by the formula \cite[1.8]{Del02} for the decomposition 
of the Schur functor of a direct sum we have 
\begin{equation*}
0 =
S^\lambda(f'{}\!_{11} \oplus h) = \bigoplus_{|\mu| + |\nu| = (m+1)(n+1)}
(S^\mu(f'{}\!_{11}) \otimes_{\sO_{X'}} S^\nu(h))^{[\lambda:\mu,\nu]}
\end{equation*}
with the multiplicity $[\lambda:\mu,\nu]$ given by the Littlewood--Richardson rule.
If $[\mu] \subset [\lambda]$ and $|\nu| = 1$, then $[\lambda:\mu,\nu] = 1$ 
\cite[1.5.1]{Del02}, and $S^\mu(f'{}\!_{11})$ is an isomorphism of non-zero 
constant rank super vector bundles.
Thus $h = 0$. 
\end{proof}

\begin{defn}
Let $X$ be a super $\Q$\nd scheme.
A functor $T$ with target $\Mod(X)$ will be called \emph{pointwise faithful} 
if for every point $x$ of $X$,
the composite of passage to the fibre at $x$ with $T$ is faithful. 
\end{defn}

Let $T:\sC \to \Mod(X)$ be a tensor functor. 
If $T$ is pointwise faithful, then every morphism $f$ in the image of $T$ satisfies
the hypothesis of Lemma~\ref{l:vecbunhom}.
If $T$ is pointwise faithful and $h$ is a non-zero morphism in $\sC$, 
then $T(h)$ is strongly regular in $\Mod(X)$ and even in $\MOD(X)$, because 
by Lemma~\ref{l:vecbunhom} it factors locally on $X$ as a retraction of super vector bundles
followed by a section.
Thus any pointwise faithful $T$ factors through a tensor functor from $\sC_\mathrm{fr}$
to $\Mod(X)$, which is also pointwise faithful.
When $\sC$ is rigid, $T$ is pointwise faithful if and only if it sends every non-zero
morphism with target $\I$ in $\sC$ to an epimorphism in $\MOD(X)$.
For $\sC$ super Tannakian, it follows from Lemma~\ref{l:supertensexact} 
that $T$ is pointwise faithful if and only if the composite of the embedding of 
$\Mod(X)$ into $\MOD(X)$ with $T$ is exact.

Let $\sA$ and $\sA'$ be cocomplete abelian categories, and $\sA_0$ be an essentially
small full abelian subcategory of $\sA$. 
Suppose that every object of $\sA_0$ is of finite presentation in $\sA$,
and that every object of $\sA$ is a filtered colimit of objects in $\sA_0$. 
Then the embedding $\sA_0 \to \sA$ is exact, and every morphism in $\sA$ is a filtered colimit
of morphisms in $\sA_0$.
The embedding is also dense, because if $A$ in $\sA$ is the filtered colimit 
$\colim_{\lambda \in \Lambda} A_\lambda$ with the $A_\lambda$ in $\sA_0$,
then the canonical functor $\Lambda \to \sA_0/A$ is cofinal.
Let $H_0:\sA_0 \to \sA'$ be a right exact functor.
The additive left Kan extension $H:\sA \to \sA'$
of $H_0$ along the embedding of $\sA_0$ into $\sC$ exists, with the universal natural 
transformation from $H_0$ to the restriction of $H$ to $\sA_0$ an isomorphism.
It is given by the coend formula
\begin{equation*}
H = \int^{A_0 \in \sA_0} \sA(A_0,-) \otimes_{\Z} H_0(A_0),
\end{equation*}
and it is preserved by any cocontinuous functor $\sA' \to \sA''$.
By the formula, $H$ commutes with filtered colimits,
and the canonical natural transformation from the additive left Kan extension of the embedding
$\sA_0 \to \sA$ along itself to the identity of $\sA$ is an isomorphism (i.e.\ $\sA_0$
is additively dense in $\sA$).
Thus $H$ is cocontinuous, and it is exact if $H_0$ is.
It follows that restriction from $\sA$ to $\sA_0$ defines an equivalence from cocontinuous
functors $\sA \to \sA'$ to right exact functors $\sA_0 \to \sA'$, with quasi-inverse
given by additive left Kan extension.

Suppose now that $\sA$ and $\sA'$ have tensor structures with the tensor products cocontinuous,
and that $\sA_0$ is a full tensor subcategory of $\sA$.
If $H:\sA \to \sA'$ is a cocontinuous functor with restriction $H_0$ to $\sA_0$, 
then by density of $\sA_0 \to \sA$, any tensor structure
on $H_0$ can be extended uniquely to a tensor structure on $H$.
Also if $H':\sA \to \sA'$ is a tensor functor, then
any natural isomorphism $H \iso H'$ whose restriction to $\sA_0$ is a tensor isomorphism
is itself a tensor isomorphism.
Thus restriction from $\sA$ to $\sA_0$ defines an equivalence from the groupoid of
cocontinuous tensor functors $\sA \to \sA'$ to the groupoid of right exact tensor
functors $\sA_0 \to \sA'$.

The conditions on $\sA$ and $\sA_0$ are satisfied in particular when
$\sA = \MOD_{K,\varepsilon}(X)$ and 
$\sA_0 = \Mod_{K,\varepsilon}(X)$ for $X$ a super scheme over a field $k$ of characteristic $0$
and $(K,\varepsilon)$ a transitive affine groupoid over $X/k$, and hence by Theorem~\ref{t:Tannequiv}
when $\sA = \widetilde{\sC}$ and $\sA_0 = (\widetilde{\sC})_{\mathrm{rig}}$ for $\sC$ pseudo-Tannakian.

\begin{thm}\label{t:TannhullModX}
Let $\sC$ be a pseudo-Tannakian category, $U:\sC \to \sC'$ be a super Tannakian 
hull of $\sC$, and $X$ be a super $\Q$\nd scheme.
Then composition with $U$ defines an equivalence from the groupoid of pointwise faithful tensor 
functors $\sC' \to \Mod(X)$ to the groupoid of pointwise faithful tensor functors $\sC \to \Mod(X)$.
\end{thm}

\begin{proof}
We may suppose by Theorem~\ref{t:Tannhull} that $\sC' = (\widetilde{\sC})_\mathrm{rig}$ and 
\begin{equation*}
U:\sC \to (\widetilde{\sC})_\mathrm{rig} 
\end{equation*}
is the canonical tensor functor.
Since the composite of the embedding of $\Mod(X)$ into $\MOD(X)$ with any pointwise
faithful tensor functor $(\widetilde{\sC})_\mathrm{rig} \to \Mod(X)$ is regular and exact, 
the required full faithfulness follows from Lemma~\ref{l:Tannhullff} with $\sA = \MOD(X)$. 

Let $T:\sC \to \Mod(X)$ be a pointwise faithful tensor functor.
To prove the essential surjectivity, it is to be shown that $T$ factors up to tensor
isomorphism through $U$.
The composite $T_1:\sC \to \MOD(X)$ of the embedding into $\MOD(X)$ with $T$ gives by
additive left Kan extension a cocontinuous tensor functor
\begin{equation*}
T_1{}\!^*:\widehat{\sC} \to \MOD(X)
\end{equation*}
with $T_1{}\!^*h_-$ tensor isomorphic to $T_1$.
We show that $T_1{}\!^*$ sends isomorphisms up to torsion to isomorphisms.
It will follow that $T_1{}\!^*$ factors through the projection 
\begin{equation*}
P:\widehat{\sC} \to \widetilde{\sC},
\end{equation*}
and composing with $h_-$ will give the required factorisation of $T$.

Let $j:M \to N$ be an isomorphism up to torsion in $\widehat{\sC}$.
Then for any subobject $N_0$ of $N$, the morphism
$j^{-1}(N_0) \to N_0$ induced by $j$ is an isomorphism up to torsion.
If $N_0$ is of finite type, and we write $j^{-1}(N_0)$ as the filtered colimit
of its subobjects $M_0$ of finite type, then by Theorem~\ref{t:Tannequiv}
and the exactness and cocontinuity of $P$, there is an $M_0$ such that the
morphism $P(M_0) \to P(N_0)$ induced by $P(j)$ is an isomorphism.
Thus $M_0 \to N_0$ induced by $j$ is an isomorphism up to torsion.
It follows that $j$ may be written as a filtered colimit of isomorphisms 
up to torsion between objects of finite type.
To prove that $T_1{}\!^*(j)$ is an isomorphism, we thus reduce by cocontinuity 
of $T_1{}\!^*$ to the case where $M$ and $N$ are of finite type.

The restriction of $T_1{}\!^*$ to the pseudo-abelian hull of $\sC$ in $\widehat{\sC}$ is 
pointwise faithful, and hence by Lemma~\ref{l:vecbunhom} applied to identities,
it factors through the full subcategory of $\MOD(X)$ of vector bundles over $X$
of constant super rank.
Thus by right exactness of $T_1{}\!^*$ and Lemma~\ref{l:vecbunhom}, 
$T_1{}\!^*(M)$ is a vector bundle over $X$ of constant super rank for $M$ of finite 
presentation in $\widehat{\sC}$.
Suppose that $M$ is of finite type in $\widehat{\sC}$.
Then $M$ is the colimit of a filtered system $(M_\lambda)$ of objects $M_\lambda$
of finite presentation in $\widehat{\sC}$ with transition morphisms epimorphisms.
Thus $T_1{}\!^*(M)$ again is a super vector bundle over $X$ of constant super rank
by cocontinuity of $T_1{}\!^*$.

Write $F_x:\Mod(X) \to \Mod(\kappa(x))$ and 
$F_{1x}:\MOD(X) \to \MOD(\kappa(x))$ for the tensor functors defined by
passing to the fibre at the point $x$ of $X$.
Since $T$ is pointwise faithful, $F_xT$ is faithful.
Then we have a diagram
\begin{equation*}
\xymatrix{
\widehat{\sC} \ar^-{P}[r] & \widetilde{\sC} \ar^-{L_x}[r] & \MOD(\kappa(x)) \\
\sC \ar^{h_-}[u] \ar^-{U}[r] & (\widetilde{\sC})_\mathrm{rig} \ar[u] \ar^-{H_x}[r] & 
\Mod(\kappa(x)) \ar[u]
}
\end{equation*}
where the middle and right vertical arrows are the embeddings, the left square commutes
by definition, $H_x$ is given up to tensor isomorphism by factoring
$F_xT$ through $U$ and is right exact by Lemma~\ref{l:supertensexact},
and $L_x$ is given up to tensor isomorphism by requiring that it be cocontinuous
and that the right square commute up to tensor isomorphism.
Then $F_{1x} T_1{}\!^*h_-$ is tensor isomorphic to the bottom right leg of the diagram,
so that $L_xPh_-$ and $F_{1x} T_1{}\!^*h_-$ are tensor isomorphic.
There thus exists a tensor isomorphism
\begin{equation*}
L_xP \iso F_{1x} T_1{}\!^*,
\end{equation*}
because $L_xP$ and $F_{1x} T_1{}\!^*$ are cocontinuous.

Now let $j:M \to N$ be an isomorphism up to torsion in $\widehat{\sC}$ with $M$ and $N$ 
of finite type.
Then $L_x(P(j))$ and hence the fibre  $F_{1x}(T_1{}\!^*(j))$ of $T_1{}\!^*(j)$ above $x$
is an isomorphism for every point $x$ of $X$.
Since $T_1{}\!^*(M)$ and $T_1{}\!^*(N)$ are super vector bundles over $X$, it follows that
$T_1{}\!^*(j)$ is an isomorphism, as required.
\end{proof}

Let $k$ be a field of characteristic $0$ and $(G,\varepsilon)$ be an affine super 
$k$\nd group with involution.
Then the right action of $G \times_k G$ on $G$ for which $(g_1,g_2)$ sends $g$ to $g_1{}\!^{-1}gg_2$
defines a structure of $(G \times_k G,(\varepsilon,\varepsilon))$\nd module on $k[G]$.
For $i = 1,2$, pullback along the $i$th projection defines a $k$\nd tensor functor from
$(G,\varepsilon)$\nd modules to $(G \times_k G,(\varepsilon,\varepsilon))$\nd modules.
If $V$ is a $(G,\varepsilon)$\nd module, the action of $G$ on $V$ is a morphism
\begin{equation}\label{e:VGGact}
\pr_2{}\!^*V \to \pr_1{}\!^*V \otimes_k k[G]
\end{equation}
of $(G \times_k G,(\varepsilon,\varepsilon))$\nd modules.
Evaluation at the identity of $G$ defines a $k$\nd linear map left inverse
to \eqref{e:VGGact}.
If $G_i$ denotes the normal super $k$\nd subgroup of $G \times_k G$ given by embedding $G$ as 
the $i$th factor, then the restriction of the left inverse to the 
$(G \times_k G,(\varepsilon,\varepsilon))$\nd submodule
$(\pr_1{}\!^*V \otimes_k k[G])^{G_1}$ of invariants under $G_1$ is injective,
because a $G_1$\nd invariant section of $\pr_1{}\!^*V \otimes_k \sO_G$ above $G$ is determined
by its value at the identity.
Thus \eqref{e:VGGact} factors through an isomorphism
\begin{equation}\label{e:VGGactinv}
\pr_2{}\!^*V \iso (\pr_1{}\!^*V \otimes_k k[G])^{G_1}
\end{equation}
of $(G \times_k G,(\varepsilon,\varepsilon))$\nd modules.

The embedding of $\Z/2$ into $G$ that sends $1$ to $\varepsilon$ defines a super $k$\nd subgroup
with involution
\begin{equation}\label{e:Gprimedef}
(G',\varepsilon') = (G \times_k (\Z/2),(\varepsilon,1))
\end{equation}
of $(G \times_k G,(\varepsilon,\varepsilon))$.
Restriction from $G \times_k G$ to $G'$ then defines a structure of $(G',\varepsilon')$\nd module
on $k[G]$.
We may identify $\MOD_{G,\varepsilon}(k)$ with the full subcategory
of $\MOD_{G',\varepsilon'}(k)$ consisting of those 
$(G',\varepsilon')$\nd modules on which the factor $\Z/2$ of $G'$ acts trivially,
and the category $\MOD_{\Z/2,1}(k)$ of super $k$\nd vector spaces with the
$(G',\varepsilon')$\nd modules on which $G$ acts trivially.
If we write $\Omega$ for the forgetful functor from $(G,\varepsilon)$\nd modules 
to super $k$\nd vector spaces, then for any $(G,\varepsilon)$\nd module $V$ the
action of $G$ on $V$ is by \eqref{e:VGGact} a morphism
\begin{equation}\label{e:VGprimeact}
\Omega(V) \to V \otimes_k k[G]
\end{equation}
of $(G',\varepsilon')$\nd modules, and by \eqref{e:VGGactinv} it factors through an isomorphism
\begin{equation}\label{e:VGprimeactinv}
\Omega(V) \iso (V \otimes_k k[G])^G
\end{equation} 
of $(G',\varepsilon')$\nd modules.

The central embedding of the factor $\Z/2$ of $G'$ defines a $(\Z/2)$\nd grading
on $\MOD_{G',\varepsilon'}(k)$, so that $\MOD_{G',\varepsilon'}(k)$ is $k$\nd tensor equivalent 
to the $k$\nd tensor category of $(\Z/2)$\nd graded objects of $\MOD_{G,\varepsilon}(k)$ 
with symmetry given by the Koszul rule.
Let $W$ be a $(G',\varepsilon')$\nd module.
For $V$ a representation of $(G,\varepsilon)$ we have
a morphism
\begin{equation}\label{e:WVG}
V^\vee \otimes_k (V \otimes_k W)^G \to W
\end{equation}
of $(G',\varepsilon')$\nd modules,
natural in $W$ and extranatural in $V$, defined by the embedding of the invariants under $G$
and the counit for $V^\vee$.
Thus we have a morphism 
\begin{equation}\label{e:intWVG}
\int^{V \in \Mod_{G,\varepsilon}(k)} V^\vee \otimes_k (V \otimes_k W)^G \to W
\end{equation}
of $(G',\varepsilon')$\nd modules, natural in $W$.
If $W$ lies in $\MOD_{G,\varepsilon}(k)$ then $(V \otimes_k W)^G$ is the trivial 
$(G',\varepsilon')$\nd module given by the $k$\nd vector space $\Hom_G(V^\vee,W)$,
and \eqref{e:WVG} is
\begin{equation*}
V^\vee \otimes_k \Hom_G(V^\vee,W) \to W
\end{equation*}
defined by evaluation.
For $W$ in $\Mod_{G,\varepsilon}(k)$,  \eqref{e:intWVG} is thus an isomorphism, with 
inverse the composite of  the coprojection at $V = W^\vee$ and
\begin{equation*}
W \to W \otimes_k \Hom_G(W,W)
\end{equation*}
defined by $1_W$.
Since $ \Hom_G(V^\vee,-)$ commutes with filtered colimits, \eqref{e:intWVG}
is an isomorphism for any $W$ in $\MOD_{G,\varepsilon}(k)$.
It follows that \eqref{e:intWVG} is an isomorphism for $W$ the tensor product
of a $(G,\varepsilon)$\nd module with $k^{0|1}$, because the factor $k^{0|1}$ may
be taken outside $(-)^G$.
Hence \eqref{e:intWVG} is an isomorphism for an arbitrary $(G',\varepsilon')$\nd module $W$.
Taking $W = k[G]$ and using \eqref{e:VGprimeactinv} gives an isomorphism
\begin{equation}\label{e:Vomegaint}
\int^{V \in \Mod_{G,\varepsilon}(k)} V^\vee \otimes_k \Omega(V) \iso k[G]
\end{equation}
of $(G',\varepsilon')$\nd modules, with component at $V$ obtained from \eqref{e:VGprimeact} by dualising.

Let $\sC$ be an essentially small rigid tensor category, $X$ be a super $\Q$\nd scheme, 
and $T_1$ and $T_2$ be tensor functors $\sC \to \Mod(X)$.
It can be seen as follows that the functor on super schemes over $X$ that assigns to 
$p:X' \to X$ the set $\Iso^\otimes(p^*T_1,p^*T_2)$ of tensor
isomorphisms from $p^*T_1$ to $p^*T_2$ is represented by an affine super scheme 
\begin{equation*}
\underline{\Iso}^\otimes(T_1,T_2) = \Spec(\sR)
\end{equation*}
over $X$.
The object $\sR$ in $\MOD(X)$ is given by 
\begin{equation}\label{e:Rint}
\sR = \int^{C \in \sC} T_2(C)^\vee \otimes_{\sO_X} T_1(C).
\end{equation}
Its structure of commutative algebra is given by the morphisms
\begin{equation*}
(T_2(C)^\vee \otimes_{\sO_X} T_1(C)) \otimes_{\sO_X} 
(T_2(C')^\vee \otimes_{\sO_X} T_1(C')) \to
(T_2(C \otimes C')^\vee \otimes_{\sO_X} T_1(C \otimes C'))
\end{equation*}
and $\sO_X \to T_2(\I)^\vee \otimes_{\sO_X} T_1(\I)$ defined using the tensor structures 
of $T_1$ and $T_2$.
We have bijections
\begin{equation*}
\Hom_{\sO_{X'}}(p^*\sR,\sO_{X'}) \iso 
\int_{C \in \sC} \Hom_{\sO_{X'}}(p^*T_1(C),p^*T_2(C))
\iso \Nat(p^*T_1,p^*T_2)
\end{equation*}
natural in $p:X' \to X$, which when $p$ is the structural morphism $\pi$ of $\Spec(\sR)$
send the canonical morphism $\pi^*\sR \to \sO_{\Spec(\sR)}$ to
\begin{equation*}
\alpha:\pi^*T_1 \to \pi^*T_2
\end{equation*}
where $\alpha_C:\pi^*T_1(C) \to \pi^*T_2(C)$ corresponds to $T_1(C) \to \sR \otimes_{\sO_X} T_2(C)$
dual to the coprojection at $C$ of the coend defining $\sR$.
Further $p^*\sR \to \sO_{X'}$ is a homomorphism of algebras if and only if the
corresponding natural transformation $p^*T_1 \to p^*T_2$ is compatible with the 
tensor structures.
Thus we have a bijection
\begin{equation*}
\Hom_X(X',\Spec(\sR)) \iso \Iso^\otimes(p^*T_1,p^*T_2)
\end{equation*}
natural in $p:X' \to X$, which gives the required representation, with universal element $\alpha$.

The affine super scheme $\underline{\Iso}^\otimes(T_1,T_2)$ over $X$ is functorial in $T_1$ and $T_2$.
It is clear from the definition that we have a canonical isomorphism
\begin{equation}\label{e:Isopull}
\underline{\Iso}^\otimes(p^*T_1,p^*T_2) \iso p^*\underline{\Iso}^\otimes(T_1,T_2)
\end{equation}
over $X'$ for every $p:X' \to X$.
If $U:\sC \to \sC'$ is a tensor functor with $\sC'$ essentially small and rigid such that
$T_i = T'{}\!_iU$ for $i = 1,2$, 
then composition with $U$ defines a morphism
\begin{equation}\label{e:IsoTannhull}
\underline{\Iso}^\otimes(T'{}\!_1,T'{}\!_2) \to \underline{\Iso}^\otimes(T_1,T_2)
\end{equation} 
over $X$, which is an isomorphism if 
$\Iso^\otimes(p^*T'{}\!_1,p^*T'{}\!_2) \to \Iso^\otimes(p^*T_1,p^*T_2)$ is bijective for every $p$.

Suppose that $X$ is non-empty and that $T_1$ and $T_2$ are pointwise faithful.
Then $\sC$ is integral and pseudo-Tannakian, and  $T_1$ and $T_2$ factor through $\sC_\mathrm{fr}$.
Thus $T_1$ and $T_2$ induce homomorphisms from $\kappa(\sC)$ to $H^0(X,\sO_X)$.
The equaliser of the corresponding morphisms $X \to \Spec(\kappa(\sC))$
is a closed super subscheme $Y$ of $X$,
and $\Iso^\otimes(p^*T_1,p^*T_2)$ is empty unless $p:X' \to X$ factors through $Y$.
Thus $\underline{\Iso}^\otimes(T_1,T_2)$ may be regarded as a scheme over $Y$.
The following lemma, together with \eqref{e:Isopull} with $p$ the embedding of $Y$,
shows that $\underline{\Iso}^\otimes(T_1,T_2)$ is faithfully
flat over $Y$.

\begin{thm}\label{t:Isofthfl}
Let $\sC$ be  an essentially small rigid tensor category, $X$ be a non-empty 
super $\Q$\nd scheme,
and $T_1$ and $T_2$ be pointwise faithful tensor functors from $\sC$ to $\Mod(X)$ 
which induce the same homomorphism from $\kappa(\sC)$ to $H^0(X,\sO_X)$.
Then $\underline{\Iso}^\otimes(T_1,T_2)$ is faithfully flat over $X$.
\end{thm}

\begin{proof}
With $T'{}\!_i$ the factorisation 
of $T_i$ through $\sC_\mathrm{fr}$,  \eqref{e:IsoTannhull} is an isomorphism.
After replacing $\sC$ by $\sC_\mathrm{fr}$, we may thus suppose that $\End_{\sC}(\I)$
is a field $k$ of characteristic $0$ and that $\sC$ is a $k$\nd tensor
category with $\kappa(\sC) = k$.
The homomorphism from $k$ to $H^0(X,\sO_X)$ induced by the $T_i$
then defines a structure of super $k$\nd scheme on $X$ such that the $T_i$ are
$k$\nd tensor functors.

Let $\overline{k}$ be an algebraic closure of $k$.
By \eqref{e:Isopull} with $p$ the projection $X_{\overline{k}} \to X$, we may
after replacing $X$ by $X_{\overline{k}}$ suppose that the structure of
super $k$\nd scheme on $X$ extends to a structure of $\overline{k}$\nd scheme.
Then with 
\begin{equation*}
T'{}\!_i:\overline{k} \otimes_k \sC \to \Mod(X)
\end{equation*} 
the $\overline{k}$\nd tensor functor through which $T_i$ factors,  \eqref{e:IsoTannhull}
is an isomorphism.
Since $\kappa(\sC) = k$, it follows from Lemma~\ref{l:ext} that 
$\overline{k} \otimes_k \sC$ is integral with 
$\kappa(\overline{k} \otimes_k \sC) = \overline{k}$ and that the $T'{}\!_i$
are pointwise faithful.
Replacing $k$ by $\overline{k}$ and $\sC$ by $\overline{k} \otimes_k \sC$, 
we may thus further suppose that $k$ is algebraically closed.

Let $U:\sC \to \sC'$ be a super Tannakian hull of $\sC$.
Then by Theorem~\ref{t:TannhullModX}, \eqref{e:IsoTannhull} is an isomorphism with 
$T_i$ replaced by a tensor isomorphic functor $T'{}\!_iU$.
Since $\kappa(\sC') = k$ by  Corollary~\ref{c:Tannhulleq}, we may after
replacing $\sC$ by $\sC'$ suppose that $\sC$ is super Tannakian.
Thus by Corollary~\ref{c:Tannequivab} we may suppose finally that
\begin{equation*}
\sC = \Mod_{G,\varepsilon}(k)
\end{equation*}
for an affine super $k$\nd group with involution $(G,\varepsilon)$.

Write $T_0$ for the forgetful functor 
from $\Mod_{G,\varepsilon}(k)$ to $\Mod(k)$ followed by pullback along
the structural morphism of $X$.
Consider first the case where $T_1 = T_0$.
We may extend $T_2$ to a cocontinuous $k$\nd tensor functor from 
$\MOD_{G,\varepsilon}(k)$ to $\MOD(X)$, because it is pointwise faithful and 
hence as a functor to $\MOD(X)$ right exact.
With $(G',\varepsilon')$ as in \eqref{e:Gprimedef}, $\MOD_{G',\varepsilon'}(k)$
is the $k$\nd tensor category of $\Z/2$\nd graded objects of $\MOD_{G,\varepsilon}(k)$.
Thus $T_2$ extends further to a cocontinuous $k$\nd tensor functor
\begin{equation*}
T:\MOD_{G',\varepsilon'}(k) \to \MOD(X)
\end{equation*}
which sends $k^{0|1}$ to $\sO_X{}\!^{0|1}$.
It follows from \eqref{e:Vomegaint} that $\sR$ in \eqref{e:Rint} is isomorphic
to $T(k[G])$, because the $k$\nd tensor functor $V \mapsto T(\Omega(V))$ from 
$\Mod_{G,\varepsilon}(k)$ to $\Mod(X)$ is tensor isomorphic to $T_0$.
Since the unit $k \to k[G]$ is non-zero and the functor from $\Mod_{G',\varepsilon'}(k)$ 
to $\Mod(X)$ induced by $T$ is pointwise faithful, writing $k[G]$ as the filtered colimit 
of its subobjects in $\Mod_{G',\varepsilon'}(k)$ shows that the $\sO_X$\nd module $\sR$ 
is flat with non-zero fibres, as required.

The case where $T_1$ is arbitrary reduces to that where $T_1 = T_0$ by \eqref{e:Isopull} 
with $p$ the faithfully flat structural morphism of $\underline{\Iso}^\otimes(T_0,T_1)$.
\end{proof}

Let $k$ be an algebraically closed field of characteristic $0$ 
and $G$ be an affine super $k$\nd group.
It can be seen as follows that if $Y$ is a super $G$\nd scheme such that $Y_{k'}$
is isomorphic to $G_{k'}$ acting on itself by left translation for some extension $k'$ of $k$,
then $Y$ has a $k$\nd point, and hence is isomorphic to $G$ acting on itself by left translation. 
When $G$ is of finite type this is clear.
Replacing $G$ and $Y$ by $G_{\mathrm{red}}$ and $Y_{\mathrm{red}}$, 
we may suppose that $G$ is an affine $k$\nd group.
Write $G$ as the filtered limit of its affine $k$\nd quotients $G_\lambda = G/H_\lambda$ of finite type.
If $Y = \Spec(R)$ and $Y_\lambda = \Spec(R^{H_\lambda})$, then $(Y_\lambda)_{k'}$ is
$(G_\lambda)_{k'}$\nd isomorphic to $(G_\lambda)_{k'}$.
Thus $Y_\lambda$ is $G_\lambda$\nd isomorphic to $G_\lambda$, so that
$G_\lambda(k)$ acts simply transitively on  $Y_\lambda(k)$.
Since $Y = \lim_\lambda Y_\lambda$, it follows from \cite[Lemma~1.1.1]{O10} that $Y$ has a $k$\nd point.

\begin{cor}\label{c:fibfununique}
Let $\sC$ be  an essentially small rigid tensor category and  $k$ be an algebraically closed
field of characteristic $0$.
Then any two faithful tensor functors from $\sC$ to $\Mod(k)$ 
which induce the same homomorphism from $\kappa(\sC)$ to $k$ are tensor isomorphic.
\end{cor}

\begin{proof}
Let $T_1$ and $T_2$ be faithful tensor functors from $\sC$ to $\Mod(k)$
which induce the same homomorphism from $\kappa(\sC)$ to $k$.
Write $G$ for the affine super $k$\nd group $\underline{\Iso}^\otimes(T_2,T_2)$, and
$Y$ for the super $G$\nd scheme $\underline{\Iso}^\otimes(T_1,T_2)$.
By Theorem~\ref{t:Isofthfl}, $Y$ is non-empty, and hence has a $k'$\nd point $y$ 
for some extension $k'$ of $k$.
Then $y$ defines a tensor isomorphism from $p^*T_1$ to $p^*T_2$ with $p^*$ extension
of scalars from $k$ to $k'$, and hence by \eqref{e:Isopull} a $G_{k'}$\nd isomorphism from
$G_{k'}$ to $Y_{k'}$.
By the above, $Y$ has thus a $k$\nd point.
\end{proof}

Let $\sC$ be  an essentially small rigid tensor category, 
$X$ be a non-empty super $\Q$\nd scheme,
and $T:\sC \to \Mod(X)$ be a pointwise faithful tensor functor.
Then if $X$ is regarded as a super $\kappa(\sC)$\nd scheme using the homomorphism
$\kappa(\sC) \to H^0(X,\sO_X)$ defined by $T$, 
we have by Theorem~\ref{t:Isofthfl} a transitive affine groupoid
\begin{equation*}
\underline{\Iso}^\otimes(T) = \underline{\Iso}^\otimes(\pr_1{}\!^*T,\pr_2{}\!^*T)
\end{equation*}
over $X/\kappa(\sC)$, with the evident identities and composition, 
and $C \mapsto \iota_{T(C)}$ defines an involution $\varepsilon_T$ of 
$\underline{\Iso}^\otimes(T)$.
We have a canonical factorisation
\begin{equation*}
\sC \xrightarrow{\Phi_T}  \Mod_{\underline{\Iso}^\otimes(T),\varepsilon_T}(X) \to \Mod(X)
\end{equation*}
of $T$, where $\Phi_T$ is defined by the universal element for $\underline{\Iso}^\otimes(T)$
and the second arrow is the forgetful functor.
Given $U:\sC\to \sC'$ and $T'$ with $T = T'U$, composition with $U$ defines as in
\eqref{e:IsoTannhull} a morphism $\underline{\Iso}^\otimes(T') \to \underline{\Iso}^\otimes(T)$, and the square
\begin{equation}\label{e:ModIsosquare}
\begin{gathered}
\xymatrix{
\Mod_{\underline{\Iso}^\otimes(T)}(X) \ar[r] & \Mod_{\underline{\Iso}^\otimes(T')}(X) \\
\sC \ar^{\Phi_T}[u] \ar^{U}[r] & \sC' \ar_{\Phi_{T'}}[u]
}
\end{gathered}
\end{equation}
commutes.
For $p:X' \to X$ with $X'$ non-empty, we have
by \eqref{e:Isopull} a canonical isomorphism from $\underline{\Iso}^\otimes(p^*T)$ to the pullback of
$\underline{\Iso}^\otimes(T)$ along $p$.
Thus $\Phi_{p^*T}$ factors as 
\begin{eqnarray}\label{e:ModIsofac}
\sC \xrightarrow{\Phi_T} \Mod_{\underline{\Iso}^\otimes(T)}(X) \to \Mod_{\underline{\Iso}^\otimes(p^*T)}(X')
\end{eqnarray} 
with the second arrow an equivalence.

Let $S$ be a super $\Q$\nd scheme and $(K,\varepsilon)$ be
a super groupoid with involution over $S$.
Write
\begin{equation*}
\omega_{K,\varepsilon}:\Mod_{K,\varepsilon}(S) \to \Mod(S)
\end{equation*} 
for the forgetful tensor functor.
The action of $K$ defines a morphism
\begin{equation*}
\theta_{K,\varepsilon}:K \to \underline{\Iso}^\otimes(\omega_{K,\varepsilon}) 
\end{equation*}
of groupoids over $S$, compatible with the involutions.
Then $\Phi_{\omega_{K,\varepsilon}}$ followed by pullback along $\theta_{K,\varepsilon}$ is the
identity of $\Mod_{K,\varepsilon}(S)$.
Thus $\Phi_{\omega_{K,\varepsilon}}$ is an equivalence when $\theta_{K,\varepsilon}$
is an isomorphism.

Let $k$ be a commutative ring, $k'$ be a commutative $k$\nd algebra, and $\sA$ be a cocomplete abelian 
$k$\nd tensor category with cocontinuous tensor product.
Then we have a $k$\nd tensor functor
\begin{equation}\label{e:tenscatext}
k' \otimes_k -:\sA \to \MOD_{\sA}(k' \otimes_k \I).
\end{equation}
If $\sA'$ is a cocomplete abelian $k'$\nd tensor category with cocontinuous tensor product,
the canonical $k$\nd homomorphism from $k'$ to $\End_{\sA'}(\I)$ defines a morphism of algebras
from $k' \otimes_k \I$ to $\I$ in $\sA'$.
Composition with \eqref{e:tenscatext} then defines an equivalence from the category of 
cocontinuous $k'$\nd tensor functors $\MOD_{\sA}(k' \otimes_k \I) \to \sA'$ to the category of 
cocontinuous $k$\nd tensor functors $\sA \to \sA'$, with a quasi-inverse sending $H$ to 
$H(-) \otimes_{k' \otimes_k \I} \I$.

\begin{lem}\label{l:thetaPhi}
Let $S$ be a non-empty super scheme over a field $k$ of characteristic $0$ and $(K,\varepsilon)$ be a 
transitive affine super groupoid over $S/k$.
Then $\theta_{K,\varepsilon}$ is an isomorphism and $\Phi_{\omega_{K,\varepsilon}}$ is an equivalence.
\end{lem}

\begin{proof}
It is enough to prove that $\theta_{K,\varepsilon}$ is an isomorphism.
Suppose first that $S = \Spec(k)$.
Then $K$ is an affine $k$\nd super group $G$, and $\theta_{G,\varepsilon}$ is the morphism
associated to the isomorphism of super $k$\nd algebras underlying \eqref{e:Vomegaint}.

To prove the general case, let $k'$ be an extension of $k$ for which $S$ has a 
$k'$\nd point $s$ over $k$.
Then the fibre of $(K,\varepsilon)$ above the $k'$\nd point $(s,s)$ of $S$
is an affine super $k'$\nd group $(K_{s,s},\varepsilon_s)$.
If we take $X = S$ in \eqref{e:MODAMODK}, then \eqref{e:tenscatext} is extension of scalars from
$\MOD_{K,\varepsilon}(S)$ to $\MOD_{K_{k'},\varepsilon_{k'}}(S_{k'})$. 
Composing with the equivalence from $\MOD_{K_{k'},\varepsilon_{k'}}(S_{k'})$ to 
$\MOD_{K_{s,s},\varepsilon_s}(k')$ defined by taking the fibre at the $k'$\nd point
of $s$ of $S_{k'}$ over $k'$ shows that for $\sA'$ as above, 
composition with restriction
\begin{equation*}
\MOD_{K,\varepsilon}(S) \to \MOD_{K_{s,s},\varepsilon_s}(k')
\end{equation*}
from $K$ to $K_{s,s}$ defines an equivalence from $k'$\nd tensor functors from
$\MOD_{K_{s,s},\varepsilon_s}(k')$ to $\sA'$ to $k$\nd tensor functors from
$\MOD_{K,\varepsilon}(S)$ to $\sA'$.
Restricting to categories of representations then shows that the same holds with
$\MOD$ replaced by $\Mod$.
By taking $\sA'$ of the form $\MOD(X')$ for super $k'$\nd schemes $X'$, it follows that
the morphism
\begin{equation}\label{e:IsoomegaKfib}
\underline{\Iso}^\otimes(\omega_{K_{s,s},\varepsilon_s}) \to 
\underline{\Iso}^\otimes(\omega_{K,\varepsilon}) 
\end{equation}
defined by the embedding $K_{s,s} \to K$ is an isomorphism
onto $\underline{\Iso}^\otimes(\omega_{K,\varepsilon})_{s,s}$.

We have a commutative square
\begin{equation*}
\begin{gathered}
\xymatrix{
K_{s,s} \ar[d] \ar^-{\theta_{K_{s,s}},\varepsilon_s}[r] & 
\underline{\Iso}^\otimes(\omega_{K_{s,s},\varepsilon_s}) \ar[d] \\
K \ar^-{\theta_{K,\varepsilon}}[r] & \underline{\Iso}^\otimes(\omega_{K,\varepsilon})
}
\end{gathered}
\end{equation*}
where the left arrow is the embedding and the right arrow is \eqref{e:IsoomegaKfib}.
By the case where $S = \Spec(k)$ with $k'$ for $k$, the top arrow is an isomorphism.
Thus the bottom arrow is an isomorphism
because it is an isomorphism on fibres over $(s,s)$ and $K$ and  
$\underline{\Iso}^\otimes(\omega_{K,\varepsilon})$ are transitive affine over $S/k$.
\end{proof}

\begin{thm}
Let $\sC$ be a pseudo-Tannakian category, $X$ be a non-empty super $\Q$\nd scheme,
and $T:\sC \to \Mod(X)$ be a pointwise faithful tensor functor.
Then the canonical tensor functor $\sC \to \Mod_{\underline{\Iso}^\otimes(T),\varepsilon_T}(X)$
is a super Tannakian hull of $\sC$.
\end{thm}

\begin{proof}
Let $U:\sC \to \sC'$ is a super Tannakian hull of $\sC$.
Replacing $T$ by a tensor isomorphic functor, we may assume by Theorem~\ref{t:TannhullModX}
that $T$ factors as $T = T'U$.
Then by Theorem~\ref{t:TannhullModX}, $U$ induces an isomorphism from $\underline{\Iso}^\otimes(T')$ to 
$\underline{\Iso}^\otimes(T)$, so that the top arrow of \eqref{e:ModIsosquare} is an isomorphism.
Replacing $T$ by $T'$, it is thus enough to show that when $\sC$ is super Tannakian
$\Phi_T$ is an equivalence.

By Corollary~\ref{c:Tannequivab}, we may assume that $\sC = \Mod_{K,\varepsilon}(k')$
for an algebraically closed extension $k'$ of a field $k$ of characteristic $0$
and transitive affine groupoid with involution $(K,\varepsilon)$ over $k'/k$.
We may suppose after replacing $k'$ by an extension that $X$ has a $k'$\nd point $x$ over $k$.
Then by \eqref{e:ModIsofac} with $p:X' \to X$ the inclusion of $x$ we may after replacing $X$
by $X'$ assume that $X = \Spec(k')$.
By Corollary~\ref{c:fibfununique} may further assume that $T$ is the forgetful functor 
$\omega_{K,\varepsilon}$.
That $\Phi_T$ is an equivalence then follows from Lemma~\ref{l:thetaPhi} with $S = \Spec(k')$.
\end{proof}

\end{document}